\documentclass[11pt,reqno]{amsart}
\usepackage[a4paper,margin=1.15in]{geometry}
\usepackage{amsmath,amssymb,amsthm,mathtools,tikz,dsfont}
\usepackage{mathrsfs}
\usepackage{changepage}
\usepackage{caption}
\usepackage{microtype}
\usepackage{enumitem}
\usepackage[colorlinks=true,linkcolor=blue,citecolor=magenta,urlcolor=blue]{hyperref}
\usepackage[nameinlink,noabbrev]{cleveref}

\numberwithin{equation}{section}
\newtheorem{theorem}{Theorem}[section]
\newtheorem{proposition}[theorem]{Proposition}
\newtheorem{conjecture}{Conjecture}
\newtheorem{lemma}[theorem]{Lemma}

\theoremstyle{definition}

\newtheorem{remark}[theorem]{Remark}

\title{Sums of von Mangoldt convolutions via trace-sum averages over $U(N)$}

\author{Ayesha Irfan}
\address{School of Mathematics, University of Bristol, Bristol, United Kingdom}
\email{ayesha.irfan@bristol.ac.uk}
\date{\today}

\begin{document}


\begin{abstract}
We study piecewise polynomial functions $\delta_k(c)$ that appear in the conjectured
asymptotics for the variance of short-interval sums of von Mangoldt convolutions $\Lambda^{*k}$. By a theorem of Kuperberg and Lalín, a unitary matrix integral involving a sum of trace products governs the large-$q$ limit of the function-field analogue of this variance. We identify $\delta_k(c)$ as the leading-order coefficient of this integral when the total trace degree and matrix size grow with a fixed ratio $c$. Through an asymptotic analysis of two lattice-point representations for the matrix integral, we obtain explicit forms of $\delta_k(c)$. A finite spline expansion and a contour-integral representation expressing $\delta_k(c)$ as the inverse Laplace transform of a two-weight block Hankel determinant are also obtained.
\end{abstract}
\maketitle
\section{Introduction}\label{sec:intro}

\subsection{Motivation from divisor sums.} The study of arithmetic functions in short intervals has brought to light an intriguing bridge between classical problems in analytic number theory and the spectral behaviour of random matrix ensembles. This correspondence was developed in depth for the $k$-th divisor function $d_k(n)$, which counts the number of ways to write a positive integer $n$ as a product of $k$ positive integers: \begin{equation}
\label{eq:divisor-function}
d_k(n)
:=
\#\bigl\{
(n_1,\ldots,n_k)\in\mathbb{N}^k:
n_1\cdots n_k=n
\bigr\}.
\end{equation}
For real $x,H>0$, let
\begin{equation}
\label{eq:divisor_short_interval_error}
\Delta^{(d)}_k(x;H)
:=
\sum_{x<n\leq x+H}d_k(n)
-
\operatorname*{Res}_{s=1}
\left(
\zeta(s)^k\frac{(x+H)^s-x^s}{s}
\right)
\end{equation}
be the remainder term for sums of $d_k$ over short intervals $(x, x+H]$.

Over the polynomial ring $\mathbb{F}_q[t]$, Keating, Rodgers, Roditty-Gershon and Rudnick \cite{kr3} studied the variance of short-interval sums of the analogous divisor function on monic polynomials. They showed that the leading coefficient in its large-$q$ limit is given by a unitary matrix integral of the form:
\begin{equation}
\label{eq:Ik-secular-integral}
I^{(d)}_k(m;N)
:=
\int_{U(N)}
\Biggl|
\sum_{\substack{j_1+\cdots+j_k=m\\
                0\leq j_1,\ldots,j_k\leq N}}
\operatorname{Sc}_{j_1}(U)\cdots
\operatorname{Sc}_{j_k}(U)
\Biggr|^2
dU,
\end{equation}
defined for $m\in \mathbb{N}_0$ and $N \in \mathbb{N}$. Here $dU$ denotes the normalised Haar probability measure on the unitary group and $\operatorname{Sc}_j(U)$ is the $j$th secular coefficient of
$U$. Let $c:= m/N$; for $c \in [0,k]$, they proved that
\begin{equation}
\label{eq:Ik-gamma-asymptotic}
I^{(d)}_k(m;N)=
\gamma_k(c)N^{k^2-1}
+
O_k\bigl(N^{k^2-2}\bigr),
\qquad
N\rightarrow\infty,
\end{equation}
where $\gamma_k(c)$ is a piecewise polynomial of degree $k^2-1$. 
This led to a conjecture for the corresponding integer problem, in which the fixed ratio $c$ controls the short-interval scale. For $H=X^{1-1/c}$ with $1<c<k$, they conjectured that
\begin{equation}
\label{eq:integer-divisor-variance-conjecture}
\frac{1}{HX}
\int_X^{2X}
\big|\Delta^{(d)}_k(x;H)\bigr|^2\,dx
\sim
a_k\gamma_k(c)
\bigl(\log X^{1/c}\bigr)^{k^2-1},
\qquad
X\rightarrow\infty,
\end{equation}
where $a_k$ is an arithmetic factor.

Furthermore, $\gamma_k(c)$ was shown to exhibit an additional layer of analytic structure; Basor, Ge, and Rubinstein represented it through the inverse Fourier transform of a Hankel determinant, relating it to the Painlevé V equation \cite{gammakc}. Bettin and Conrey \cite{bettinconrey} later proved that the predicted $\gamma_k(c)$ is in fact recovered in the integer setting, i.e. \eqref{eq:integer-divisor-variance-conjecture} holds, if one assumes a suitable shifted moments conjecture for $\zeta(s)$. 

We will study the piecewise polynomial functions $\delta_k(c)$ that emerge from the short-interval variance problem for the $k$-fold von Mangoldt convolution $\Lambda^{*k}$ instead; the arithmetic motivation for this quantity is discussed shortly. The work of Kuperberg and Lal\'{i}n \cite{matilde} provides the passage from the arithmetic variance problem for $\Lambda^{*k}$ to a matrix integral over the unitary group. Through their function-field analysis, the relevant large-$q$ limits are expressed as trace-sum averages over compact matrix groups. This enables us to determine $\delta_k(c)$ through the asymptotic
analysis of two lattice representations for the trace-sum averages;
these descriptions form the basis of our conjecture for the
asymptotic variance of short-interval sums of $\Lambda^{*k}$.

\subsection{Classical theory.} 
For $\Re(s)> 1$, the logarithmic derivative of the Riemann zeta function is given by the Dirichlet series
\begin{equation}
\label{eq:log-derivative-zeta}
      \frac{\zeta'}{\zeta}(s) = -\sum_{n=1}^{\infty} \frac{\Lambda(n)}{n^{s}}.
 \end{equation}
Here, $\Lambda(n)$ is the von Mangoldt function, defined as
 \begin{align}
 \Lambda(n) & := \begin{cases} 
   \log{p}
  &  \exists \, m \in \mathbb{N }, \, p \in \mathbb{P}: n = p^{m} \\
      0
     & \text{otherwise.}
   \end{cases}
\end{align}
Thus $\Lambda$ is supported on prime powers, assigning the logarithmic weight
$\log p$ to each power $p^m$. The sum of $\Lambda(n)$ over $n\leq x$ gives Chebyshev's function
$\psi(x)$, the asymptotic behaviour for which is equivalent to the Prime Number Theorem \cite{kn:dav80}:
\begin{equation}
\psi(x) := \sum_{n \leq x} \Lambda(n) \sim x \quad \text{as} \quad x \rightarrow \infty.
\end{equation} 
The expected main term for the increment $\psi(x+H)-\psi(x)$ is $H$; the short-interval problem asks how large the fluctuations around this main term are. Let $x, H>0$ be real. If we define
\begin{equation}
\label{eq:short-interval-von-mangoldt}
   \psi(x;H):=\psi(x+H)-\psi(x)
   =\sum_{x<n\leq x+H}\Lambda(n),
\end{equation}
then, assuming the Riemann Hypothesis, Goldston--Montgomery \cite{goldstonmont87} proved that
\begin{equation}
\label{GM}
   \frac{1}{X}\int_1^X
   \left|\psi(x;H)-H\right|^2\,dx
   \sim H\bigl(\log X-\log H\bigr), \qquad \text{as } X \to \infty,
\end{equation}
holds uniformly in the range $X^{\varepsilon}\leq H\leq X^{1-\varepsilon}$ for every fixed $0<\varepsilon<\tfrac{1}{2}$  if and only if the strong pair-correlation conjecture holds \cite{kn:mont73}. Keating and Rudnick \cite{keating-rudnick} have since proved the function-field analogue of \eqref{GM}. The variance considered in this paper is a higher-order analogue of the
classical variance of primes in short intervals shown above. For the $k$-fold Dirichlet convolution of the von Mangoldt function, we write
\begin{equation}
   \Lambda^{*k}:=\underbrace{\Lambda * \cdots * \Lambda \mathstrut}_{k \text{ times}} \, ,
\end{equation}
or equivalently, 
\begin{equation}\label{eq:lambda-convolution-representation}
   \Lambda^{*k}(n)
   =
   \sum_{n_1\cdots n_k=n}
   \Lambda(n_1)\cdots\Lambda(n_k).
\end{equation}
As each factor $\Lambda(n_i)$ is non-zero only when $n_i$ is a prime power, the sum in \eqref{eq:lambda-convolution-representation} is a weighted
count of ordered factorisations
\begin{equation}
\label{almost-prime}
   n=q_1\cdots q_k,
   \qquad q_i=p_i^{m_i},\quad p_i\in\mathbb P, \quad m_i\geq 1,
\end{equation}
with weight $(\log p_1)\cdots(\log p_k)$ attached to such a factorisation. We adopt the terminology of Farmer, Gonek, Lee and Lester \cite{kn:fargon} by referring to an integer $n$ that satisfies \eqref{almost-prime} as a $k$-fold almost prime. Using this convention, an almost prime $n$ is an integer that can be represented as a product of $k$ prime powers, not necessarily a product of $k$ primes. Note that \eqref{eq:log-derivative-zeta} gives
\begin{equation}\label{eq:dirichlet-series-lambda-convolution}
   \sum_{n=1}^{\infty}  \frac{\Lambda^{*k}(n)}{n^s}
   =
   \left(-\frac{\zeta'}{\zeta}(s)\right)^k,
   \qquad \Re s>1.
\end{equation}
\subsection{Variance of almost-primes.} Let $X$ be a large real parameter, and suppose $H=H(X)$ satisfies $H=o(X)$; then for any starting point $x\in[X,2X]$, we regard the interval $(x, x+H]$ as short. For real $x,H>0$, define
\begin{equation}\label{eq:psi}
 \Psi_k(x;H):=\sum_{x<n\le x+H}\Lambda^{*k}(n).
\end{equation}
The expected main term is the contribution from the pole of $( -\zeta'/\zeta(s))^k$ at $s=1$. We therefore define the centred error term by
\begin{align}\label{eq:def-Delta}
\Delta_{k}(x;H)
:=
\Psi_k(x;H)
-
\operatorname*{Res}_{s=1}
\left[
\left(-\frac{\zeta'}{\zeta}(s)\right)^k
\frac{(x+H)^s-x^s}{s}
\right].
\end{align}
The normalised short-interval variance is
\begin{equation}
\label{def:variance}
    \mathcal{V}_{k}(X,H)
    :=
    \frac{1}{XH}
    \int_{X}^{2X}
    \left|\Delta_{k}(x;H)\right|^{2}\,dx.
\end{equation}
When $k=1$, this recovers a normalised, dyadic form of the classical variance of primes in short intervals \eqref{GM}. In what follows, we specialise to the scale $H=X^{1-1/\alpha}$ for a real parameter $\alpha > 0$; this choice gives $H=o(X)$ for every $\alpha>0$.

Other statistics involving $\Lambda^{*k}$ have been studied by
Farmer, Gonek, Lee and Lester \cite{kn:fargon}, who investigate
the relationship between correlations of the Riemann zeros, weighted almost-prime statistics, and $t$-aspect mean values of products
of shifted logarithmic derivatives $\zeta'/\zeta$ near the
critical line. Banks and Sinha \cite{kn:banks-sinha25} study twisted summatory functions involving both $\Lambda^{*k}$ and the generalised von Mangoldt function
\begin{equation}\label{eq:selberg-von-mangoldt}
   \Lambda_j(n)
   :=
   \sum_{d\mid n}\mu(d)\log^j(n/d),
\end{equation}
obtaining uniform estimates for such sums which are equivalent to
the Riemann Hypothesis. Note that $\mathcal V_k(X,H)$ is the short-interval variance from the integer setting to be modelled: our conjectural asymptotic and other results in this paper are guided
by the corresponding matrix integral in random matrix theory, rather than by a
direct evaluation of the variance over the integers. By Theorem~\ref{thm:delta-cycle-formula}, we propose the following conjecture. 

Here and throughout, $\delta(\cdot)$ denotes the Dirac delta distribution and $t_+=\max(t,0)$; the symbols $i^{+}$ and $i^{-}$ denote two labelled copies of each index $i\in\{1,\ldots,k\}$. It is worth noting that the conjecture below predicts
$\mathfrak{D}_k(\eta)=(1-\eta)^{2k-1}
\mathcal{D}_k(\frac{1}{1-\eta})$,
with $\mathcal{D}_k(\alpha)$ being the leading-order coefficient
for the variance in \eqref{def:variance}, normalised by
$(\log X^{1/\alpha})^{2k-1}$, as defined in related work
\cite{kn:bcis26}. In that work, the identity
$\mathcal D_k(\alpha)=\delta_k(\alpha)$ is established under
a conjecture for shifted moments of $\zeta'/\zeta$.

\begin{conjecture}
\label{conj-variance}
Fix an integer $k\geq1$ and $0<\eta<1$, and set $H=X^\eta$. Then, as
$X\to\infty$,
\begin{equation*}
\mathcal V_k(X,H)
\sim
\mathfrak D_k(\eta)(\log X)^{2k-1},
\qquad
\mathfrak D_k(\eta)
:=
(1-\eta)^{2k-1}
\delta_k\!\left(\frac{1}{1-\eta}\right),
\end{equation*}
where, for $c>0$,
\begin{equation*}
\delta_k(c) =\sum_{(\pi,\sigma)\in\mathfrak J_k}
\hspace{-0.1in} \operatorname{sgn}(\sigma)
\int_{\mathbb R_{>0}^{2k}}
\hspace{-0.05in} \delta \Bigl(\sum_{i=1}^k u_i-c\Bigr)
\hspace{-0.05in} \prod_{\mathcal C\in\operatorname{Cyc}(\sigma)}
\hspace{-0.1in}\delta \left(D_{\mathcal C,|\mathcal{C}|}\right)
\Bigl(1 - \max_{0\le r< |\mathcal{C}|}D_{\mathcal C,r} + \min_{0\le r< |\mathcal{C}|}D_{\mathcal C,r}\Bigr)_+
d\mathbf u\,d\mathbf v.
\end{equation*}
Here, $\mathfrak J_k$ is the set of pairs $(\pi,\sigma)$ such that
$\pi$ is a set partition of $\{1^+,\ldots,k^+\}\cup\{1^-,\ldots,k^-\}$ and $\sigma$ permutes the
blocks of $\pi$. Define
\begin{equation*}
D_{\mathcal C,0}:=0,
\qquad
D_{\mathcal C,r}
:=
\sum_{s=1}^r
\biggl( \,
\sum_{i^+\in B_s}u_i
-
\sum_{j^-\in B_s}v_j
\biggr),
\qquad 1\leq r\leq |\mathcal{C}|,
\end{equation*}
for a cycle $\mathcal C=(B_1,\ldots,B_\ell) \in \operatorname{Cyc}(\sigma)$.
\end{conjecture}

The function-field literature gives a closer short-interval analogue. As mentioned, Keating
and Rudnick \cite{keating-rudnick} obtained the leading term for the function-field version of \eqref{def:variance} for $k=1$. Rodgers \cite{rodgers-almost-primes} generalised this short-interval framework to higher-order variance and covariance problems for almost-prime counts in $\mathbb F_q[T]$. His main theorem uses
the generalised von Mangoldt function \eqref{eq:selberg-von-mangoldt}, rather than the convolved $\Lambda^{*k}$ considered here. However, in his work on arithmetic consequences of the GUE
conjecture, Rodgers observed that analogous GUE-based arguments may be applied to the
convolution weights $\Lambda^{*k}$; although, unlike the $\Lambda_k$-weighted
case, the resulting formulae do not seem to possess a comparably simple
algebraic form as the convolution order $k$ grows \cite{rodgers-gue}.  

\subsection{Trace Averages over the Unitary Group}
The matrix integral we consider arises as the function-field analogue of the short-interval variance in \eqref{def:variance}. Let $q$ be an odd prime power and let $\mathbb{F}_q[T]$ denote the polynomial ring over the finite field $\mathbb{F}_q$. In this setting, one replaces integers by monic polynomials $f \in \mathbb{F}_q[T]$, and studies sums over $f$ lying in a short interval around a fixed monic polynomial $A$. Let $\mathcal{M}_n$ denote the set of monic polynomials in
$\mathbb{F}_q[T]$ of degree $n$, and write $|f|=q^{\deg f}$. For $A\in\mathcal{M}_n$ and an integer $0 \leq h \leq n-1$, the short interval is described by
\begin{equation}
   |f-A|\leq q^h.   
\end{equation}
Thus $f$ must have the same coefficients as $A$ in degrees $h+1,h+2,\ldots,n$, while the $h+1$ lower-degree coefficients may be chosen freely, giving $q^{h+1}$ possible monic polynomials. For a monic polynomial $f \in \mathbb{F}_q[T]$, we define
\begin{equation}
    \Lambda^{*k}_q(f) := \sum_{\substack{f_1 \ldots f_k = f \\ f_i \, \text{monic}}} \Lambda_q(f_1) \cdots \Lambda_q(f_k),
\end{equation}
where $\Lambda_q(f)$ is the von Mangoldt function
 \begin{align}
 \Lambda_q(f) & = \begin{cases} 
   \deg(P)
  &  f = P^m \text{ for some monic irreducible } P \text{ and } m \geq 1, \\
      0
     & \text{otherwise.}
   \end{cases}
\end{align}
Kuperberg and Lal\'{i}n prove that, in the short-interval setting over $\mathbb{F}_q[t]$, the large-$q$ limit of the variance analogous to \eqref{def:variance} is given by a Haar integral over the unitary group involving products of traces. This result is stated below.
\begin{theorem}[Kuperberg--Lal\'{i}n \cite{matilde}, Theorem~1.3]
\label{thm:kuperberg-lalin-trace-average}
Let $A\in \mathbb{F}_q[T]$ be a monic polynomial of degree $n$, let
$0\leq h \leq n-2$, and define
\begin{equation*}
   \mathcal{N}^{U}_{0,\Lambda^{*k}_q}(A;h)
   :=
   \sum_{\substack{f\ \mathrm{monic}\\ |f-A|\leq q^h\\ f(0)\neq 0}}
  \Lambda^{*k}_q(f).
\end{equation*}
Assume $k\leq n$. Then, as $q\to\infty$,
\begin{equation*}
   \left\langle \mathcal{N}^{U}_{0,\Lambda^{*k}_q}\right\rangle
   \sim
   q^{h+1}\binom{n-1}{k-1},
\end{equation*}
and
\begin{equation*}
   \frac{1}{q^n}
   \sum_{A\in\mathcal{M}_n}
   \left|
      \mathcal{N}^{U}_{0,\Lambda^{*k}_q}(A;h)
      -\left\langle \mathcal{N}^{U}_{0,\Lambda^{*k}_q}\right\rangle
   \right|^2
   \sim
   q^{h+1}
   \int_{U(n-h-2)}
   \left|
      \sum_{\substack{j_1+\cdots+j_k=n\\ 1\leq j_1,\ldots,j_k}}
      \operatorname{Tr}(U^{j_1})\cdots \operatorname{Tr}(U^{j_k})
   \right|^2
   dU.
\end{equation*}
\end{theorem}
Here 
$\langle \mathcal{N}^{U}_{0,\Lambda^{*k}}\rangle$ denotes the $\frac {1}{q^n}$-normalised average of
$\mathcal{N}^{U}_{0,\Lambda^{*k}}(A;h)$ over $A\in\mathcal{M}_n$. Thus the
function-field variance is governed, in the large-$q$ limit, by the trace-sum average over $U(N)$ appearing on the right-hand side. This trace average depends on two parameters with distinct arithmetic origins:
the total trace degree $n$, coming from the degree of the polynomials being
counted, and the matrix size $N=n-h-2$, determined by the length of the
short interval. For fixed integer $k\geq 1$, we regard the integral in Kuperberg and Lal\'{i}n's
formula as a two-parameter random matrix average and define
\begin{equation}
\label{eq:variance-rmt}
    I_{k}(n;N) := \int_{U(N)} \biggl| \sum_{\substack{j_{1}, \ldots, j_{k} \geq 1 \\ j_{1} + \cdots + j_{k} = n}} \operatorname{Tr}(U^{j_{1}}) \cdots \operatorname{Tr}(U^{j_{k}})  \biggr|^{2}dU,
\end{equation}
for $n,N\in\mathbb{N}$. For $0 \leq h \leq n-3$, the specialisation $N=n-h-2$ recovers the matrix integral in Theorem~\ref{thm:kuperberg-lalin-trace-average} after normalising by the short-interval size $q^{h+1}$:
\begin{equation*}
\frac{1}{q^{n+h+1}}
\sum_{A\in\mathcal M_n}
\left|
\mathcal N^U_{0,\Lambda_q^{*k}}(A;h)
-\left\langle\mathcal N^U_{0,\Lambda_q^{*k}}\right\rangle
\right|^2
\sim I_k(n,n-h-2),
\end{equation*}
which is the exact function-field analogue of $\mathcal V_k(X,H)$ in \eqref{def:variance}.

The same scaling principle is used by Keating, Rodgers, Roditty-Gershon and Rudnick in the divisor-function setting \cite{kr3}; their function-field variance formula leads to unitary averages of secular coefficients similarly with degree $n$ and matrix size $N=n-h-2$, and the limiting piecewise polynomial function $\gamma_k(c)$ is obtained by letting $n,N\to\infty$ with $c = \tfrac{n}{N}$ fixed. Likewise, we define the leading asymptotic of $I_k(n;N)$ as the function
\begin{equation}
\label{eq:rmt-delta-final}
\delta_{k}(c) := \lim_{N \rightarrow \infty} \frac{I_{k}(\lfloor c N \rfloor;N)}{N^{2k-1}}, \qquad k \in \mathbb{N}, c > 0,
\end{equation}
whenever the limit exists. Under the correspondence $X\leftrightarrow q^n$ and $H\leftrightarrow q^{h+1}$, the logarithmic interval exponent is
\begin{equation*}
\eta=\frac{\log H}{\log X}=\frac{h+1}{n}.
\end{equation*}
As $N=n-h-2$, fixing $c= \tfrac n N$ gives $\eta\to 1- \tfrac 1 c$, while $N^{2k-1}\sim n^{2k-1}(1-\eta)^{2k-1}$. Thus the asymptotic $I_k(n;N)\sim N^{2k-1}\delta_k(c)$ naturally leads to the coefficient $\mathfrak D_k(\eta)$ appearing in Conjecture~\ref{conj-variance}. 

\section{Main results}
\label{sec:main-results}

We begin with two exact lattice representations of $I_k(n;N)$; their asymptotic analysis in Section~\ref{sec:lattice-asymptotics} then yields two explicit integral representations of $\delta_k(c)$. In Section~\ref{section:finite-spline-structure}, one of these representations is evaluated to obtain a finite-spline expansion, with coefficients given by finite sums of products of Cauchy determinants. Section~\ref{section:block-Hankel-determinant} then relates $\delta_k(c)$ to a block Hankel determinant.

\subsection{Lattice representations for \texorpdfstring{$\delta_k(c)$}{delta_k(c)}}

Following the strategy of Keating et al. \cite{kr3}, we apply a character expansion to rewrite $I_k(n;N)$ as a weighted lattice count, assigning each point the square of a character sum as its weight. Although this is not an ordinary lattice-point count, in which each admissible lattice point contributes $1$, the form lends itself well to asymptotic analysis. Passing to the corresponding limiting volumes gives our first explicit representation of $\delta_k(c)$. For an integer $m\geq1$ and $(x,y)\in\mathbb R^2$, define
\begin{equation}
\label{eq:Theta-legendre-def}
\Theta_m(x,y)
:=
\frac{(x+y)^{m-1}}{m!(m-1)!}
P_{m-1}\!\left(\frac{x-y}{x+y}\right),
\end{equation}
where $P_j$ denotes the Legendre polynomial of degree $j$. Then, for integers $k,r\geq1$ and vectors $\mathbf x,\mathbf y\in\mathbb R^r$, define
\begin{equation}
\label{eq:Phi-def}
\Phi_{k,r}(\mathbf x,\mathbf y)
:=
k!\sum_{\substack{m_1+\cdots+m_r=k\\m_i\geq1}}
\det_{1\leq i,j\leq r}
\bigl[\Theta_{m_i}(x_i,y_j)\bigr].
\end{equation}
The apparent singularity at $x+y=0$ in \eqref{eq:Theta-legendre-def} is removable: after expanding the Legendre polynomial, the factor $(x+y)^{m-1}$ cancels all denominators.

\begin{theorem}\label{thm:leading-volume}
Let $k\geq1$ be an integer, and let
$\mathcal K\subset(0,\infty)$ be compact. For $c\in\mathcal K$, put
$n=\lfloor cN\rfloor$. Then, uniformly for $c\in\mathcal K$,
\begin{equation}\label{eq:Ik-asymptotic}
I_k(n;N)
=
\delta_k(c)N^{2k-1}
+
O_{k,\mathcal K}(N^{2k-2}),
\end{equation}
where
\begin{equation}\label{eq:delta-def}
\delta_k(c)
=
\sum_{r=1}^{\lfloor\sqrt{k}\rfloor}
\frac{1}{(r!)^2}
\int_{\substack{x_i\geq0\\0\leq y_i\leq1}}
\Phi_{k,r}(\mathbf x,\mathbf y)^2
\,
\delta\left(
\sum_{i=1}^r(x_i+y_i)-c
\right)
\prod_{i=1}^r d x_i \, d y_i.
\end{equation}
Here, $\Phi_{k,r}$ is as defined in \eqref{eq:Phi-def}.
\end{theorem}

In particular, Theorem~\ref{thm:leading-volume} proves the existence of the scaling limit defining $\delta_k(c)$ and shows that only the ranks $1\leq r\leq\lfloor\sqrt{k}\rfloor$ contribute to its leading order. To obtain a more elementary description of $I_k(n;N)$, i.e. one closer to a classical lattice count, we will use a determinantal formula for unitary trace averages stated in Theorem~\ref{thm:AoT}. The advantage of this form is that the only weight is the permutation sign from a determinant, which yields a signed lattice count. Proposition~\ref{prop:signed-lattice-count} expresses the trace average as
\begin{equation*}
I_k(n;N)
=
\#A_{k,n,N}^{+}
-
\#A_{k,n,N}^{-},
\end{equation*}
where $A_{k,n,N}^{+}$ and $A_{k,n,N}^{-}$ are explicitly defined
collections of admissible lattice points associated with even and odd
permutations, respectively. 

To describe the corresponding scaling limit, we introduce
\begin{equation*}
\mathbf u=(u_1,\ldots,u_k),
\qquad
\mathbf v=(v_1,\ldots,v_k),
\end{equation*}
with $u_i,v_j>0$. We distinguish the two families using the signed index set
\begin{equation*}
\Omega_k
:=
\{1^+,\ldots,k^+\}
\cup
\{1^-,\ldots,k^-\},
\end{equation*}
where $i^+$ indexes $u_i$ and $j^-$ indexes $v_j$.

In the finite signed lattice representation, each signed index is assigned
to a lattice position. Indices assigned to the same position form a block,
and the collection of these blocks gives a set partition $\pi$ of
$\Omega_k$. For a block $B\in\pi$, write
\begin{equation*}
B^+
:=
\{i:i^+\in B\},
\qquad
B^-
:=
\{j:j^-\in B\},
\end{equation*}
and define its net displacement by
\begin{equation}
\label{eq:block-displacement}
d_B(\mathbf u,\mathbf v)
:=
\sum_{i\in B^+}u_i
-
\sum_{j\in B^-}v_j.
\end{equation}

The permutation associated with the signed lattice configuration permutes
these positions and hence induces a permutation $\sigma$ of the blocks of
$\pi$. Let $\Pi(\Omega_k)$ denote the set of all set partitions of
$\Omega_k$, and let $\mathfrak S(\pi)$ denote the permutation group of
the blocks of $\pi$. We define
\begin{equation}
\label{eq:block-permutation-set}
\mathfrak J_k
:=
\left\{
(\pi,\sigma):
\, \pi\in\Pi(\Omega_k),\quad \sigma\in\mathfrak S(\pi)
\right\}.
\end{equation}
For $(\pi,\sigma)\in\mathfrak J_k$, let
$\operatorname{Cyc}(\sigma)$ denote the set of cycles in the disjoint-cycle
decomposition of $\sigma$. For
$\mathcal C\in\operatorname{Cyc}(\sigma)$ and $|\mathcal{C}| = \ell$, write
\begin{equation*}
\mathcal C=(B_1,\ldots,B_\ell),
\qquad
\sigma(B_j)=B_{j+1}\quad(1\leq j<\ell),
\qquad
\sigma(B_\ell)=B_1,
\end{equation*}
and define the partial displacements
\begin{equation}
\label{eq:D_CJ_def}
D_{\mathcal C,0}:=0,
\qquad
D_{\mathcal C,j}
:=
\sum_{a=1}^{j}d_{B_a}(\mathbf u,\mathbf v),
\qquad
1\leq j\leq\ell.
\end{equation}
The cycle width is defined by
\begin{equation}
\label{eq:cycle-width}
\omega_{\mathcal C}(\mathbf u,\mathbf v)
:=
\max_{0\leq j<\ell}D_{\mathcal C,j}
-
\min_{0\leq j<\ell}D_{\mathcal C,j},
\end{equation}
and the associated kernel is
\begin{equation}
\label{eq:cycle-kernel}
\mathcal W_{\mathcal C}(\mathbf u,\mathbf v)
:=
\delta\!\left(D_{\mathcal C,\ell}\right)
\bigl(1-\omega_{\mathcal C}(\mathbf u,\mathbf v)\bigr)_+.
\end{equation}
Each cycle $\mathcal C$ imposes one balance condition,
$D_{\mathcal C,|\mathcal C|}=0$.  On its support, changing the initial block cyclically reorders the partial displacements $D_{\mathcal{C},0}, \ldots, D_{\mathcal{C}, \ell-1}$ and shifts them all by the same constant, so $\omega_{\mathcal C}$, and hence
$\mathcal W_{\mathcal C}$, depends only on the directed cycle
$\mathcal C$. The cycle-balance condition also ensures that \eqref{eq:cycle-kernel} eliminates cycles containing only $+$ or only $-$ indices. 

\begin{theorem}
\label{thm:delta-cycle-formula}
Let $k\geq1$ be an integer. For every $c>0$,
\begin{equation}
\label{eq:delta-cycle-final}
\delta_k(c)
=
\sum_{(\pi,\sigma)\in\mathfrak J_k}
\operatorname{sgn}(\sigma)
\int_{\mathbb R_{>0}^{2k}}
\delta\!\left(\sum_{i=1}^{k}u_i-c\right)
\prod_{\mathcal C\in\operatorname{Cyc}(\sigma)}
\mathcal W_{\mathcal C}(\mathbf u,\mathbf v)
\,d\mathbf u\,d\mathbf v,
\end{equation}
with the notation as in
\eqref{eq:block-permutation-set}--\eqref{eq:cycle-kernel}. For every compact $\mathcal K\subset(0,\infty)$, \eqref{eq:Ik-asymptotic} holds uniformly for
$c\in\mathcal K$ when
$n=\lfloor cN\rfloor$, with $\delta_k(c)$ given by
\eqref{eq:delta-cycle-final} and an error
$O_{k,\mathcal K}(N^{2k-2})$.
\end{theorem}

This is the cycle-reduced form of the signed lattice asymptotic. It retains the permutation sign responsible for the cancellation in the exact count, while organising the remaining contribution cycle by cycle through the balance and displacement conditions.

\subsection{Finite-spline and Lauricella structure}

The preceding formulas give geometric descriptions of $\delta_k(c)$. For a closed expression, it is more convenient to return to Theorem~\ref{thm:leading-volume} and integrate the polynomial $\Phi_{k,r}(\mathbf x,\mathbf y)^2$ term by term. For $\alpha,\beta\in\mathbb Z_{\geq0}^r$, write
$|\alpha|=\sum_{i=1}^r \alpha_i$ and
$\mathbf x^\alpha= \prod_{i=1}^{r}x_i^{\alpha_i}$. Define $p_{\alpha,\beta}^{(k,r)}$ by
\begin{equation}
\label{eq:p-definition}
\Phi_{k,r}(\mathbf x,\mathbf y)^2
=
\sum_{\alpha,\beta\in\mathbb Z_{\geq0}^r}
p_{\alpha,\beta}^{(k,r)}
\mathbf x^\alpha\mathbf y^\beta.
\end{equation}
An explicit form for the coefficients $p_{\alpha,\beta}^{(k,r)}$, as a finite sum involving products of Cauchy determinants, is stated in Proposition~\ref{prop:p-explicit}.

The contribution indexed by $r$ in \eqref{eq:delta-def} involves the bounded variables $\mathbf y\in[0,1]^r$, and its form can change only at the integer levels $1,\ldots,r$ which we refer to as candidate breakpoints. Thus every positive breakpoint of $\delta_k$ belongs to
\begin{equation*}
\{1,\ldots,\lfloor\sqrt{k}\rfloor\}.
\end{equation*}
We call the open intervals determined by the actual breakpoints the chambers of $\delta_k$, and refer to $c\geq\lfloor\sqrt{k}\rfloor$ as its final polynomial range. The piecewise-polynomial structure of $\delta_k(c)$ follows from a finite-spline expansion valid for every $c>0$ and, on the final polynomial range, a formula in terms of a classical multivariable hypergeometric function.

That function is one of Lauricella's four generalisations of the Gauss hypergeometric series ${}_2F_1$ to several variables, introduced by Lauricella in 1893~\cite{kn:lau93} and put on a systematic footing by Appell and Kamp\'e de F\'eriet~\cite{kn:ak26}. We will only need the first of the four Lauricella functions, $F_A^{(r)}$. For $a\in\mathbb C$, vectors $\mathbf b,\mathbf c,\mathbf z\in\mathbb C^r$ with $c_i\notin\{0,-1,-2,\ldots\}$ and $|z_1|+\cdots+|z_r|<1$, define
\begin{equation}\label{eq:FA-definition}
 F_A^{(r)}
 \left(a;\mathbf b;\mathbf c;\mathbf z\right)
 =
 \sum_{m_1,\ldots,m_r\geq0}
 \frac{(a)_{|m|}}{m_1!\cdots m_r!}
 \prod_{i=1}^r
 \frac{(b_i)_{m_i}}{(c_i)_{m_i}}z_i^{m_i},
 \qquad
 |m|=m_1+\cdots+m_r.
\end{equation}
Here $(a)_j$ denotes the rising factorial; if $a$ is a non-positive integer, the series terminates and defines a polynomial in $\mathbf z$.

\begin{theorem}
\label{thm:lauricella}
Let $k\geq1$ be an integer. For $1\leq r\leq\lfloor\sqrt{k}\rfloor$ and $\alpha,\beta\in\mathbb Z_{\geq0}^r$, set $d_{\alpha,r}
 =|\alpha|+r-1$ and let $\mathbf1_r=(1,\ldots,1)$ with length $r$. For $c>0$,
\begin{align*}
 \delta_k(c)
 &=
 \sum_{r=1}^{\lfloor\sqrt{k}\rfloor}\frac{1}{(r!)^2}
 \sum_{\substack{\alpha,\beta\in\mathbb Z_{\geq0}^r\\
                  |\alpha|+|\beta|=2(k-r)}}
 \alpha!\beta!\,p_{\alpha,\beta}^{(k,r)} \\& \hspace{1.8in} \times
 \sum_{S\subseteq[r]}(-1)^{|S|}
 \sum_{\substack{0\leq q_i\leq\beta_i \\ i\in S}}
 \frac{(c-|S|)_+^{\,2k-1-\sum_{i\in S}q_i}}
 {\left(2k-1-\sum_{i\in S}q_i\right)!
  \displaystyle\prod_{i\in S}q_i!}.
\end{align*}
On the range
$c\geq\lfloor\sqrt{k}\rfloor$, this becomes
\begin{align}
 \label{eq:delta-Lauricella}
 \delta_k(c)
 &=
 \sum_{r=1}^{\lfloor\sqrt{k}\rfloor}\frac{1}{(r!)^2}
 \sum_{\substack{\alpha,\beta\in\mathbb Z_{\geq0}^r\\
                  |\alpha|+|\beta|=2(k-r)}}
 \frac{\alpha!\,p_{\alpha,\beta}^{(k,r)}
       c^{d_{\alpha,r}}}
 {d_{\alpha,r}!\displaystyle\prod_{i=1}^r(\beta_i+1)}
 F_A^{(r)}
 \left(
  -d_{\alpha,r};
  \beta+\mathbf1_r;
  \beta+2\mathbf1_r;
  c^{-1}\mathbf1_r
 \right).
\end{align}
Here, $F_A^{(r)}$ is the Lauricella function defined in \eqref{eq:FA-definition} and $p_{\alpha,\beta}^{(k,r)}$ in \eqref{eq:p-definition}. In the above, $\alpha!=\prod_i\alpha_i!$, $\beta!=\prod_i\beta_i!$ and empty products are interpreted as $1$.
\end{theorem}
As the Lauricella series in \eqref{eq:delta-Lauricella} terminates,
$\delta_k$ is a polynomial of degree at most $2k-2$ on the final range.
We will see that, at a candidate breakpoint $j$, Theorem~\ref{thm:lauricella} shows $\delta_k$ is at least $C^{2j-2}$; cancellations may increase the smoothness or remove
the breakpoint altogether.

\subsection{Cumulative trace averages and a block Hankel form}
\label{subsec:block-hankel}

From its definition, the function $\delta_k(c)$ describes the leading coefficient of $I_k(n;N)$ with one fixed total trace degree $n$. It is therefore useful to consider the cumulative trace average obtained by allowing the total trace degree to be at most $n$:
\begin{equation}
\label{eq:cumulative-Ik}
\tilde I_k(n;N)
:=
\int_{U(N)}
\biggl|
\sum_{\substack{j_1,\ldots,j_k\geq1\\j_1+\cdots+j_k\leq n}}
\operatorname{Tr}(U^{j_1})\cdots\operatorname{Tr}(U^{j_k})
\biggr|^2dU.
\end{equation}
By Haar invariance under $U\mapsto e^{i\varphi}U$, any integral pairing
the degree-$m$ and degree-$\ell$ pieces vanishes unless $m=\ell$. Consequently,
\begin{equation}
\label{eq:cumulative-exact-degree-decomposition}
\tilde I_k(n;N)
=
\sum_{m=k}^{n}I_k(m;N).
\end{equation}
We encode the leading large-$N$ behaviour of this cumulative average by
\begin{equation}
\label{eq:Mk-definition}
M_k(\alpha)
:=
\lim_{N\to\infty}
\frac{\tilde I_k(\lfloor \alpha N\rfloor;N)}{N^{2k}},
\qquad \alpha>0.
\end{equation}
The compact-uniform asymptotic in Theorem~\ref{thm:leading-volume} applies on intervals bounded away from $c=0$. The behaviour near $c=0$ is controlled by the stable-range evaluation \eqref{eq:Ik-stable-range}, proved in Subsection~\ref{subsec:stable-range}. Together, these results give
\begin{equation}
\label{eq:Mk-delta-relation}
M_k(\alpha)=\int_0^\alpha \delta_k(c)\,dc,
\qquad
M_k'(\alpha)=\delta_k(\alpha).
\end{equation}

Basor, Ge and Rubinstein expressed the divisor-function coefficient $\gamma_k(c)$ as the inverse Fourier transform of a Hankel determinant \cite{gammakc}. Motivated by their result, we use a cumulative contour representation of $M_k(\alpha)$, as stated in Theorem~\ref{thm:two-centre-P-input}, to express $\delta_k(c)$ as an inverse Laplace transform of a two-weight multiple Hankel determinant.

Let $C_0$ and $C_1$ be positively oriented circles of radius $\tau<\tfrac12$ centred at $0$ and $1$, respectively. For an integer $n\geq0$, set
\begin{align}
X_n(\theta;s,u,v)
&:=
\frac{1}{2\pi i}
\left(\int_{C_0}+e^{i\theta}\int_{C_1}\right)
 z^{n-k}
\exp\!\left(-sz+\frac{u}{z}+\frac{v}{z-1}\right)dz,
\label{eq:two-centre-X}
\\
W_n(\theta;u,v)
&:=
\frac{1}{2\pi i}
\left(\int_{C_1}+e^{-i\theta}\int_{C_0}\right)
 z^n(z-1)^{-k}
\exp\!\left(\frac{u}{z}+\frac{v}{z-1}\right)dz.
\label{eq:two-centre-W}
\end{align}
Form the $2k\times k$ moment matrices
\begin{align}
\mathsf X_k(\theta;s,u,v)
&:=
\bigl[X_{m+q}(\theta;s,u,v)\bigr]_{\substack{0\leq m\leq2k-1\\0\leq q\leq k-1}},
\notag\\
\mathsf W_k(\theta;u,v)
&:=
\bigl[W_{m+q}(\theta;u,v)\bigr]_{\substack{0\leq m\leq2k-1\\0\leq q\leq k-1}},
\label{eq:two-centre-moment-blocks}
\end{align}
and define the $2k\times2k$ determinant
\begin{equation}
\label{eq:two-centre-D}
D_k(\theta;s,u,v)
:=
\det\bigl[\,\mathsf X_k(\theta;s,u,v)\mid\mathsf W_k(\theta;u,v)\,\bigr].
\end{equation}
Here, $[A\mid B]$ denotes the matrix obtained by adjoining the columns of $B$ to those of $A$. Within each family of columns, the entries depend on the row and column indices only through their sum. Thus $D_k$ is a two-weight multiple Hankel determinant, equivalently a Hankel determinant with $2\times2$ matrix-valued moments.

We use the inverse Laplace notation
\begin{equation}
\label{eq:two-centre-Laplace-definition}
\mathcal L^{-1}_{s\to c}[F(s)]
:=
\frac{1}{2\pi i}
\int_{\sigma-i\infty}^{\sigma+i\infty}e^{cs}F(s)\,ds,
\qquad \sigma>0.
\end{equation}
For the expressions below, the finite local contour residues are evaluated before the integral in \eqref{eq:two-centre-Laplace-definition}.

\begin{theorem}
\label{thm:two-centre}
For an integer $k\geq1$ and $c>0$,
\begin{equation}
\label{eq:two-centre-final-formula}
 \delta_k(c)
 =
 (-1)^k\mathcal L^{-1}_{s\to c}
 \left[
 \left.
 s^{-k}\left(1+\frac{\partial_u}{s}\right)^k
 \partial_v^kD_k(0;s,u,v)
 \right|_{u=v=0}
 \right],
\end{equation}
where $D_k$ is defined using \eqref{eq:two-centre-X}-\eqref{eq:two-centre-D}, and
$\mathcal L^{-1}_{s\to c}$ as defined by
\eqref{eq:two-centre-Laplace-definition}.
\end{theorem}

\section{Lattice asymptotics for $I_k(n;N)$}
\label{sec:lattice-asymptotics}

This section develops the two lattice models used to analyse the matrix integral $I_k(n;N)$ when $n$ and $N$ grow at a fixed ratio $c$ with $c>0$. We begin with the partition and symmetric-function notation used
in the weighted lattice count method.

\subsection{Partitions, Young tableaux and Frobenius coordinates}
\label{subsec:partition-preliminaries}
We now introduce the partition and symmetric-function notation needed for this section. Readers familiar with partitions, Young tableaux and
Schur functions may proceed directly to Subsection~\ref{subsec:stable-range}.

A \emph{partition} is a finite weakly decreasing sequence of positive
integers
\begin{equation*}
 \lambda=(\lambda_1,\ldots,\lambda_{\ell}),
 \qquad
 \lambda_1\geq\cdots\geq\lambda_{\ell}>0.
\end{equation*}
Its \emph{size} and \emph{length} are
\begin{equation*}
 |\lambda|:=\sum_{i=1}^{\ell}\lambda_i,
 \qquad
 \ell(\lambda):=\ell,
\end{equation*}
and we write $\lambda\vdash n$ when $|\lambda|=n$. If the part $j$
occurs $m_j(\lambda)$ times, then
\begin{equation*}
 \ell(\lambda)=\sum_{j\geq1}m_j(\lambda),
 \qquad
 |\lambda|=\sum_{j\geq1}j\,m_j(\lambda).
\end{equation*}
An \emph{ordered composition} of $n$ into $k$ parts is an ordered
$k$-tuple of positive integers whose sum is $n$; sorting its entries gives
a partition of $n$ with $k$ parts.

A \emph{Young diagram} of $\lambda$ is an array of boxes
organised into rows, with $\lambda_i$ boxes in its $i$th row such that the leftmost boxes of all rows are vertically aligned. A \emph{Young tableau} is a
filling of these boxes with entries and, in particular, we use \emph{semistandard Young tableaux}: their
entries are positive integers, weakly increasing along rows and strictly
increasing down columns. For a semistandard Young tableau $T$ of shape $\lambda$ with entries in $\{1,\ldots,M\}$, let $m_j(T)$ denote the number of entries of $T$ equal to $j$. Then the Schur polynomial in $M$ variables may be defined by
\begin{equation}
\label{eq:schur-tableau-definition}
 s_\lambda(x_1,\ldots,x_M)
 :=
 \sum_T \prod_{j=1}^M x_j^{m_j(T)},
\end{equation}
where the sum is over all such tableaux $T$. Thus the partition specifies the shape, while
the tableau supplies the entries used in the monomial weight.
\begin{figure}[htbp]
\begin{adjustwidth}{1mm}{1mm}
\centering

\begin{minipage}[t]{0.44\linewidth}
\centering
\begin{tikzpicture}[
  x=0.50cm,
  y=0.50cm,
  baseline=(current bounding box.north)
]
  \foreach \row/\len in {1/6,2/4,3/3,4/1}{
    \foreach \col in {1,...,\len}{
      \draw
        ({\col-1},{1-\row}) rectangle ({\col},{-\row});
    }
  }
\end{tikzpicture}

\vspace{5mm}
{\small Young diagram of $\lambda=(6,4,3,1)$.}
\end{minipage}
\hspace{-0.02\linewidth}
\begin{minipage}[t]{0.44\linewidth}
\centering
\begin{tikzpicture}[
  x=0.50cm,
  y=0.50cm,
  baseline=(current bounding box.north)
]
  \foreach \x/\y/\entry in {
    0/0/1,1/0/1,2/0/2,3/0/3,4/0/3,5/0/7,
    0/1/2,1/1/3,2/1/3,3/1/4,
    0/2/4,1/2/4,2/2/6,
    0/3/7}{
      \draw (\x,-\y) rectangle ++(1,-1);
      \node at ({\x+0.5},{-\y-0.5}) {\entry};
  }
\end{tikzpicture}

\vspace{5mm}
{\small A semistandard tableau of the same shape.}
\end{minipage}

\captionsetup{width=0.94\linewidth,margin=0pt}
\caption{A Young diagram (left) and a semistandard Young tableau of the same shape (right). The tableau entries are weakly increasing along rows and strictly increasing down columns.}
\label{fig:young-diagram-tableau}

\end{adjustwidth}
\end{figure}

\noindent The \emph{conjugate partition} $\lambda'$ is obtained by transposing the
Young diagram; equivalently,
\begin{equation*}
 \lambda'_j=\#\{i:\lambda_i\geq j\}.
\end{equation*}
The number of boxes on the main diagonal is
\begin{equation*}
 r=r(\lambda):=\max\{i:\lambda_i\geq i\}.
\end{equation*}
It is called the \emph{Frobenius rank}.
For $1\leq i\leq r$, define
\begin{equation}
\label{eq:Frobenius}
 a_i:=\lambda_i-i,
 \qquad
 b_i:=\lambda'_i-i.
\end{equation}
Thus $a_i$ is the number of boxes to the right of the $i$th diagonal box and
$b_i$ is the number below it. In the Frobenius notation, we may express a partition $\lambda$ as 
\begin{equation*}
 \lambda=(a_1,\ldots,a_r\mid b_1,\ldots,b_r).
\end{equation*}
The coordinate sequences satisfy
\begin{equation*}
 a_1>\cdots>a_r\geq0,
 \qquad
 b_1>\cdots>b_r\geq0,
\end{equation*}
and conversely any two such sequences of the same length determine a unique
partition. The size and length become
\begin{equation}
\label{eq:frobenius-size-length}
 |\lambda|=r+\sum_{i=1}^r(a_i+b_i),
 \qquad
 \ell(\lambda)=b_1+1.
\end{equation}
A simple way to understand the first identity is through hooks. 

A \emph{hook} is a partition of the form
\begin{equation*}
(a+1,1^b),
\qquad a,b\geq0.
\end{equation*}
The Young diagram for a hook consists of a single diagonal box together with $a$ boxes to its right and $b$ boxes below it. Thus a hook has Frobenius rank one, and in Frobenius notation it is simply written
\begin{equation}
\label{eq:frob-hook}
(a\mid b).
\end{equation}
The first identity in \eqref{eq:frobenius-size-length} decomposes the diagram into its $r$ diagonal hooks; each $i$th diagonal hook contributes $1+a_i+b_i$ boxes. The
second records the number of boxes in the first column. 

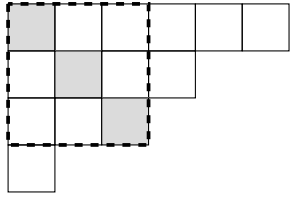
\vspace{3mm}
\begin{figure}[htbp]
\begin{adjustwidth}{-13mm}{-13mm}
\centering

\begin{minipage}[c]{0.44\linewidth}
\centering
\begin{tikzpicture}[x=0.62cm,y=0.62cm]
  \foreach \row/\len in {1/6,2/4,3/3,4/1}{
    \foreach \col in {1,...,\len}{
      \ifnum\col=\row
        \fill[gray!28]
          ({\col-1},{1-\row}) rectangle ({\col},{-\row});
      \fi
      \draw
        ({\col-1},{1-\row}) rectangle ({\col},{-\row});
    }
  }
  \draw[line width=1.4pt, dashed] (0,0) rectangle (3,-3);
\end{tikzpicture}
\end{minipage}
\hspace{-0.08\linewidth}
\begin{minipage}[c]{0.44\linewidth}
\centering
\renewcommand{\arraystretch}{1.22}
\begin{tabular}{c|ccc}
  diagonal box $i$ & $1$ & $2$ & $3$ \\ \hline
  arm length $a_i$ & $5$ & $2$ & $0$ \\
  leg length $b_i$ & $3$ & $1$ & $0$
\end{tabular}

\medskip
\[
  \lambda=(5,2,0\mid 3,1,0),
  \qquad
  r=3.
\]
\end{minipage}

\caption{Frobenius data for $\lambda=(6,4,3,1)$. The grey boxes on the
main diagonal are read from the top left to the bottom right, and the dashed
$3\times3$ square highlights the Frobenius rank. The corresponding arm and
leg lengths are listed separately.}
\label{fig:frobenius-coordinates}
\end{adjustwidth}
\end{figure}
\vspace{-3mm}
\subsection{Symmetric functions}
\label{subseq:symmetric-functions} For a positive integer $m$, the $m$th
power sum symmetric function is
\begin{equation*}
 p_m(z_1,\ldots,z_N):=z_1^m+\cdots+z_N^m.
\end{equation*}
For a partition $\mu=(\mu_1,\ldots,\mu_s)$, write
$p_\mu:=p_{\mu_1}\cdots p_{\mu_s}$. If $z_1,\ldots,z_N$ are the
eigenvalues of a matrix $U\in U(N)$, then
\begin{equation}
\label{eq:power-sum-trace}
 p_m(U):=p_m(z_1,\ldots,z_N)=\operatorname{Tr}(U^m).
\end{equation}

To relate the power sums to Schur functions, we note that partitions describe the cycle structure of permutations in the symmetric group $S_n$. Every permutation in $S_n$ decomposes into disjoint
cycles, whose lengths form a partition $\mu\vdash n$, called its \emph{cycle
type}. Characters
take the same value on permutations with the same cycle type, so it is enough
to index these values by $\mu$. For $\lambda,\mu\vdash n$, let
$\chi^\lambda_\mu$ denote the value of the irreducible character indexed by
$\lambda$ on permutations of cycle type $\mu$.

The Frobenius character formula then gives
\begin{equation}
\label{eq:frobenius-character-formula}
 p_\mu=\sum_{\lambda\vdash n}\chi^\lambda_\mu s_\lambda,
 \qquad \mu\vdash n.
\end{equation}
Thus the irreducible character values of $S_n$ are encoded by the change of
basis from power sums to Schur functions: these functions admit their own character
interpretation for the unitary group. When $\ell(\lambda)\leq N$, the function $s_\lambda(U)$ is an irreducible
character of $U(N)$; when $\ell(\lambda)>N$, it vanishes in $N$ variables.
Haar orthogonality therefore gives
\begin{equation}
\label{eq:schur-orthogonality}
 \int_{U(N)}s_\lambda(U)\overline{s_\mu(U)}\,dU
 =
 \begin{cases}
  1,&\lambda=\mu\text{ and }\ell(\lambda)\leq N,\\
  0,&\text{otherwise.}
 \end{cases}
\end{equation}

Symmetric functions have been evaluated at finite sets of variables thus far,
in particular at the eigenvalues of a unitary matrix. It is useful to allow a
broader notion of evaluation. A \emph{specialisation} is a rule $\rho$ that
assigns a scalar $f(\rho)$ to every symmetric function $f$, while preserving
addition and multiplication. Equivalently, it is an algebra homomorphism
\begin{equation*}
 \rho:\Lambda\longrightarrow\mathbb{C},
\end{equation*}
where $\Lambda$ denotes the algebra of symmetric functions. Since the power
sums $p_1,p_2,\ldots$ generate $\Lambda$ over $\mathbb{C}$, a specialisation
is determined by the values
\begin{equation*}
 p_m(\rho),\qquad m\geq1.
\end{equation*}
These values then determine the value of every symmetric function.

\subsection{The stable range}
\label{subsec:stable-range}

The partition notation gives a particularly simple evaluation of
$I_k(n;N)$ in the stable range, by which we mean $k\leq n\leq N$.  For a
partition $\lambda$, set
\begin{equation*}
 z_\lambda:=\prod_{j\geq1}j^{m_j(\lambda)}m_j(\lambda)!.
\end{equation*}
The Diaconis-Shahshahani trace moment formula \cite{kn:diasha} states that, for
$\lambda,\mu\vdash n$ with $n\leq N$,
\begin{equation}
\label{eq:DS-stable}
 \int_{U(N)}p_\lambda(U)\overline{p_\mu(U)}\,dU
 =
 \begin{cases}
  z_\lambda,&\lambda=\mu,\\
  0,&\lambda\neq\mu.
 \end{cases}
\end{equation}
Grouping the ordered
compositions $j_1+\cdots+j_k=n$ by their associated partition gives
\begin{equation*}
 \sum_{\substack{j_1+\cdots+j_k=n\\j_i\geq1}}
 p_{j_1}(U)\cdots p_{j_k}(U)
 =
 \sum_{\substack{\lambda\vdash n\\\ell(\lambda)=k}}
 \frac{k!}{\prod_{j\geq1}m_j(\lambda)!}\,p_\lambda(U).
\end{equation*}
Applying \eqref{eq:DS-stable} and using the number of orderings of each
partition, we obtain
\begin{align*}
 I_k(n;N)
 &=
 (k!)^2
 \sum_{\substack{\lambda\vdash n\\\ell(\lambda)=k}}
 \prod_{j\geq1}\frac{j^{m_j(\lambda)}}{m_j(\lambda)!}=
 k!
 \sum_{\substack{j_1+\cdots+j_k=n\\j_i\geq1}}
 j_1\cdots j_k.
\end{align*}
Writing $j_i=r_i+1$ and iterating the standard binomial convolution \cite[(5.26), p.~169]{kn:graham-knuth94} gives
\begin{equation*}
 \sum_{\substack{j_1+\cdots+j_k=n\\j_i\geq1}}j_1\cdots j_k
 =
 \sum_{\substack{r_1+\cdots+r_k=n-k\\r_i\geq0}}
 \prod_{i=1}^k\binom{r_i+1}{1}
 =
 \binom{n+k-1}{2k-1}.
\end{equation*}
Consequently,
\begin{equation}
\label{eq:Ik-stable-range}
 I_k(n;N)=k!\binom{n+k-1}{2k-1},
 \qquad k\leq n\leq N,
\end{equation}
and hence
\begin{equation}
\label{eq:delta-stable-range}
 \delta_k(c)=\frac{k!}{(2k-1)!}c^{2k-1},
 \qquad 0<c\leq1.
\end{equation}
The boundary at $c=1$ becomes transparent in the Schur basis. Combining
\eqref{eq:frobenius-character-formula} and \eqref{eq:schur-orthogonality} gives
\begin{equation*}
\int_{U(N)}p_\lambda(U)\overline{p_\mu(U)}\,dU
 =
 \sum_{\substack{\nu\vdash n\\\ell(\nu)\leq N}}
 \chi^\nu_\lambda\chi^\nu_\mu.
\end{equation*}
When $n\leq N$, every partition $\nu\vdash n$ satisfies
$\ell(\nu)\leq n\leq N$, so no terms are removed from this sum. It is therefore
the full character-orthogonality relation
\begin{equation*}
 \sum_{\nu\vdash n}\chi^\nu_\lambda\chi^\nu_\mu
 =\begin{cases}
 z_\lambda,&\lambda=\mu,\\
 0,&\lambda\neq\mu,
 \end{cases}
\end{equation*}
which gives \eqref{eq:DS-stable}. Once $n>N$, partitions with more than $N$
parts occur, but their Schur functions vanish in $N$ variables. Removing these
terms leaves only a partial character sum, so the cancellations responsible
for orthogonality are incomplete: for $\lambda\neq\mu$, the pairing of
$p_\lambda$ and $p_\mu$ need no longer vanish. As a result, the partition sum
for $I_k(n;N)$ does not reduce to the stable-range binomial formula. General exact Haar-integration methods remain available for $n > N$; Weingarten calculus \cite{kn:weingarten78, kn:collinssniady06} is one such approach which expands unitary matrix moments as finite sums over permutations, but it does not collapse here to a comparably simple form. In the scaling $n=\lfloor cN\rfloor$ with $c>0$, direct asymptotic analysis of either formula is difficult because its indexing set grows with $N$: the character sum ranges over partitions of $n$, while the Weingarten expansion involves permutations in $S_n$. In both cases, one must control cancellations across the full sum. This is the regime addressed by the lattice
asymptotics below.

\subsection{Lattice point estimates}
\label{subsec:lattice-point-estimates}

Both lattice models below are analysed by comparing integer points with the
volume of a continuous region. A \emph{convex polytope} in $\mathbb R^D$
is a bounded region cut out by finitely many linear inequalities. Each
inequality confines the polytope to one side of a hyperplane. At an interior
point, the relevant inequalities are strict; when one becomes an equality,
the point is restricted to the corresponding hyperplane and one degree of
freedom is lost. A boundary face of dimension $D-1$ is called a
\emph{facet}. We write $\operatorname{vol}_D$ for $D$-dimensional
Euclidean volume.

Davenport's theorem compares the number of lattice points in a convex
region with its volume. In the applications below, the regions have diameter
of order $N$, so the error is of order $N^{D-1}$, one power of $N$ smaller
than the main volume.

\begin{theorem}[Davenport, \cite{kn:davenport51}]
\label{thm:davenport}
Let $S\subset\mathbb R^D$ be a convex region contained in a closed ball of
radius $\rho$. Then
\begin{equation}
\label{eq:davenport}
 \#(S\cap\mathbb Z^D)
 =
 \operatorname{vol}_D(S)
 +O_D(\rho^{D-1}+1).
\end{equation}
\end{theorem}

The signed lattice model for $I_k(n;N)$, stated in Proposition~\ref{prop:signed-lattice-count}, expresses $I_k(n;N)$
as a signed combination of ordinary lattice counts; Davenport's theorem applies directly to this case. The Frobenius-coordinate model, recorded in Lemma~\ref{lem:weighted-lattice-count}, is weighted: a lattice point $\mathbf m$ contributes a smooth function evaluated at $\mathbf m/N$ in the large-$N$ limit. Therefore, for the proof of Theorem~\ref{thm:leading-volume}, we require a weighted analogue of Davenport's theorem adapted to the family of Frobenius-coordinate polytopes arising there; the lemma below is the weighted
form that will be used. For a continuously differentiable function
$\mathcal P$, the norm $\|\mathcal P\|_{C^1}$ denotes a common bound for
$|\mathcal P|$ and for all its first partial derivatives.

\begin{lemma}
\label{lem:weighted-lattice-volume}
Let $D\geq1$ be an integer, and let $\{V_t\}_{t\in T}$ be a family of
$D$-dimensional convex polytopes in $\mathbb R^D$, contained in a common
compact set and having a uniformly bounded number of facets. Let $U$ be a
fixed, open neighbourhood of this compact set, and let
$\mathcal P_t\in C^1(U)$ satisfy
\begin{equation*}
 \sup_{t\in T}\|\mathcal P_t\|_{C^1(U)}<\infty.
\end{equation*}
Then, as $N\to\infty$, uniformly for $t\in T$,
\begin{equation}
\label{eq:weighted-lattice-volume}
 \sum_{\mathbf m\in NV_t\cap\mathbb Z^D}
 \mathcal P_t(\mathbf m/N)
 =
 N^D\int_{V_t}\mathcal P_t(\mathbf z)\,d\mathbf z
 +O(N^{D-1}).
\end{equation}
The same estimate remains valid if the lattice points lying on any collection of facets of $NV_t$ are omitted from the sum.
\end{lemma}
\begin{proof}
To relate the discrete sum to an integral, we associate a unit cube to
each lattice point. Let $\mathcal C$ denote the common compact set containing
all the polytopes $V_t$. As $\mathcal{C}$ is bounded, we may choose a fixed bounded box
\begin{equation}
\label{eq:B-definition}
 B=\prod_{j=1}^D[a_j,b_j]
\end{equation}
containing every $V_t$. For each $\mathbf m\in\mathbb Z^D$, define the
half-open unit cube
\begin{equation}
\label{eq:Q-unit-cube-definition}
 Q_{\mathbf m}:=\mathbf m+[0,1)^D.
\end{equation}
The cubes $Q_{\mathbf m}$, as $\mathbf m$ ranges over $\mathbb Z^D$, tile
$\mathbb R^D$, and each has volume $1$. The half-open convention makes the cubes pairwise disjoint and ensures that every point of $\mathbb R^D$ belongs to exactly one of them. We now replace each lattice point $\mathbf m\in NV_t\cap\mathbb Z^D$ by its
corresponding unit cube and set
\begin{equation}
\label{eq:union-of-cubes-W}
 W_{N,t}:=
 \bigcup_{\mathbf m\in NV_t\cap\mathbb Z^D}Q_{\mathbf m}.
\end{equation}
Thus $W_{N,t}$ is a cubical approximation to $NV_t$.
This construction is useful because the summand
$\mathcal P_t(\mathbf m/N)$ in \eqref{eq:weighted-lattice-volume} can be compared with the integral of
$\mathcal P_t(\mathbf u/N)$ over the unit cube $Q_{\mathbf m}$. Summing
these cube integrals will then produce an integral over $W_{N,t}$.

It will also be important to have a uniform bound for the number of cubes
involved. Since $V_t\subset B$, we have $ NV_t\subset NB$ for every $t$. In each coordinate direction, the box $NB$ has length
$O(N)$ and hence contains $O(N)$ possible integer coordinates. It follows that
\begin{equation*}
 \#(NV_t\cap\mathbb Z^D)=O(N^D),
\end{equation*}
so $W_{N,t}$ is the union of $O(N^D)$ unit cubes. All the implied constants
depend only on the fixed box $B$ and on $D$, and are therefore uniform in $t$. Put $M:=\sup_{t\in T}\|\mathcal P_t\|_{C^1(U)}$.
For $\mathbf m\in NV_t\cap\mathbb Z^D$ and
$\mathbf u\in Q_{\mathbf m}$, we have
$\|\mathbf u-\mathbf m\|\leq\sqrt D$.
Since $\mathcal C$ is compact and $U$ is an open neighbourhood
of $\mathcal C$, the segment joining $\mathbf m/N$ to
$\mathbf u/N$ lies in $U$ for all sufficiently large $N$,
uniformly in $t,\mathbf m,\mathbf u$.
The mean value theorem therefore gives
\begin{equation*}
 \left|
 \mathcal P_t(\mathbf u/N)-\mathcal P_t(\mathbf m/N)
 \right|
 \leq \frac{DM}{N}.
\end{equation*}
As $Q_{\mathbf m}$ has volume $1$, it follows that
\begin{equation*}
 \mathcal P_t(\mathbf m/N)
 =
 \int_{Q_{\mathbf m}}\mathcal P_t(\mathbf u/N)\,d\mathbf u
 +O(N^{-1}).
\end{equation*}
There are $O(N^D)$ such lattice points; summing the preceding identity over
$\mathbf m\in NV_t\cap\mathbb Z^D$ gives a total error
\begin{equation*}
 O(N^D)\,O(N^{-1})=O(N^{D-1}).
\end{equation*}
The half-open cubes $Q_{\mathbf m}$ are pairwise disjoint and
their union is $W_{N,t}$, so
\begin{equation*}
 \sum_{\mathbf m\in NV_t\cap\mathbb Z^D}
 \int_{Q_{\mathbf m}}\mathcal P_t(\mathbf u/N)\,d\mathbf u
 =
 \int_{W_{N,t}}\mathcal P_t(\mathbf u/N)\,d\mathbf u.
\end{equation*}
And this gives
\begin{equation}
\label{eq:weighted-cube-comparison}
 \sum_{\mathbf m\in NV_t\cap\mathbb Z^D}
 \mathcal P_t(\mathbf m/N)
 =
 \int_{W_{N,t}}\mathcal P_t(\mathbf u/N)\,d\mathbf u
 +O(N^{D-1}).
\end{equation}

It remains to compare $W_{N,t}$ with $NV_t$; the two regions agree except near the boundary. Denoted by
$\partial(NV_t)$, the boundary of $NV_t$ is the union of the facets of $NV_t$. Recall that a facet is a flat face of dimension $D-1$, and is therefore one dimension lower than the $D$-dimensional polytope itself. For two sets $A$ and $B$, their \emph{symmetric difference} is
\begin{equation}
\label{eq:symmetric-difference-def}
 A\mathbin{\triangle}B
 :=
 (A\setminus B)\cup(B\setminus A);
\end{equation}
it is the set of points belonging to exactly one of $A$ and $B$.

Let $\mathbf x\in W_{N,t}\mathbin{\triangle}NV_t$, and let
$Q_{\mathbf m}$ be the unique half-open unit cube containing $\mathbf x$.
By the definition of $W_{N,t}$ in \eqref{eq:union-of-cubes-W}, the cube
$Q_{\mathbf m}$ was selected if and only if $\mathbf m\in NV_t$. Because the cubes are pairwise disjoint and $\mathbf x\in Q_{\mathbf m}$, it
follows that
\begin{equation*}
 \mathbf x\in W_{N,t}
 \quad\Longleftrightarrow\quad
 \mathbf m\in NV_t.
\end{equation*}
Therefore, as $\mathbf x$ belongs to the symmetric difference, exactly
one of $\mathbf x$ and $\mathbf m$ belongs to $NV_t$. Since the line
segment joining them lies in $Q_{\mathbf m}$, it must pass through the
boundary of $NV_t$ and we have
\begin{equation*}
 Q_{\mathbf m}\cap\partial(NV_t)\neq\varnothing.
\end{equation*}
We now choose $\mathbf y$ in this intersection; the points $\mathbf x$ and
$\mathbf y$ lie in the same unit cube, whose diameter is $\sqrt D$, so $\|\mathbf x-\mathbf y\|\leq\sqrt D$.
This means that every point where $W_{N,t}$ and $NV_t$ differ lies in a
fixed-width neighbourhood of $\partial(NV_t)$. Equivalently,
\begin{equation}
\label{eq:symmetric-difference-boundary-neighbourhood}
 W_{N,t}\mathbin{\triangle}NV_t
 \subset
 \left\{
  \mathbf x\in\mathbb R^D:
  \operatorname{dist}\bigl(\mathbf x,\partial(NV_t)\bigr)
  \leq\sqrt D
 \right\}.
\end{equation}

As every $V_t$ lies in the fixed bounded box $B$, each facet of $V_t$
has diameter $O(1)$, uniformly in $t$. The dilation from $V_t$ to $NV_t$ multiplies every distance by $N$, so each
facet of $NV_t$ has diameter $O(N)$.

Let $\mathcal F$ be a facet of $NV_t$, and let $H$ be the hyperplane
containing it. Each line perpendicular to $H$ is determined by its unique
point of intersection $\mathbf p$ with $H$. Such a line can meet the
$\sqrt D$-neighbourhood of $\mathcal F$ only if $\mathbf p$ lies within
distance $\sqrt D$ of $\mathcal F$. Fixing any point of $\mathcal F$ as
a centre, all these points $\mathbf p$ lie in a ball of radius $O(N)$ in
$H$. This ball has  $(D-1)$-dimensional volume $O(N^{D-1})$ within $H$, while the part of the
neighbourhood on each perpendicular line has length at most $2\sqrt D$.
Tonelli's theorem
\cite[Theorem~2]{vanDoorn2021} calculates the volume of the
neighbourhood by adding these lengths over the corresponding points
$\mathbf p\in H$, giving
\begin{equation*}
 \operatorname{vol}_D
 \left(
  \left\{
   \mathbf x\in\mathbb R^D:
   \operatorname{dist}(\mathbf x,\mathcal F)\leq\sqrt D
  \right\}
 \right)
 \leq
 2\sqrt D\,O(N^{D-1})
 =
 O(N^{D-1}).
\end{equation*}

By hypothesis, $V_t$, and therefore $NV_t$, has at most $J$
facets, where $J$ is independent of $t$. This also means that the boundary $\partial(NV_t)$ is the union of at most $J$ facets. Every
point lying within distance $\sqrt D$ of the boundary lies within
distance $\sqrt D$ of at least one of these facets, so the
$\sqrt D$-neighbourhood of the boundary is the union of at most $J$
facet neighbourhoods. Each of these
has volume $O(N^{D-1})$ by the estimate above, which gives an $O(N^{D-1})$ volume for the $\sqrt D$-neighbourhood of $\partial(NV_t)$ as it is the volume of their
union. Equation \eqref{eq:symmetric-difference-boundary-neighbourhood} states
that $W_{N,t}\mathbin{\triangle}NV_t$ is contained in this neighbourhood. Overall, we obtain
\begin{align}
\label{eq:symmetric-difference-volume}
 \operatorname{vol}_D
 \bigl(W_{N,t}\mathbin{\triangle}NV_t\bigr)
 &\leq
 \operatorname{vol}_D
 \left(
  \left\{
   \mathbf x\in\mathbb R^D:
   \operatorname{dist}\bigl(\mathbf x,\partial(NV_t)\bigr)
   \leq\sqrt D
  \right\}
 \right) =
 O(N^{D-1}).
\end{align}

For all sufficiently large $N$, we have $\mathbf u/N\in U$ throughout
$W_{N,t}\cup NV_t$: this follows from the preceding segment argument on
$W_{N,t}$, while on $NV_t$ it is immediate from
$\mathbf u/N\in V_t\subset U$. Hence the uniform $C^1$ bound gives
$|\mathcal P_t(\mathbf u/N)|\leq M$ on both regions. Their common part
contributes equally to the two integrals, while the remaining points
belong to their symmetric difference. Therefore,
\begin{align*}
 \left|
  \int_{W_{N,t}}\mathcal P_t(\mathbf u/N)\,d\mathbf u
  -\int_{NV_t}\mathcal P_t(\mathbf u/N)\,d\mathbf u
 \right| &\leq
 \int_{W_{N,t}\mathbin{\triangle}NV_t}
 \left|\mathcal P_t(\mathbf u/N)\right|\,d\mathbf u \\
 &\leq
 M\operatorname{vol}_D
 \bigl(W_{N,t}\mathbin{\triangle}NV_t\bigr)
 =
 O(N^{D-1}),
\end{align*}
where the final equality follows from
\eqref{eq:symmetric-difference-volume}. Thus
\begin{equation}
\label{eq:weighted-boundary-comparison}
 \int_{W_{N,t}}\mathcal P_t(\mathbf u/N)\,d\mathbf u
 =
 \int_{NV_t}\mathcal P_t(\mathbf u/N)\,d\mathbf u
 +O(N^{D-1}).
\end{equation}
Combining this with \eqref{eq:weighted-cube-comparison} gives
\begin{equation*}
 \sum_{\mathbf m\in NV_t\cap\mathbb Z^D}
 \mathcal P_t(\mathbf m/N)
 =
 \int_{NV_t}\mathcal P_t(\mathbf u/N)\,d\mathbf u
 +O(N^{D-1}).
\end{equation*}
Under the change of variables $\mathbf u=N\mathbf z$, we have
$d\mathbf u=N^D\,d\mathbf z$, and hence
\begin{equation*}
 \int_{NV_t}\mathcal P_t(\mathbf u/N)\,d\mathbf u
 =N^D\int_{V_t}\mathcal P_t(\mathbf z)\,d\mathbf z.
\end{equation*}
This proves \eqref{eq:weighted-lattice-volume}. To conclude, consider the effect of omitting lattice points on facets. A
facet of $NV_t$ lies on a hyperplane
\begin{equation*}
 \alpha_1x_1+\cdots+\alpha_Dx_D=\beta,
\end{equation*}
where at least one coefficient, say $\alpha_j$, is non-zero. Inside $NB$, there are
$O(N^{D-1})$ integer choices for the other $D-1$ coordinates. Once these coordinates are fixed, the equation determines $x_j$ uniquely, so there is at most
one possible integer value of $x_j$. Thus each facet contains $O(N^{D-1})$ lattice
points. Note that each $NV_t$ has at most $J$ facets and, as
$|\mathcal P_t(\mathbf m/N)|\leq M$, omitting the lattice points lying on
any chosen subset of these facets changes the sum by at most
\begin{equation*}
 M J\,O(N^{D-1})=O(N^{D-1}).
\end{equation*}
The implied constant is independent of $t$ and of the chosen subset.
\end{proof}

\subsection{Weighted lattice asymptotics}
\label{subsec:weighted-lattice-asymptotics}
We begin with a lemma which states the character expansion for $I_k(n;N)$ and enables us to interpret the matrix integral as a weighted lattice count.

\begin{lemma}
\label{lem:weighted-lattice-count}
Let $k,n,N \geq 1$ be integers. For a partition $\lambda\vdash n$, let $\chi^\lambda_\mu$ denote the
irreducible character of $S_n$ indexed by $\lambda$, evaluated on the
conjugacy class of cycle type $\mu$.  If
$\nu=(\nu_1,\ldots,\nu_k)$ is an ordered composition of $n$, write
$\chi^\lambda_\nu$ for the corresponding character value and put
\begin{equation}
\label{eq:Ak-lambda-def}
 A_k(\lambda)
 :=
 \sum_{\substack{\nu_1+\cdots+\nu_k=n\\ \nu_i\geq1}}
 \chi^\lambda_\nu.
\end{equation}
For an integer $r\geq1$, define the Frobenius coordinate set
\begin{equation}
\label{eq:Frobenius-coordinate-set}
 \mathcal F_r(n,N)
 :=
 \left\{
 (\mathbf a,\mathbf b)\in(\mathbb Z_{\geq0}^r)^2:
 \begin{aligned}
  &a_1>\cdots>a_r,\qquad N>b_1>\cdots>b_r,\\
  &\sum_{i=1}^r(a_i+b_i)=n-r
 \end{aligned}
 \right\}.
\end{equation}
For $(\mathbf a,\mathbf b)\in\mathcal F_r(n,N)$, write
$A_k(\mathbf a,\mathbf b):=A_k(\lambda)$, where
$\lambda=(a_1,\ldots,a_r\mid b_1,\ldots,b_r)$ is the corresponding
partition. Then
\begin{equation}
\label{eq:Ik-lattice-count}
 I_k(n;N)
 =
 \sum_{r\geq1}
 \sum_{(\mathbf a,\mathbf b)\in\mathcal F_r(n,N)}
 A_k(\mathbf a,\mathbf b)^2.
\end{equation}
Only finitely many of the sets $\mathcal F_r(n,N)$ are non-empty.
\end{lemma}

\begin{proof} By definition,
\begin{equation*}
 I_{k}(n;N) = \int_{U(N)}|F_{k,n}(U)|^2 dU, \qquad F_{k,n}(U)
:=\sum_{\substack{\nu_1+\cdots+\nu_k=n\\\nu_i\geq1}}
   \operatorname{Tr}(U^{\nu_1})\cdots\operatorname{Tr}(U^{\nu_k}).
\end{equation*}
Using \eqref{eq:power-sum-trace}, $F_{k,n}(U)$ can be written as a sum of products of power-sum
symmetric functions at the eigenvalues of $U$. By \eqref{eq:frobenius-character-formula}, we express each power-sum product $p_\mu = p_{\mu_1} \cdots p_{\mu_k}$ in the Schur basis. Applying it to every
ordered composition $\nu=(\nu_1,\ldots,\nu_k)$ of $n$ gives
\begin{align}
 F_{k,n}(U)
 &=\sum_{\substack{\nu_1+\cdots+\nu_k=n\\\nu_i\geq1}}
   p_{\nu_1}(U)\cdots p_{\nu_k}(U)\nonumber\\
 &=\sum_{\substack{\nu_1+\cdots+\nu_k=n\\\nu_i\geq1}}
   \sum_{\lambda\vdash n}\chi^\lambda_\nu s_\lambda(U)\nonumber\\
 &=\sum_{\lambda\vdash n}A_k(\lambda)s_\lambda(U).
 \label{eq:F-Schur-expansion}
\end{align}
By definition, $I_k(n;N)$ is the Haar integral of $|F_{k,n}(U)|^2$.
Substituting \eqref{eq:F-Schur-expansion} and applying
\eqref{eq:schur-orthogonality} eliminates all cross terms.  Since the
characters of $S_n$ are real-valued, $|A_k(\lambda)|^2=A_k(\lambda)^2$,
and therefore
\begin{equation}\label{eq:Ik-partition-sum}
 I_k(n;N)
 =\sum_{\substack{\lambda\vdash n\\\ell(\lambda)\leq N}}
 A_k(\lambda)^2.
\end{equation}
If $\lambda$ has Frobenius rank $r$, the identities
\eqref{eq:frobenius-size-length} show that the conditions in
\eqref{eq:Ik-partition-sum} are precisely those defining
$\mathcal F_r(n,N)$.  Every partition has a unique Frobenius rank and unique
Frobenius coordinates, so grouping \eqref{eq:Ik-partition-sum} by rank proves
\eqref{eq:Ik-lattice-count}.  The sum is finite: the strict inequalities in
the Frobenius coordinates give $a_i,b_i\geq r-i$, and hence
$|\lambda|\geq r^2$.
\end{proof}

\begin{proof}[Proof of Theorem~\ref{thm:leading-volume}]
Let $k \geq 1$ be a fixed integer, and fix a non-empty compact set $\mathcal K\subset(0,\infty)$ for which we set $c^+ := \max \mathcal{K}$. For an integer $N \geq 1$ and $c\in \mathcal K$, put $n=\lfloor cN\rfloor$; all estimates below are uniform for such $c$. We begin by
identifying the leading polynomial term from the character weight and then
use Lemma~\ref{lem:weighted-lattice-volume} to replace the weighted lattice
sum by a volume, before removing the ordering constraints by symmetry.

\subsubsection{The leading character weight}
Using the Cauchy identity, we will first obtain an expression that sums over partitions of all sizes and then extract its degree-$n$ part, which we then compare with the expansion
\eqref{eq:F-Schur-expansion}. For each $u\in\mathbb C$, we define the specialisation $\mathcal A_u$ by
\begin{equation}
\label{eq:virtual-alphabet}
 p_j(\mathcal A_u)=ju,
 \qquad j\geq1.
\end{equation}
As noted in Subsection~\ref{subseq:symmetric-functions}, these values determine
$f(\mathcal A_u)$ for every symmetric function $f$. Moreover, since $f$ is a
polynomial in finitely many power sums, $f(\mathcal A_u)$ is a polynomial in
the parameter $u$. For two sets of variables $X$ and $Y$, the Cauchy identity in
power-sum form is
\begin{equation}
\label{eq:cauchy-general}
 \sum_\lambda s_\lambda(X)s_\lambda(Y)
 =
 \exp\!\left(
  \sum_{j\geq1}\frac{p_j(X)p_j(Y)}{j}
 \right).
\end{equation}
We interpret \eqref{eq:cauchy-general}, and the specialised identities
derived from it, one homogeneous degree in $Y$ at a time. In each fixed
degree only finitely many terms occur; after applying $\mathcal A_u$,
both sides are polynomials in $u$, so taking the coefficient of $u^k$
is legitimate.

Applying $\mathcal A_u$ to the symmetric functions in $X$ and using
\eqref{eq:virtual-alphabet}, we obtain
\begin{align*}
 \sum_\lambda s_\lambda(\mathcal A_u)s_\lambda(Y)
 &=
 \exp\!\left(
  \sum_{j\geq1}\frac{p_j(\mathcal A_u)p_j(Y)}{j}
 \right)=
 \exp\!\left(u\sum_{j\geq1}p_j(Y)\right).
\end{align*}
We suppress the variables $Y$ from now on, writing $p_j=p_j(Y)$ and
$s_\lambda=s_\lambda(Y)$. For a non-empty partition $\lambda$, we can write
\begin{equation}
\label{eq:c-expansion}
 s_\lambda(\mathcal A_u)
 =\sum_{m\geq1}c_m(\lambda)u^m.
\end{equation}
As $s_{\varnothing}(\mathcal A_u)=1$ and
$s_\lambda(\mathcal A_0)=0$ for $\lambda\neq\varnothing$, the positive
powers of $u$ arise only from non-empty partitions. Taking the coefficient of $u^k$, with $k\geq1$, in the specialised Cauchy
identity above gives
\begin{equation*}
 \sum_{\lambda, \, |\lambda| \geq 1} c_k(\lambda)s_\lambda
 =
 \frac{1}{k!}\biggl(\sum_{j\geq1}p_j\biggr)^k.
\end{equation*}
Here $p_j$ has degree $j$, so
$p_{\nu_1}\cdots p_{\nu_k}$ has degree
$\nu_1+\cdots+\nu_k$.  Similarly, $s_\lambda$ has degree
$|\lambda|$. It follows that the terms of degree $n$ are
\begin{align*}
 \sum_{\lambda\vdash n}c_k(\lambda)s_\lambda
 &=
 \frac{1}{k!}
 \sum_{\substack{\nu_1+\cdots+\nu_k=n\\ \nu_i\geq1}}
 p_{\nu_1}\cdots p_{\nu_k}=
 \frac{1}{k!}
 \sum_{\lambda\vdash n}A_k(\lambda)s_\lambda,
\end{align*}
where the second equality follows from
\eqref{eq:frobenius-character-formula} and the definition
\eqref{eq:Ak-lambda-def} of $A_k(\lambda)$. As the Schur functions are linearly independent, their coefficients on the two sides are equal. Therefore
\begin{equation}
\label{eq:Ak-specialized-Schur}
 c_k(\lambda)=\frac{A_k(\lambda)}{k!},
 \qquad \lambda\vdash n.
\end{equation}
We first evaluate $c_k(\lambda)$ when $\lambda$ is a hook. This is not a restriction to a special class of partitions in the final argument; hooks are considered first as they serve as building blocks for Schur functions of arbitrary partitions. Namely, if a partition $\lambda = (a_1, \ldots, a_r | b_1, \ldots,b_r)$ has Frobenius rank $r$, then  Giambelli's
formula \cite[Ex.~9, p.~47]{kn:macdonald1995} states
\begin{equation}
\label{eq:frobenius-gambelli}
 s_\lambda
 =\det_{1\leq i,j\leq r}
   \bigl[s_{(a_i\mid b_j)}\bigr],
\end{equation}
where
$(a\mid b)$ is the Frobenius notation for a hook as introduced in \eqref{eq:frob-hook}.
Thus, after computing the specialised hook functions
$s_{(a\mid b)}(\mathcal A_u)$, we will return to a general partition,
apply the specialisation entrywise to this determinant, and then extract the
coefficient of $u^k$. 

Let $h_q$ and $e_q$ denote the complete and elementary symmetric
functions, and define their generating series by
\begin{align*}
 H(x):=\sum_{q\geq0}h_qx^q
      =\exp \, \biggl(\sum_{j\geq1}\frac{p_jx^j}{j}\biggr), \qquad
 E(y):=\sum_{q\geq0}e_qy^q
      =\exp \, \biggl(\sum_{j\geq1}
        \frac{(-1)^{j-1}p_jy^j}{j}\biggr).
\end{align*}
These two identities follow from equations 2.10 and 2.10' in \cite{kn:macdonald1995}. For a hook $(a\mid b)$, the Jacobi--Trudi formula gives
\begin{equation*}
 s_{(a\mid b)}
 =
 \sum_{j=0}^{b}(-1)^j h_{a+1+j}e_{b-j}.
\end{equation*}
The Schur functions indexed by hooks therefore satisfy
\begin{align*}
 \sum_{a,b\geq0}s_{(a\mid b)}x^ay^b
 &=
 \sum_{a,b\geq0}\sum_{j=0}^{b}
 (-1)^j h_{a+1+j}e_{b-j}x^ay^b\\
 &=
 E(y)\sum_{a,j\geq0}
 (-1)^j h_{a+1+j}x^ay^j\\
 &=
 E(y)\frac{H(x)-H(-y)}{x+y}.
\end{align*}
As $H(-y)E(y)=1$, this becomes
\begin{equation}
\label{eq:hook-generating-identity}
 \sum_{a,b\geq0}s_{(a\mid b)}x^ay^b
 =
 \frac{H(x)E(y)-1}{x+y}.
\end{equation} 
Next, we apply the specialisation \eqref{eq:virtual-alphabet} to the
generating series $H(x)$ and $E(y)$. Define $H_{\mathcal A_u}(x) := \sum_{q\geq 0}h_q(\mathcal A_u)x^q$ and $E_{\mathcal A_u}(y) := \sum_{q\geq 0}e_q(\mathcal A_u)y^q$. We can write both as
\begin{align*}
 H_{\mathcal A_u}(x)
 &=\exp \biggl(
   \sum_{j\geq1}\frac{p_j(\mathcal A_u)x^j}{j}
  \biggr)
  =\exp \biggl(u\sum_{j\geq1}x^j\biggr)
  =\exp \biggl(\frac{ux}{1-x}\biggr),\\
 E_{\mathcal A_u}(y)
 &=\exp\!\biggl(
   \sum_{j\geq1}\frac{(-1)^{j-1}p_j(\mathcal A_u)y^j}{j}
  \biggr)
  =\exp \biggl(u\sum_{j\geq1}(-1)^{j-1}y^j\biggr)
  =\exp \biggl(\frac{uy}{1+y}\biggr).
\end{align*}
Substituting these expressions into
\eqref{eq:hook-generating-identity}, we obtain
\begin{align}
\label{eq:hook-series}
 \sum_{a,b\geq0}s_{(a\mid b)}(\mathcal A_u)x^ay^b
 &=\frac{1}{x+y}
   \left[
    \exp\!\left(
      u\left(\frac{x}{1-x}+\frac{y}{1+y}\right)
    \right)-1
   \right]\\ \nonumber
 &=\frac{1}{x+y}
   \left[
    \exp\!\left(
      u\frac{x+y}{(1-x)(1+y)}
    \right)-1
   \right]\\
 &=\sum_{m\geq1}
   \frac{(x+y)^{m-1}}{m!(1-x)^m(1+y)^m}u^m. \nonumber
\end{align}
We also arrive at a separate expression for the series in \eqref{eq:hook-series} by comparing with the coefficients in \eqref{eq:c-expansion}. For each hook $(a\mid b)$, define $\theta_m(a,b)$ by
\begin{equation}
\label{eq:theta-expansion}
 s_{(a\mid b)}(\mathcal A_u)
 =\sum_{m\geq1}\theta_m(a,b)u^m,
\end{equation}
and thus $\theta_m(a,b)=c_m((a\mid b))$. In particular, for a hook of
size $a+b+1=n$, \eqref{eq:Ak-specialized-Schur} gives
\begin{equation*}
A_k((a\mid b))=k!\,\theta_k(a,b).
\end{equation*}
Using \eqref{eq:theta-expansion}, we can write
\begin{align*}
 \sum_{a,b\geq0}s_{(a\mid b)}(\mathcal A_u)x^ay^b
 &=\sum_{a,b\geq0}\sum_{m\geq1}
   \theta_m(a,b)u^mx^ay^b=\sum_{m\geq1}
   \biggl( \,
    \sum_{a,b\geq0}\theta_m(a,b)x^ay^b
   \biggr)u^m.
\end{align*}
The coefficients of $u^m$ in these two expressions must agree.
Therefore
\begin{equation}
\label{eq:theta-generating-function}
 \sum_{a,b\geq0}\theta_m(a,b)x^ay^b
 =
 \frac{(x+y)^{m-1}}{m!(1-x)^m(1+y)^m},
 \qquad m\geq1.
\end{equation}
We now set $d=m-1$ and expand the factors on the right-hand side of
\eqref{eq:theta-generating-function},
\begin{align*}
 (x+y)^d
 &=\sum_{j=0}^{d}\binom{d}{j}x^{d-j}y^j, \qquad
 (1-x)^{-m}
 =\sum_{r\geq0}\binom{r+d}{d}x^r, 
\end{align*}
and $ (1+y)^{-m}$ similarly. For a fixed $j$, the term $x^ay^b$ is obtained by taking
$r=a-d+j$ and $s=b-j$. Hence
\begin{equation}
\label{eq:Qm-exact}
 (-1)^b\theta_m(a,b)
 =
 \frac1{m!}
 \sum_{j=0}^{d}(-1)^j\binom{d}{j}
 \binom{a+j}{d}
 \binom{b+d-j}{d}.
\end{equation}
Here, as usual, a binomial coefficient is zero when its upper entry is a
non-negative integer smaller than its lower entry. Denote the polynomial on
the right-hand side of \eqref{eq:Qm-exact} by \(Q_m(a,b)\), so that
\begin{equation}
\label{eq:theta-polynomial}
 \theta_m(a,b)=(-1)^bQ_m(a,b).
\end{equation}
We now establish the bound $\deg Q_m\leq m-1$ and identify all terms of
$Q_m$ of total degree $m-1$ in the variables $a$ and $b$. We use the fact that
\begin{equation*}
 \sum_{j=0}^{d}(-1)^j\binom{d}{j}f(j)=0
\end{equation*}
whenever $f$ is a polynomial in $j$ of degree less than $d$. Consider
the coefficient of $a^pb^q$ in \eqref{eq:Qm-exact}. The coefficient of
$a^p$ in $\binom{a+j}{d}$ has degree at most $d-p$ as a polynomial in
$j$, while the coefficient of $b^q$ in $\binom{b+d-j}{d}$ has degree at
most $d-q$. Their product therefore has degree at most
$2d-p-q$ in $j$. If $p+q>d$, this degree is less than $d$, and the
alternating sum vanishes. Thus
\begin{equation*}
 \deg Q_m\leq d=m-1.
\end{equation*}

Focusing on the terms of $Q_m$ with total degree $d$, the case
$d=0$ is immediate, so assume $d\geq1$ and fix $p,q\geq0$ with
$p+q=d$. When viewed as a polynomial in $j$, the coefficient of $a^p$ in
$\binom{a+j}{d}$ has highest-power term
\begin{equation*}
 \frac{1}{d!}\binom{d}{p}j^{d-p},
\end{equation*}
while the coefficient of $b^q$ in $\binom{b+d-j}{d}$ has highest-power term
\begin{equation*}
 \frac{(-1)^{d-q}}{d!}\binom{d}{q}j^{d-q}.
\end{equation*}
Since $p+q=d$, their product has highest power $j^d$. Using the alternating
binomial identity
\begin{equation*}
 \sum_{j=0}^{d}(-1)^j\binom{d}{j}j^d
 =(-1)^d d!,
\end{equation*}
together with $\binom{d}{p}=\binom{d}{q}$, and collecting the resulting
monomials over all $p+q=d$, we obtain the entire highest-degree part of
$Q_m$. Recalling that $d=m-1$, we denote it by
\begin{equation}
\label{eq:Qm-leading}
 \Theta_m(a,b)
 =
 \frac{1}{m!(m-1)!}
 \sum_{j=0}^{m-1}(-1)^j\binom{m-1}{j}^{\!2}
 a^{m-1-j}b^j.
\end{equation}
Thus $\Theta_m$ consists precisely of the terms of $Q_m$ having total degree
$m-1$. By substituting $n=m-1$, $x=a$ and $y=b$ into the polynomial identity \eqref{eq:Legendre-homogeneous-expansion}, we may express $\Theta_m$ in terms of a Legendre polynomial, as stated in \eqref{eq:Theta-legendre-def}.

This separation is useful because the terms in $\Theta_m$ determine the
leading order when the Frobenius coordinates are scaled by $N$. The remainder
$Q_m-\Theta_m$ has degree at most $m-2$ since
$\Theta_m$ is precisely the degree-$(m-1)$ part of $Q_m$.
For $(\mathbf a,\mathbf b)\in\mathcal F_r(n,N)$, the Frobenius coordinates
satisfy
\begin{equation*}
 0\leq \frac{a_i}{N}\leq c^+,
 \qquad
 0\leq \frac{b_j}{N}<1.
\end{equation*}
Thus the scaled arguments $(a_i/N,b_j/N)$ remain in a fixed compact region.
It follows uniformly in $i$ and $j$ that
\begin{equation}
\label{eq:theta-leading}
 \theta_m(a_i,b_j)
 =
 (-1)^{b_j}\left(
  N^{m-1}\Theta_m(a_i/N,b_j/N)
  +O_{m,\mathcal K}(N^{m-2})
 \right),
\end{equation}
with no error term when $m=1$.
We now use the previous hook calculations to recover the character weight for
$(\mathbf a,\mathbf b)\in\mathcal F_r(n,N)$, with $\lambda=(a_1,\ldots,a_r\mid b_1,\ldots,b_r)$
denoting the corresponding partition. Using \eqref{eq:frobenius-gambelli}, and applying the
specialisation $\mathcal A_u$ entrywise, gives
\begin{equation*}
 s_\lambda(\mathcal A_u)
 =
 \det_{1\leq i,j\leq r}
 \bigl[s_{(a_i\mid b_j)}(\mathcal A_u)\bigr]
 =
 \det_{1\leq i,j\leq r}
 \biggl[
   \sum_{m\geq1}\theta_m(a_i,b_j)u^m
 \biggr].
\end{equation*}
By multilinearity of the determinant in its rows, the coefficient of $u^k$
is therefore
\begin{equation*}
 c_k(\lambda)
 =
\sum_{\substack{m_1+\cdots+m_r=k\\m_i\geq1}}
 \det_{1\leq i,j\leq r}
 \bigl[\theta_{m_i}(a_i,b_j)\bigr].
\end{equation*}
Using \eqref{eq:theta-polynomial}, each entry in column $j$ contributes the
factor $(-1)^{b_j}$. Factoring these signs from the columns gives
\begin{equation*}
 c_k(\lambda)
 =
 (-1)^{b_1+\cdots+b_r}
 \sum_{\substack{m_1+\cdots+m_r=k\\m_i\geq1}}
 \det_{1\leq i,j\leq r}
 \bigl[Q_{m_i}(a_i,b_j)\bigr].
\end{equation*}
From \eqref{eq:Ak-specialized-Schur}, we have 
$A_k(\lambda)=k!\,c_k(\lambda)$ and, using
$A_k(\mathbf a,\mathbf b):=A_k(\lambda)$ for
$(\mathbf a,\mathbf b)\in\mathcal F_r(n,N)$, we obtain
\begin{equation}
\label{eq:Ak-determinant}
 A_k(\mathbf a,\mathbf b)
 =
 (-1)^{b_1+\cdots+b_r}k!
 \sum_{\substack{m_1+\cdots+m_r=k\\m_i\geq1}}
 \det_{1\leq i,j\leq r}
 \bigl[Q_{m_i}(a_i,b_j)\bigr].
\end{equation}
In particular, the sum is empty when $k<r$, so
$A_k(\mathbf a,\mathbf b)=0$ in that case. We rename the sum above with the $k!$ prefactor as $\mathcal Q_{k,r}(\mathbf a,\mathbf b)$, which then gives
\begin{equation*}
  A_k(\mathbf a,\mathbf b)
 =
 (-1)^{b_1+\cdots+b_r} \mathcal Q_{k,r}(\mathbf a,\mathbf b).
\end{equation*}
A polynomial is alternating in a set of variables if interchanging two of them
changes its sign. Swapping two $b$-variables swaps two determinant columns.
Swapping $a_i$ and $a_j$, and then relabelling $m_i$ and $m_j$ in the
composition sum, swaps the corresponding rows. Thus
$\mathcal Q_{k,r}$ is alternating separately in the $a$-variables and the
$b$-variables. Every
alternating polynomial is divisible by the corresponding Vandermonde
polynomial, so
\begin{equation*}
 \Delta(\mathbf a)\Delta(\mathbf b)
 \mid \mathcal Q_{k,r}(\mathbf a,\mathbf b),
 \qquad
 \Delta(\mathbf a):=\prod_{i<j}(a_j-a_i),
 \quad
 \Delta(\mathbf b):=\prod_{i<j}(b_j-b_i).
\end{equation*}
The product $\Delta(\mathbf a)\Delta(\mathbf b)$ has total degree
$r(r-1)$.  On the other hand, each term in
\eqref{eq:Ak-determinant} has degree at most
\begin{equation*}
 \sum_{i=1}^r(m_i-1)=k-r.
\end{equation*}
If $r^2>k$, then $k-r<r(r-1)$, so the divisibility forces
\begin{equation}
\label{eq:rank-cutoff}
 A_k(\mathbf a,\mathbf b)=0.
\end{equation}
Thus only ranks $r\leq R:=\lfloor\sqrt{k}\rfloor$ contribute. Because $\mathcal K$ is a non-empty compact subset of $(0,\infty)$, it has
a positive minimum. As $R$ depends only on $k$, we may henceforth assume
that $N$ is sufficiently large that
$n=\lfloor cN\rfloor\geq R+1$ uniformly for $c\in\mathcal K$. Therefore, $n-r\geq1$ for every $1\leq r\leq R$.

Replacing each $Q_{m_i}$ in \eqref{eq:Ak-determinant} by its leading part
\eqref{eq:Qm-leading} gives
\begin{equation}
\label{eq:A-leading}
 A_k(\mathbf a,\mathbf b)
 =
 (-1)^{b_1+\cdots+b_r}N^{k-r}
 \Phi_{k,r}(\mathbf a/N,\mathbf b/N)
 +O_{k,\mathcal K}(N^{k-r-1}),
\end{equation}
where
\begin{equation*}
 \Phi_{k,r}(\mathbf x,\mathbf y)
 =
 k!\sum_{\substack{m_1+\cdots+m_r=k\\m_i\geq1}}
 \det_{1\leq i,j\leq r}
 \bigl[\Theta_{m_i}(x_i,y_j)\bigr].
\end{equation*}
By squaring \eqref{eq:A-leading}, we have
\begin{equation}
\label{eq:A-square-leading}
 A_k(\mathbf a,\mathbf b)^2
 =
 N^{2k-2r}\Phi_{k,r}(\mathbf a/N,\mathbf b/N)^2
 +O_{k,\mathcal K}(N^{2k-2r-1}).
\end{equation}
Recall from \eqref{eq:Ik-lattice-count} that
\begin{equation*}
 I_k(n;N)
 =
 \sum_{r\geq1}
 \sum_{(\mathbf a,\mathbf b)\in\mathcal F_r(n,N)}
 A_k(\mathbf a,\mathbf b)^2.
\end{equation*}
Thus, to sum the pointwise error in \eqref{eq:A-square-leading}, we need an
upper bound for the number of points in $\mathcal F_r(n,N)$ for each fixed
$r$. Dropping both the ordering conditions and the
upper bound $b_1<N$ leaves non-negative integers
$a_1,\ldots,a_r,b_1,\ldots,b_r$ satisfying only $ \sum_{i=1}^r(a_i+b_i)=n-r$. The number of such $2r$-tuples is
\begin{equation*}
 \binom{(n-r)+2r-1}{2r-1}
 =
 \binom{n+r-1}{2r-1},
\end{equation*}
since this counts the ways of distributing $n-r$ identical units among
$2r$ non-negative coordinates. Hence
\begin{equation*}
 \#\mathcal F_r(n,N)
 \leq
 \binom{n+r-1}{2r-1}
 =
 O_{r,\mathcal K}(N^{2r-1}),
\end{equation*}
where we used $n=\lfloor cN\rfloor$ with $c\in\mathcal K$.

For fixed $r$, the error in \eqref{eq:A-square-leading} is therefore summed
over at most $O_{r,\mathcal K}(N^{2r-1})$ points, giving a total contribution
\begin{equation*}
 O_{k,\mathcal K}(N^{2k-2r-1})
 \,O_{r,\mathcal K}(N^{2r-1})
 =
 O_{k,r,\mathcal K}(N^{2k-2}).
\end{equation*}
The rank cutoff restricts the outer sum to
$r\leq R$, so only finitely many values of $r$
occur. Substituting \eqref{eq:A-square-leading} into
\eqref{eq:Ik-lattice-count} therefore yields
\begin{equation}
\label{eq:Ik-after-leading-weight}
 I_k(n;N)
 =
 \sum_{r=1}^{R}N^{2k-2r}
 \sum_{(\mathbf a,\mathbf b)\in\mathcal F_r(n,N)}
 \Phi_{k,r}(\mathbf a/N,\mathbf b/N)^2
 +O_{k,\mathcal K}(N^{2k-2}).
\end{equation}

\subsubsection{Replacing the lattice sum by a volume}
We now aim to represent the count in \eqref{eq:Ik-after-leading-weight} by an
integral at leading order. For fixed $r$, the Frobenius coordinates satisfy
\begin{equation*}
 \sum_{i=1}^r(a_i+b_i)=n-r.
\end{equation*}
After scaling $a_i=Nx_i$ and $b_i=Ny_i$, this becomes
\begin{equation}
\label{eq:scaled-size}
 \sum_{i=1}^r(x_i+y_i)=d_{r,N}(c),
 \qquad
 d_{r,N}(c):=\frac{n-r}{N}.
\end{equation}
Thus the $2r$ scaled coordinates satisfy one linear constraint, leaving
$2r-1$ independent variables. By our restriction to sufficiently large $N$, the parameter $d_{r,N}(c)$ is positive
for every $1\leq r\leq R$. Since $n=\lfloor cN\rfloor$,
\begin{equation*}
 d_{r,N}(c)
 =\frac{n-r}{N}
 =c+O_r(N^{-1}),
\end{equation*}
so $d_{r,N}(c)\to c$ as $N\to\infty$.

It is convenient to introduce a general
parameter $d>0$ to describe both the finite-$N$ region, where $d=d_{r,N}(c)$, and its
limiting form, where $d=c$. To apply
Lemma~\ref{lem:weighted-lattice-volume} in these independent coordinates, we
eliminate $y_r$. For $d>0$, let
$V_{r,d}\subset\mathbb R^{2r-1}$ have coordinates
$(x_1,\ldots,x_r,y_1,\ldots,y_{r-1})$, set
\begin{equation}
\label{eq:yr-eliminate}
 y_r
 =
 d-\sum_{i=1}^r x_i-\sum_{i=1}^{r-1}y_i,
\end{equation}
and impose
\begin{equation}
\label{eq:Vrd}
 x_1\geq\cdots\geq x_r\geq0,
 \qquad
 1\geq y_1\geq\cdots\geq y_r\geq0.
\end{equation}
When $r=1$, the above sets of inequalities become
$1\geq d-x_1\geq0$ and $x_1\geq0$. These are linear inequalities, so $V_{r,d}$ is a convex
polytope. It is bounded because all coordinates are non-negative and their
total sum is $d$.

For $d=d_{r,N}(c)$, the same elimination reduces the
original lattice count to $2r-1$ integer coordinates; the size
condition determines the omitted coordinate uniquely as
\begin{equation*}
 b_r
 =
 n-r-\sum_{i=1}^r a_i-\sum_{i=1}^{r-1}b_i.
\end{equation*}
The right-hand side is an integer, so $y_r$ also
lies on the scaled lattice $N^{-1}\mathbb Z$. Thus every choice of the
remaining $2r-1$ scaled lattice coordinates satisfying the inequalities of
$V_{r,d_{r,N}(c)}$ determines a valid value of the omitted Frobenius
coordinate $b_r=Ny_r$. Equivalently, multiplying the independent scaled coordinates by $N$ sends
such points to integer points of $NV_{r,d_{r,N}(c)}\subset\mathbb R^{2r-1}$.
After removing the lattice points on the boundary hyperplanes corresponding to
the strict inequalities in \eqref{eq:Frobenius-coordinate-set}, the points
of $\mathcal F_r(n,N)$ are in one-to-one correspondence with the remaining
lattice points of $NV_{r,d_{r,N}(c)}\cap\mathbb Z^{2r-1}$.

Next, we define the weight
\begin{equation}
\label{eq:Prd-def}
 \mathcal P_{k,r,d}(\mathbf x,\tilde{\mathbf y})
 :=
 \Phi_{k,r}\!\left(
 \mathbf x,
 \biggl( y_1,\ldots,y_{r-1},
 d-\sum_{i=1}^r x_i-\sum_{i=1}^{r-1}y_i
 \biggr) \right)^2,
\end{equation}
where $\tilde{\mathbf y}=(y_1,\ldots,y_{r-1})$, and set
\begin{equation}
\label{eq:Jkr-def}
 J_{k,r}(d)
 :=
 \int_{V_{r,d}}
 \mathcal P_{k,r,d}(\mathbf x,\tilde{\mathbf y})
 \,d\mathbf x\,d\tilde{\mathbf y}.
\end{equation}

We verify the hypotheses of Lemma~\ref{lem:weighted-lattice-volume} with
$D=2r-1$ and with $d$ as the family parameter. As $\mathcal K$ is a compact subset of $(0,\infty)$, we may choose a
compact interval $\mathcal I\subset(0,\infty)$ whose interior contains $\mathcal K$. Let $d^+:=\max\mathcal I$. Since $1\leq r\leq R$ and
\begin{equation}
 |d_{r,N}(c)-c|
 \leq \frac{r+1}{N}
 \leq \frac{R+1}{N},
\end{equation}
we have $d_{r,N}(c)\in\mathcal I$ for all sufficiently large $N$,
uniformly for $c\in\mathcal K$ and $1\leq r\leq R$. The linear inequalities in \eqref{eq:Vrd} make $V_{r,d}$ convex. For
each $d\in \mathcal I$, choose coordinates satisfying
\begin{align*}
 1>y_1>\cdots>y_r>0,
 &\qquad
 \sum_{i=1}^r y_i<d,\\
 x_1>\cdots>x_r>0,
 &\qquad
 \sum_{i=1}^r x_i=d-\sum_{i=1}^r y_i.
\end{align*}
Every inequality in \eqref{eq:Vrd} is then strict. This produces an
interior point of $V_{r,d}$, proving that the polytope is
$(2r-1)$-dimensional. On $V_{r,d}$, one has
$0\leq y_i\leq1$ and $0\leq x_i\leq d\leq d^+$, so every polytope lies
in the common compact box
$\mathcal C_{\mathcal I}:=[0,d^+]^r\times[0,1]^{r-1}$.
As the inequalities form a fixed finite list, the number of facets is
bounded by the length of this list, and this bound is independent of $d$.

Choose a fixed bounded open neighbourhood $U$ of $\mathcal C_{\mathcal I}$. After
inserting the expression for $y_r$ from \eqref{eq:yr-eliminate},
$\mathcal P_{k,r,d}$ is polynomial jointly in the displayed $2r-1$
variables and in $d$. The polynomial and its first derivatives are
continuous on the compact set $\overline U\times \mathcal I$, so they admit a
common finite bound. The strict inequalities in \eqref{eq:Frobenius-coordinate-set} omit
only boundary lattice points satisfying $x_i=x_{i+1}$ or
$y_i=y_{i+1}$ for $1\leq i<r$, or $y_1=1$. The hyperplane-counting estimate
in the final part of the proof of
Lemma~\ref{lem:weighted-lattice-volume} shows that omitting these points
changes the sum by at most $O_{k,r,\mathcal K}(N^{2r-2})$.
Applying the lemma with $d=d_{r,N}(c)$ gives
\begin{equation}
\label{eq:weighted-volume-r}
 \sum_{(\mathbf a,\mathbf b)\in\mathcal F_r(n,N)}
 \Phi_{k,r}(\mathbf a/N,\mathbf b/N)^2
 =
 N^{2r-1}J_{k,r}(d_{r,N}(c))
 +O_{k,r,\mathcal K}(N^{2r-2}).
\end{equation} 
The parameter in \eqref{eq:weighted-volume-r} is $d_{r,N}(c)$ rather than
$c$, so we next replace it by its limiting value. Since
$n=\lfloor cN\rfloor$,
\begin{equation}
\label{eq:dN-to-c}
 |d_{r,N}(c)-c|
 =
 \left|\frac{\lfloor cN\rfloor-r}{N}-c\right|
 \leq\frac{r+1}{N}.
\end{equation}
The main term in \eqref{eq:weighted-volume-r} contains
$J_{k,r}(d_{r,N}(c))$, while the limiting coefficient involves
$J_{k,r}(c)$. By \eqref{eq:dN-to-c}, the two parameters differ by
$O_r(N^{-1})$. It is enough to prove that $J_{k,r}(d)$ is uniformly
Lipschitz in $d$: after multiplication by $N^{2r-1}$, an $O(N^{-1})$
change in $J_{k,r}$ contributes $O(N^{2r-2})$, matching the error in
\eqref{eq:weighted-volume-r}.

For the Lipschitz estimate, we continue to use the compact interval
$\mathcal I$ chosen above. In the independent
coordinates $\mathbf z=(x_1,\ldots,x_r,y_1,\ldots,y_{r-1})$, put
$S_r(\mathbf z):=\sum_{i=1}^r x_i+\sum_{i=1}^{r-1}y_i$, so that the
eliminated coordinate is $y_r=d-S_r(\mathbf z)$. For $r=1$, we have
$\mathbf z=x_1$. We refer separately to rank one only where the symbol
$y_{r-1}$ occurs; the estimates are otherwise identical.

Every $V_{r,d}$ with $d\in\mathcal I$ lies
in the fixed box
$\mathcal B_{r,\mathcal I}:=[0,d^+]^r\times[0,1]^{r-1}$, with the second
factor omitted for $r=1$. This holds because all the original coordinates are
non-negative, each $y_i$ is at most $1$, and
$\sum_{i=1}^r(x_i+y_i)=d\leq d^+$. We first estimate the change in the region.  For $r\geq2$, substituting
$y_r=d-S_r(\mathbf z)$ into $y_{r-1}\geq y_r\geq0$ gives
$S_r(\mathbf z)\leq d$ and $S_r(\mathbf z)+y_{r-1}\geq d$. All other
inequalities defining $V_{r,d}$ are independent of $d$. For $r=1$, the
corresponding conditions are $x_1\leq d$ and $x_1+1\geq d$. For $d,d'\in\mathcal I$, write
$\underline d:=\min\{d,d'\}$ and $\overline d:=\max\{d,d'\}$. A point belonging to exactly one of $V_{r,d}$ and $V_{r,d'}$ must lie
in one of the regions swept out by the two parameter-dependent
boundaries as the parameter varies between $d$ and $d'$. For
$r\geq2$, this gives
\begin{align}
\label{eq:Vrd-symmetric-difference-regions}
 V_{r,d}\mathbin{\triangle}V_{r,d'}
 \subset{}&
 \bigl\{\mathbf z\in\mathcal B_{r,\mathcal I}:
 \underline d\leq S_r(\mathbf z)\leq\overline d\bigr\}
 \nonumber\\
 &\qquad\qquad\mathbin{\cup}
 \bigl\{\mathbf z\in\mathcal B_{r,\mathcal I}:
 \underline d\leq S_r(\mathbf z)+y_{r-1}
 \leq\overline d\bigr\}.
\end{align}
For $r=1$, replace the two displayed functions by $x_1$ and $x_1+1$.

Fixing every coordinate except $x_1$, either set on the right of
\eqref{eq:Vrd-symmetric-difference-regions} restricts $x_1$ to an interval
of length at most $\overline d-\underline d=|d-d'|$.  The projection of
$\mathcal B_{r,\mathcal I}$ onto the other coordinates has bounded volume,
so Fubini's theorem gives
\begin{equation}
\label{eq:Vrd-symmetric-difference-volume}
 \operatorname{vol}_{2r-1}
 \bigl(V_{r,d}\mathbin{\triangle}V_{r,d'}\bigr)
 \ll_{r,\mathcal K}|d-d'|.
\end{equation}
For $r=1$, the sets in question are intervals of length at most $|d-d'|$, and the same estimate applies. We next estimate the change in the weight.  Formula \eqref{eq:Prd-def}
expresses $\mathcal P_{k,r,d}(\mathbf z)$ as a polynomial in
$(\mathbf z,d)$.  On the compact set
$\mathcal B_{r,\mathcal I}\times\mathcal I$, there are constants
$M_0,M_1\ll_{k,r,\mathcal K}1$ for which
\begin{equation*}
 \bigl|\mathcal P_{k,r,d}(\mathbf z)\bigr|\leq M_0,
 \qquad
 \left|\partial_d\mathcal P_{k,r,d}(\mathbf z)\right|\leq M_1.
\end{equation*}
The mean value theorem gives
$|\mathcal P_{k,r,d}(\mathbf z)-\mathcal P_{k,r,d'}(\mathbf z)|
\leq M_1|d-d'|$.
Splitting the two integrals over their common part and their disjoint
parts, we obtain
\begin{align*}
 |J_{k,r}(d)-J_{k,r}(d')|
 &\leq
 \int_{V_{r,d}\cap V_{r,d'}}
 \bigl|\mathcal P_{k,r,d}-\mathcal P_{k,r,d'}\bigr|\,d\mathbf z
 \\
 &\quad+
 \int_{V_{r,d}\setminus V_{r,d'}}
 \bigl|\mathcal P_{k,r,d}\bigr|\,d\mathbf z
 +
 \int_{V_{r,d'}\setminus V_{r,d}}
 \bigl|\mathcal P_{k,r,d'}\bigr|\,d\mathbf z
 \\
 &\leq
 M_1\operatorname{vol}_{2r-1}(\mathcal B_{r,\mathcal I})|d-d'|
 +M_0\operatorname{vol}_{2r-1}
 \bigl(V_{r,d}\mathbin{\triangle}V_{r,d'}\bigr)
 \\
 &\ll_{k,r,\mathcal K}|d-d'|.
\end{align*}
We have proved
\begin{equation}
\label{eq:Jkr-lipschitz}
 |J_{k,r}(d)-J_{k,r}(d')|
 \ll_{k,r,\mathcal K}|d-d'|,
 \qquad d,d'\in\mathcal I.
\end{equation}
Taking $d=d_{r,N}(c)$ and $d'=c$, and applying
\eqref{eq:dN-to-c}, gives
\begin{equation}
\label{eq:slice-lipschitz}
 J_{k,r}(d_{r,N}(c))
 =J_{k,r}(c)+O_{k,r,\mathcal K}(N^{-1}),
\end{equation}
uniformly for $c\in\mathcal K$.

Substituting \eqref{eq:slice-lipschitz} into
\eqref{eq:weighted-volume-r}, we obtain
\begin{equation}
\label{eq:weighted-volume-r-at-c}
 \sum_{(\mathbf a,\mathbf b)\in\mathcal F_r(n,N)}
 \Phi_{k,r}(\mathbf a/N,\mathbf b/N)^2
 =N^{2r-1}J_{k,r}(c)
 +O_{k,r,\mathcal K}(N^{2r-2}).
\end{equation}
The error from \eqref{eq:slice-lipschitz} has the same order as the
lattice-to-volume error.  Inserting
\eqref{eq:weighted-volume-r-at-c} into
\eqref{eq:Ik-after-leading-weight} gives
\begin{align*}
 I_k(n;N)
 ={}&
 \sum_{r=1}^{R}N^{2k-2r}
 \left(
 N^{2r-1}J_{k,r}(c)
 +O_{k,r,\mathcal K}(N^{2r-2})
 \right)
 +O_{k,\mathcal K}(N^{2k-2}).
\end{align*}
For each $r$, the main term has order $N^{2k-1}$ and the error has order
$N^{2k-2}$.  The number $R=\lfloor\sqrt{k}\rfloor$ depends only on $k$,
so the sum over $r$ preserves this error order.  We arrive at
\begin{equation}
\label{eq:volume-before-delta}
 I_k(n;N)
 =N^{2k-1}\sum_{r=1}^{R}J_{k,r}(c)
 +O_{k,\mathcal K}(N^{2k-2}).
\end{equation}
\subsubsection{The symmetric integral}

We now rewrite $J_{k,r}(c)$ in a symmetric form. By
\eqref{eq:Jkr-def}, $y_r$ is determined by
\eqref{eq:yr-eliminate}; equivalently, this constraint may be imposed using
the Dirac delta distribution. Integrating first with respect to $y_r$ gives
\begin{equation}
\label{eq:Jr-ordered-delta}
 J_{k,r}(c)
 =
 \int_{\substack{x_1>\cdots>x_r\geq0\\
                  1>y_1>\cdots>y_r\geq0}}
 \Phi_{k,r}(\mathbf x,\mathbf y)^2
 \delta\!\left(\sum_{i=1}^r(x_i+y_i)-c\right)
 \,d\mathbf x\,d\mathbf y.
\end{equation}
Indeed, the $y_r$-integration sets
\begin{equation*}
 y_r
 =
 c-\sum_{i=1}^r x_i-\sum_{i=1}^{r-1}y_i,
\end{equation*}
which is exactly \eqref{eq:yr-eliminate} with $d=c$. No additional factor
appears because the derivative of the delta argument with respect to $y_r$
is $1$. Replacing the strict inequalities by weak ones does not change the
integral, since the corresponding boundary hyperplanes have measure zero
within the constrained region.

It remains to remove the ordering conditions. The determinant defining
$\Phi_{k,r}$ is alternating separately in the $\mathbf x$- and
$\mathbf y$-variables: permuting the $y_j$ permutes determinant columns,
while permuting the $x_i$ permutes rows after relabelling the summation
variables $(m_1,\ldots,m_r)$. Hence, for $\sigma,\tau\in S_r$,
\begin{equation}
\label{eq:Phi-alternating}
 \Phi_{k,r}(\mathbf x_\sigma,\mathbf y_\tau)
 =
 \operatorname{sgn}(\sigma)\operatorname{sgn}(\tau)
 \Phi_{k,r}(\mathbf x,\mathbf y).
\end{equation}
It follows that $\Phi_{k,r}^2$ is invariant under separate permutations of
the $x$- and $y$-coordinates, as is the delta factor. Up to boundary sets of
measure zero, the domain
$\mathbb R_{\geq0}^r\times[0,1]^r$ is the disjoint union of the
$(r!)^2$ regions obtained by fixing one ordering of the $x_i$ and one
ordering of the $y_i$. Coordinate permutations preserve both the measure and
the integrand, so each region contributes equally. Therefore
\begin{equation}
\label{eq:Jr-unordered-delta}
 J_{k,r}(c)
 =
 \frac1{(r!)^2}
 \int_{\mathbb R_{\geq0}^r\times[0,1]^r}
 \Phi_{k,r}(\mathbf x,\mathbf y)^2
 \delta\!\left(\sum_{i=1}^r(x_i+y_i)-c\right)
 \,d\mathbf x\,d\mathbf y.
\end{equation}
Combining \eqref{eq:Jr-unordered-delta} with
\eqref{eq:volume-before-delta} proves the asserted asymptotic and the
integral formula for $\delta_k(c)$, uniformly for $c\in\mathcal K$.
\end{proof}

\subsection{Signed lattice asymptotics}
\label{subsec:signed-lattice-asymptotics}

The weighted model is well-suited to asymptotic analysis, but its lattice
points carry non-trivial character weights. We now obtain a second explicit count for $I_k(n;N)$ that is signed, then
organise its configurations into blocks and cycles before taking the large-$N$ limit. The exact lattice-count model arises
from the determinantal formula for products of traces stated below; throughout, we use the notation $[N]_0:=\{0,\ldots,N-1\}$. 
\begin{theorem}[\cite{kn:irfan26}, Theorem 12]
\label{thm:AoT}
Let $N,k\geq1$ be integers, and let
$\mathbf p=(p_1,\ldots,p_k)$ and
$\mathbf q=(q_1,\ldots,q_k)$ belong to $\mathbb Z_{\geq1}^k$. Then
\begin{align}
\label{eq:CN-definition-theorem}
&\mathbb E_{U(N)}
\left[
 \prod_{i=1}^{k}\operatorname{Tr}(U^{p_i})
 \prod_{j=1}^{k}\overline{\operatorname{Tr}(U^{q_j})}
\right]
=
\sum_{\mathbf r,\mathbf s\in[N]_0^k}
\det_{a,b\in[N]_0}
\left[
 \mathds{1}_{\{b=a-h_a(\mathbf r,\mathbf s)\}}
\right],
\end{align}
where $\mathds{1}_{\{E\}}$ is the
indicator of the condition $E$, equal to $1$ when $E$ holds and $0$
otherwise. For $a\in[N]_0$, define
\begin{equation}
\label{eq:zero-condition}
 h_a(\mathbf r,\mathbf s)
 :=
 \sum_{i=1}^{k}p_i\mathds{1}_{\{r_i=a\}}
 -
 \sum_{j=1}^{k}q_j\mathds{1}_{\{s_j=a\}}.
\end{equation}
\end{theorem}

For fixed $\mathbf r$ and $\mathbf s$, the row indexed by $a$ in
\eqref{eq:CN-definition-theorem} has a single $1$ when
$a-h_a(\mathbf r,\mathbf s)\in[N]_0$, and is otherwise zero. The determinant is therefore non-zero exactly when every row
contains a $1$ and these $1$'s occupy distinct columns; in
that case the matrix is a permutation matrix and its determinant is the
sign of the corresponding permutation. This observation gives the next
exact representation.

\begin{proposition}
\label{prop:signed-lattice-count}
For $k,n,N\in\mathbb Z_{\geq 1}$, let $A_{k,n,N}^+$ and $A_{k,n,N}^-$ be the
collections of lattice points
\begin{equation}
\label{eq:signed-count-matrix}
x=
\begin{pmatrix}
 x_1^{(1)}&x_2^{(1)}&\cdots&x_k^{(1)}\\
 x_1^{(2)}&x_2^{(2)}&\cdots&x_k^{(2)}\\
 x_1^{(3)}&x_2^{(3)}&\cdots&x_k^{(3)}\\
 x_1^{(4)}&x_2^{(4)}&\cdots&x_k^{(4)}
\end{pmatrix}
\in\mathbb Z^{4k}
\end{equation}
satisfying:
\begin{enumerate}[label=\textup{(\roman*)},leftmargin=*,itemsep=0.7em]
\item
$0\leq x_i^{(1)},x_i^{(2)}\leq N-1$ and
$x_i^{(3)},x_i^{(4)}\geq1$ for $1\leq i\leq k$;

\item
$\sum_{i=1}^k x_i^{(3)}=n$ and
$\sum_{i=1}^k x_i^{(4)}=n$;

\item the map $\phi_x:[N]_0\to\mathbb Z$ defined by
\begin{equation}
\label{eq:phi-x-definition}
\phi_x(m)
=
m-
\sum_{i=1}^k x_i^{(3)}\mathds{1}_{\{x_i^{(1)}=m\}}
+
\sum_{i=1}^k x_i^{(4)}\mathds{1}_{\{x_i^{(2)}=m\}}
\end{equation}
is an even permutation of $[N]_0$ for
$A_{k,n,N}^+$, and an odd permutation of $[N]_0$ for
$A_{k,n,N}^-$.
\end{enumerate}
Then
\begin{equation}
\label{eq:exact-signed-lattice-count}
I_k(n;N)=\#A_{k,n,N}^+-\#A_{k,n,N}^-.
\end{equation}
\end{proposition}

\begin{proof}
Expanding the absolute square in \eqref{eq:variance-rmt} gives
\begin{align}
\label{eq:expanded-I_k-pq}
I_k(n;N)
=
\sum_{\substack{p_1,\ldots,p_k\geq1\\p_1+\cdots+p_k=n}}
\sum_{\substack{q_1,\ldots,q_k\geq1\\q_1+\cdots+q_k=n}}
\int_{U(N)}
\prod_{i=1}^k\operatorname{Tr}(U^{p_i})
\prod_{i=1}^k\overline{\operatorname{Tr}(U^{q_i})}\,dU.
\end{align}
In particular, $\mathbf p=(p_1,\ldots,p_k)$ and $\mathbf q=(q_1,\ldots,q_k)$ range over the ordered
compositions of $n$ into $k$ positive parts.  For fixed $\mathbf p$ and
$\mathbf q$, Theorem~\ref{thm:AoT} evaluates the integral by introducing
$\mathbf r,\mathbf s\in[N]_0^k$ and gives
\begin{equation}
\label{expanded-out-Ik-exact-count}
I_k(n;N)
=
\sum_{\substack{p_1,\ldots,p_k\geq1\\p_1+\cdots+p_k=n}}
\sum_{\substack{q_1,\ldots,q_k\geq1\\q_1+\cdots+q_k=n}}
\sum_{\mathbf r,\mathbf s\in[N]_0^k}
\det_{a,b\in[N]_0}
\left[
\mathds{1}_{\{b=a-h_a(\mathbf r,\mathbf s)\}}
\right],
\end{equation}
where
\begin{equation*}
h_a(\mathbf r,\mathbf s)
=
\sum_{i=1}^k p_i\mathds{1}_{\{r_i=a\}}
-
\sum_{i=1}^k q_i\mathds{1}_{\{s_i=a\}}.
\end{equation*}
For each choice of $\mathbf p$, $\mathbf q$, $\mathbf r$ and $\mathbf s$ in
these sums, define $x\in\mathbb Z^{4k}$ by
\begin{equation}
\label{eq:indices-exact-proposition}
x_i^{(1)}=r_i,
\qquad
x_i^{(2)}=s_i,
\qquad
x_i^{(3)}=p_i,
\qquad
x_i^{(4)}=q_i.
\end{equation}
The bounds on $\mathbf r$ and $\mathbf s$, together with the composition
conditions on $\mathbf p$ and $\mathbf q$, are precisely conditions
\textup{(i)} and \textup{(ii)}.  Conversely, the four rows of any $x$
satisfying \textup{(i)} and \textup{(ii)} recover a unique choice of
$\mathbf p$, $\mathbf q$, $\mathbf r$ and $\mathbf s$.  This construction is
therefore a bijection.

For a lattice point obtained in this way, \eqref{eq:phi-x-definition} and the
formula for $h_a(\mathbf r,\mathbf s)$ give
\begin{equation}
\label{phi-exact-lattice-count}
\phi_x(a)=a-h_a(\mathbf r,\mathbf s).
\end{equation}
Write
\begin{equation}
\label{eq:M-exact-lattice-count}
M_x=
\left[
\mathds{1}_{\{b=\phi_x(a)\}}
\right]_{a,b\in[N]_0}.
\end{equation}
The summand indexed by $\mathbf p$, $\mathbf q$, $\mathbf r$ and $\mathbf s$
is $\det M_x$. If $\phi_x(a)\notin[N]_0$ for some $a$, then row $a$ of
$M_x$ is zero. If every value of $\phi_x$ lies in $[N]_0$ but $\phi_x$ is
not injective, two rows of $M_x$ are the same standard basis vector.  In both
cases, $\det M_x=0$. Thus $\det M_x$ can be non-zero only if $\phi_x$
takes values in $[N]_0$ and is injective. As $[N]_0$ has $N$ elements,
every injective map from $[N]_0$ to itself is bijective. Therefore, 
$\det M_x$ can be non-zero only when $\phi_x$ is a permutation of $[N]_0$.
When $\phi_x$ is a permutation, the Leibniz expansion is
\begin{equation}
\label{eq:detM-exact-lattice-count}
\det M_x
=
\sum_{\tau\in S_N}
\operatorname{sgn}(\tau)
\prod_{a\in[N]_0}\mathds{1}_{\{\tau(a)=\phi_x(a)\}}.
\end{equation}
The product is non-zero only for $\tau=\phi_x$, so
$\det M_x=\operatorname{sgn}(\phi_x)$.  Under the bijection above, the points
in $A_{k,n,N}^+$ contribute $+1$, the points in $A_{k,n,N}^-$ contribute
$-1$, and all other points contribute $0$. Summing these contributions gives
\eqref{eq:exact-signed-lattice-count}.
\end{proof}

\subsubsection{Blocks and cycles}

To organise the configurations in the exact signed count according to the
positions shared by their coordinates, we now make the block construction
explicit. In \eqref{expanded-out-Ik-exact-count}, the vectors $\mathbf p$
and $\mathbf q$ range over the ordered compositions of $n$, while $r_i$
specifies the position assigned to $p_i$ and $s_j$ specifies the position
assigned to $q_j$.
To distinguish the two compositions while grouping their parts by position,
introduce the signed index set
\begin{equation*}
 \Omega_k:=\{1^+,\ldots,k^+\}\cup\{1^-,\ldots,k^-\}.
\end{equation*}
Here $i^+$ labels the part $p_i$, and $j^-$ labels the part $q_j$. The
superscripts indicate the signs with which these parts enter
$h_a(\mathbf r,\mathbf s)$ in \eqref{eq:zero-condition}.

For $X\in A_{k,n,N}^+\cup A_{k,n,N}^-$, the identification in
\eqref{eq:indices-exact-proposition} allows us to write its four rows as
\begin{equation*}
 r_i:=X_i^{(1)},\qquad
 s_i:=X_i^{(2)},\qquad
 p_i:=X_i^{(3)},\qquad
 q_i:=X_i^{(4)}
 \qquad(1\leq i\leq k).
\end{equation*}
Thus the label $i^+$ is placed at $r_i\in[N]_0$, while $j^-$ is placed at
$s_j\in[N]_0$. An \emph{occupied position} is an element of $[N]_0$ equal
to at least one of the coordinates $r_i$ or $s_j$. The set of all occupied
positions is
\begin{equation}
\label{eq:occupied-position-set-lattice}
 T:=\{r_i:1\leq i\leq k\}\cup\{s_j:1\leq j\leq k\}
 \subseteq[N]_0.
\end{equation}
To group the signed labels placed at the same occupied position, define
\begin{equation}
\label{eq:occupied-position-blocks-lattice}
 B_a:=\{i^+:r_i=a\}\cup\{j^-:s_j=a\},
 \qquad
 x_{B_a}:=a
 \quad(a\in T),
 \qquad
 \pi:=\{B_a:a\in T\}.
\end{equation}
The set $B_a$ is the \emph{block} at $a$, and $x_{B_a}$ is its block
position. Every signed label is placed at exactly one position. The blocks
$B_a$ are therefore non-empty, pairwise disjoint, and have union $\Omega_k$.
Thus $\pi$ is a set partition of $\Omega_k$.

To express the combined effect of the parts placed at one position, we
associate a net displacement with each block. For $B\in\pi$, let
\begin{equation*}
 B^+:=\{i:i^+\in B\},
 \qquad
 B^-:=\{j:j^-\in B\}
\end{equation*}
be the sets of positive and negative indices in $B$. For arbitrary
$\mathbf u,\mathbf v\in\mathbb R^k$, we use the same formula as in
\eqref{eq:block-displacement}:
\begin{equation}
\label{eq:block-displacement-lattice}
 d_B(\mathbf u,\mathbf v)
 :=
 \sum_{i\in B^+}u_i-
 \sum_{j\in B^-}v_j.
\end{equation}

We next describe how the permutation associated with $X$ acts on the
occupied positions. For this point $X$, let $\phi_X$ denote the map
$\phi_x$ in
\eqref{eq:phi-x-definition} with $x=X$. For $m\in[N]_0\setminus T$,
\eqref{eq:zero-condition} gives
$h_m(\mathbf r,\mathbf s)=0$, and \eqref{phi-exact-lattice-count} gives
$\phi_X(m)=m$. Thus $\phi_X$ fixes $[N]_0\setminus T$. It is also a
bijection of $[N]_0$, so it maps $T$ onto itself. Its action on $T$ induces
a permutation $\sigma$ of the blocks of $\pi$, defined by
\begin{equation}
\label{eq:induced-block-permutation-lattice}
 \sigma(B_a):=B_{\phi_X(a)}
 \qquad(a\in T).
\end{equation}
At the position $x_B$, the definitions of $h_a$ and $d_B$ give
$h_{x_B}(\mathbf r,\mathbf s)=d_B(\mathbf p,\mathbf q)$. Substitution into
\eqref{phi-exact-lattice-count} gives the block equation
\begin{equation}
\label{eq:finite-block-equation-body}
 d_B(\mathbf p,\mathbf q)=x_B-x_{\sigma(B)}
 \qquad(B\in\pi).
\end{equation}
The permutation $\phi_X$ fixes every point outside $T$, so
$\operatorname{sgn}(\phi_X)=\operatorname{sgn}(\sigma)$.

To describe the possible block structures independently of $X$, let
$\Pi(\Omega_k)$ denote the set of set partitions of $\Omega_k$, and let
$\mathfrak S(\pi)$ denote the set of all permutations of the blocks of
$\pi$, viewed as a group under composition. For
$\sigma\in\mathfrak S(\pi)$, let $\operatorname{Cyc}(\sigma)$ denote the
set of cycles in its disjoint cycle decomposition. We orient each cycle
according to $\sigma$, so the block after $B$ is $\sigma(B)$. Define
\begin{equation}
\label{eq:block-permutation-set-lattice}
 \mathfrak B_k
 :=
 \left\{
 (\pi,\sigma):
 \begin{array}{l}
  \pi\in\Pi(\Omega_k),\quad \sigma\in\mathfrak S(\pi),\\
  \text{the union of the blocks in each cycle of $\sigma$ contains}\\
  \text{at least one $+$ index and at least one $-$ index}
 \end{array}
 \right\}.
\end{equation}
We call a pair $(\pi,\sigma)\in\mathfrak B_k$ admissible. Every pair
obtained from a point in the exact count is admissible. To verify this,
sum the block equations around one cycle. The right hand side is zero,
whereas a cycle containing indices of only one sign gives either a
non-empty sum of positive $p_i$ or the negative of a non-empty sum of
positive $q_j$ on the left. To prepare the exact count for asymptotic estimation, the next lemma
rewrites \eqref{eq:exact-signed-lattice-count} as a sum indexed by
admissible pairs.

\begin{lemma}
\label{lem:exact-admissible-block-decomposition}
Let $k,n,N\in\mathbb Z_{\geq1}$. For a partition $\pi$ of $\Omega_k$, write
\begin{equation*}
[N]_0^\pi
:=
\left\{
\mathbf x=(x_B)_{B\in\pi}:x_B\in[N]_0
\right\}.
\end{equation*}
Thus $\mathbf x\in[N]_0^\pi$ assigns a position $x_B$ to each block
$B\in\pi$. For $(\pi,\sigma)\in\mathfrak B_k$, let
$\mathcal C_N(\pi,\sigma;n)$ be the set of triples
\begin{equation*}
(\mathbf p,\mathbf q,\mathbf x)
\in
\mathbb Z_{>0}^{k}\times\mathbb Z_{>0}^{k}\times[N]_0^\pi
\end{equation*}
satisfying these three conditions:
\begin{enumerate}
[label=\textup{(\roman*)},leftmargin=*,itemsep=0.7em]
\item $\displaystyle\sum_{i=1}^k p_i=n$;
\item $x_B\neq x_C$ for all distinct $B,C\in\pi$;
\item $d_B(\mathbf p,\mathbf q)=x_B-x_{\sigma(B)}$ for every $B\in\pi$.
\end{enumerate}
Summing the block equations in condition \textup{(iii)} over $B\in\pi$
and using condition \textup{(i)} gives $\sum_{j=1}^kq_j=n$. Then
\begin{equation}
\label{eq:exact-admissible-block-decomposition}
I_k(n;N)
=
\sum_{(\pi,\sigma)\in\mathfrak B_k}
\operatorname{sgn}(\sigma)\,
\#\mathcal C_N(\pi,\sigma;n).
\end{equation}
\end{lemma}

\begin{proof}
We prove \eqref{eq:exact-admissible-block-decomposition} by
constructing a bijection that preserves the permutation sign
between $A_{k,n,N}^+\cup A_{k,n,N}^-$ and the disjoint union
of the sets $\mathcal C_N(\pi,\sigma;n)$ over
$(\pi,\sigma)\in\mathfrak B_k$.

Fix $X\in A_{k,n,N}^+\cup A_{k,n,N}^-$ and define
$T,\pi,\mathbf x,\sigma$ by the construction preceding the
lemma. That discussion shows that $(\pi,\sigma)\in\mathfrak B_k$
and $\operatorname{sgn}(\phi_X)=\operatorname{sgn}(\sigma)$.
The block positions are pairwise distinct, and
$\sum_i p_i=n$ together with
\eqref{eq:finite-block-equation-body} gives
$(\mathbf p,\mathbf q,\mathbf x)\in
\mathcal C_N(\pi,\sigma;n)$. For any triple in $\mathcal C_N(\pi,\sigma;n)$, summing
the block equations over $B\in\pi$ gives
$\sum_j q_j=\sum_i p_i=n$, so both $\mathbf p$ and
$\mathbf q$ are compositions of $n$ into $k$ positive parts.

To construct the inverse map, take $(\pi,\sigma)\in\mathfrak B_k$ and
$(\mathbf p,\mathbf q,\mathbf x)\in
\mathcal C_N(\pi,\sigma;n)$. Each signed label belongs to a unique block
of $\pi$, so the assignments
\begin{equation*}
r_i:=x_B\quad\text{when }i^+\in B,
\qquad
s_j:=x_B\quad\text{when }j^-\in B
\end{equation*}
are well defined and give $\mathbf r,\mathbf s\in[N]_0^k$. Let
$X\in\mathbb Z^{4k}$ be the point corresponding to
$(\mathbf r,\mathbf s,\mathbf p,\mathbf q)$ under the identification
\eqref{eq:indices-exact-proposition}.

Every coordinate $r_i$ or $s_j$ equals $x_B$ for some $B\in\pi$.
Conversely, every block is non-empty, so each $x_B$ is equal to at least
one coordinate $r_i$ or $s_j$. The occupied set of $X$, defined by the same
formula as \eqref{eq:occupied-position-set-lattice}, is therefore
\begin{equation*}
T
=
\{r_i:1\leq i\leq k\}
\cup
\{s_j:1\leq j\leq k\}
=
\{x_B:B\in\pi\}.
\end{equation*}
Condition \textup{(ii)} ensures that distinct blocks have distinct
positions, so the blocks determined by $X$ are exactly the blocks of
$\pi$.
At the position $x_B$, \eqref{eq:zero-condition} and
\eqref{eq:block-displacement-lattice} give
\begin{equation*}
h_{x_B}(\mathbf r,\mathbf s)
=
\sum_{i\in B^+}p_i-
\sum_{j\in B^-}q_j
=d_B(\mathbf p,\mathbf q).
\end{equation*}
The block equation \eqref{eq:finite-block-equation-body} and
\eqref{eq:phi-x-definition} now give
\begin{equation*}
\phi_X(x_B)
=x_B-d_B(\mathbf p,\mathbf q)
=x_{\sigma(B)}.
\end{equation*}
For $m\in[N]_0\setminus T$, no index is assigned to $m$, so $\phi_X(m)=m$. Thus
$\phi_X$ fixes the unoccupied positions and permutes the occupied positions
according to $\sigma$. In particular, $\phi_X$ is a permutation of $[N]_0$
and
\begin{equation*}
\operatorname{sgn}(\phi_X)=\operatorname{sgn}(\sigma).
\end{equation*}
We also have $r_i,s_j\in[N]_0$, $p_i,q_j>0$, and
$\sum_i p_i=\sum_jq_j=n$. The lattice point $X$ therefore belongs to
$A_{k,n,N}^+$ when $\sigma$ is even and to $A_{k,n,N}^-$ when $\sigma$ is
odd.

The two constructions are inverse and give a bijection
\begin{equation*}
A_{k,n,N}^{+}\sqcup A_{k,n,N}^{-}
\longleftrightarrow
\bigsqcup_{(\pi,\sigma)\in\mathfrak B_k}
\mathcal C_N(\pi,\sigma;n).
\end{equation*}
Under this bijection, a point $X$ contributes
$\operatorname{sgn}(\phi_X)$ to the signed count in
\eqref{eq:exact-signed-lattice-count}, while the corresponding triple
contributes $\operatorname{sgn}(\sigma)$ to the sum in
\eqref{eq:exact-admissible-block-decomposition}. These signs are equal, so
summing over the bijection proves
\eqref{eq:exact-admissible-block-decomposition}.
\end{proof}

\subsubsection{The asymptotic count}

To estimate the summands in
\eqref{eq:exact-admissible-block-decomposition}, fix an admissible pair
$(\pi,\sigma)\in\mathfrak B_k$ and write
\begin{equation*}
\operatorname{Cyc}(\sigma)=\{\mathcal C_1,\ldots,\mathcal C_s\}.
\end{equation*}
Here $s$ is the number of cycles of $\sigma$. The unions of the blocks in
distinct cycles are disjoint. Admissibility ensures that each such union
contains a $+$ index, so $1\leq s\leq k$. To apply Davenport's theorem
directly, we first omit
condition \textup{(ii)} in
Lemma~\ref{lem:exact-admissible-block-decomposition}. A triple satisfying
the remaining requirements is called a \emph{relaxed configuration}. We
now choose independent integer coordinates for these configurations. We shall show that their possible values are precisely the
lattice points of a convex polytope. The next
estimate shows that the excluded coincidences contribute only to the lower
order error term.

We refer to each $p_i$ or $q_j$ as a composition entry. To separate the composition entries according to their cycles, for
$1\leq a\leq s$ set
\begin{equation*}
 P_a:=\{i:i^+\in B\text{ for some }B\in\mathcal C_a\},
 \qquad
 Q_a:=\{j:j^-\in B\text{ for some }B\in\mathcal C_a\}.
\end{equation*}
The definition of $\mathfrak B_k$ in
\eqref{eq:block-permutation-set-lattice} shows that $P_a$ and $Q_a$ are
non-empty. The cycles partition the blocks of $\pi$, so
$P_1,\ldots,P_s$ and $Q_1,\ldots,Q_s$ each form a partition of
$\{1,\ldots,k\}$.

To select one position coordinate and one omitted entry from each
composition in every cycle, choose a reference block
$B_a^0\in\mathcal C_a$, together with indices $i_a\in P_a$ and
$j_a\in Q_a$, for each $a\in\{1,\ldots,s\}$. The superscript $0$
identifies the reference block, and we call $x_{B_a^0}$ the reference
position. The entries $p_{i_a}$ and $q_{j_a}$ will be recovered from the
total assigned to $\mathcal C_a$. To determine this total, sum
\eqref{eq:finite-block-equation-body} over the blocks in $\mathcal C_a$.
The position differences cancel around the cycle, giving
\begin{equation*}
 \sum_{i\in P_a}p_i
 =
 \sum_{j\in Q_a}q_j
 =:t_a.
\end{equation*}
We call $t_a$ the \emph{cycle total}. It is the total of the entries $p_i$
assigned to $\mathcal C_a$, which equals the total of the entries $q_j$
assigned to the same cycle.
Because $P_1,\ldots,P_s$ partition $\{1,\ldots,k\}$,
\begin{equation*}
 t_1+\cdots+t_s=\sum_{i=1}^k p_i=n,
\end{equation*}
so only $t_1,\ldots,t_{s-1}$ are free. A relaxed configuration is therefore
determined by
\begin{equation}
\label{eq:relaxed-free-coordinates}
 t_1,\ldots,t_{s-1},
 \qquad
 (p_i)_{i\notin\{i_1,\ldots,i_s\}},
 \qquad
 (q_j)_{j\notin\{j_1,\ldots,j_s\}},
 \qquad
 (x_{B_a^0})_{1\leq a\leq s}.
\end{equation}
There are
\begin{equation*}
 (s-1)+(k-s)+(k-s)+s=2k-1
\end{equation*}
integer variables. For fixed $n$, an affine function of these variables is
a constant plus a linear combination of them. It is integral affine when
the constant and every coefficient are integers. The final cycle total and
the omitted composition entries are recovered from
\begin{equation}
\label{eq:relaxed-coordinate-recovery}
\begin{aligned}
 t_s&=n-\sum_{a<s}t_a, \qquad p_{i_a}&=t_a-\sum_{i\in P_a\setminus\{i_a\}}p_i,
 \qquad
 q_{j_a}=t_a-\sum_{j\in Q_a\setminus\{j_a\}}q_j.
\end{aligned}
\end{equation}

Once the reference position $x_{B_a^0}$ is fixed, the block equations
determine the remaining positions in $\mathcal C_a$. To express this
dependence, we repeat the partial displacement notation from
\eqref{eq:D_CJ_def}. Let
\begin{equation*}
\mathcal C=(B_1,\ldots,B_\ell)\in\operatorname{Cyc}(\sigma),
\qquad
\sigma(B_j)=B_{j+1}\quad(1\leq j\leq\ell),
\qquad
B_{\ell+1}:=B_1.
\end{equation*}
When $\mathcal C=\mathcal C_a$, choose the initial block $B_1:=B_a^0$.
Define the partial displacements of $\mathcal C$ by
\begin{equation}
\label{eq:D-cycle-lattice}
 D_{\mathcal C,0}(\mathbf u,\mathbf v):=0,
 \qquad
 D_{\mathcal C,j}(\mathbf u,\mathbf v)
 :=
 \sum_{b=1}^{j}d_{B_b}(\mathbf u,\mathbf v),
 \qquad 1\leq j\leq\ell.
\end{equation}
The quantity $D_{\mathcal C,j}(\mathbf u,\mathbf v)$ is the cumulative
displacement after the first $j$ blocks of $\mathcal C$.
For $(\mathbf u,\mathbf v)=(\mathbf p,\mathbf q)$, iterating
\eqref{eq:finite-block-equation-body} gives
\begin{equation}
\label{eq:cycle-position-recovery}
 x_{B_{j+1}}=x_{B_1}-D_{\mathcal C,j}(\mathbf p,\mathbf q)
 \qquad(0\leq j<\ell),
\end{equation}
and returning to $B_1$ gives the cycle balance condition
$D_{\mathcal C,\ell}(\mathbf p,\mathbf q)=0$.

For the scaled count, the range of the partial displacements controls the
available interval for the reference position. This range is the cycle
width introduced in \eqref{eq:cycle-width}:
\begin{equation}
\label{eq:cycle-width-lattice}
 \omega_{\mathcal C}(\mathbf u,\mathbf v)
 :=
 \max_{0\leq j<\ell}D_{\mathcal C,j}(\mathbf u,\mathbf v)
 -
 \min_{0\leq j<\ell}D_{\mathcal C,j}(\mathbf u,\mathbf v).
\end{equation}
To combine the cycle balance condition and the length of this interval in
one factor, use the kernel introduced in \eqref{eq:cycle-kernel}:
\begin{equation}
\label{eq:cycle-kernel-lattice}
 \mathcal W_{\mathcal C}(\mathbf u,\mathbf v)
 :=
 \delta\!\left(D_{\mathcal C,\ell}(\mathbf u,\mathbf v)\right)
 \bigl(1-\omega_{\mathcal C}(\mathbf u,\mathbf v)\bigr)_+.
\end{equation}
When $D_{\mathcal C,\ell}(\mathbf u,\mathbf v)=0$, changing the initial
block reorders the partial displacements cyclically and translates them all
by a common amount. Thus $\omega_{\mathcal C}$ and
$\mathcal W_{\mathcal C}$ depend only on the directed cycle $\mathcal C$.

Equation~\eqref{eq:relaxed-coordinate-recovery} expresses the final cycle
total $t_s$ and the omitted entries $p_{i_a}$ and $q_{j_a}$ as integral
affine functions of the variables in
\eqref{eq:relaxed-free-coordinates}. Equation~\eqref{eq:cycle-position-recovery}
does the same for every block position other than the reference position
of its cycle. In each expression, every linear coefficient belongs to
$\{0,1,-1\}$.

The requirements $p_i,q_j\geq1$ and $0\leq x_B\leq N-1$ therefore become
finitely many affine inequalities in the chosen $2k-1$ variables. Their common
solution set in $\mathbb R^{2k-1}$ is therefore convex. Every free composition entry and cycle total lies
between $1$ and $n$, while every reference position lies in $[0, N-1]$, so
the solution set is bounded. It is therefore a convex polytope, and its
lattice points in the standard lattice $\mathbb Z^{2k-1}$ are exactly the
relaxed configurations. If $n=\lfloor cN\rfloor$ and $c$ ranges over a
fixed compact set $\mathcal K\subset(0,\infty)$, each chosen variable is
$O_{k,\mathcal K}(N)$. Distinct blocks must represent distinct positions. The next estimate shows
that relaxed configurations with two distinct blocks at the same position
contribute only to the error term.

\begin{lemma}
\label{lem:coincident-blocks-lower-order}
Fix an integer $k\geq1$, an admissible pair
$(\pi,\sigma)\in\mathfrak B_k$, and a compact set
$\mathcal K\subset(0,\infty)$. Uniformly for $c\in\mathcal K$, the number
of relaxed configurations with $n=\lfloor cN\rfloor$ and
$x_B=x_C$ for at least one pair of distinct blocks is
$O_{k,\mathcal K}(N^{2k-2})$.
\end{lemma}

\begin{proof}
Use the $2k-1$ integer variables chosen above. Fix distinct blocks $B$ and
$C$. We shall show that the condition $x_B=x_C$ is either impossible or is
given by one affine equation with at least one non-zero coefficient in the
chosen variables.

Suppose first that $B$ and $C$ lie in different cycles. Iterating the block
equations in \eqref{eq:finite-block-equation-body} expresses each block
position as the reference position for its cycle minus a partial
displacement. The equation $x_B=x_C$ has coefficient $1$ in the first
reference position and coefficient $-1$ in the second. It is therefore an
affine equation with a non-zero coefficient in the chosen variables.

Now suppose that $B$ and $C$ lie in the same directed cycle
$\mathcal C_a$. Let $\mathcal R$ be the set of blocks encountered when
moving from $B$ in the direction of $\sigma$, up to but not including $C$.
This is a non-empty proper subset of $\mathcal C_a$. To distinguish the
labels in this segment from all labels in the cycle, define
\begin{equation*}
 \Omega_{\mathcal R}:=\bigcup_{D\in\mathcal R}D,
 \qquad
 \Omega_a:=\bigcup_{D\in\mathcal C_a}D.
\end{equation*}
Adding \eqref{eq:finite-block-equation-body} over the blocks in
$\mathcal R$ gives
\begin{equation}
\label{eq:coincident-block-directed-segment}
 x_B-x_C
 =
 H_{\mathcal R}(\mathbf p,\mathbf q),
 \qquad
 H_{\mathcal R}(\mathbf p,\mathbf q)
 :=
 \sum_{D\in\mathcal R}d_D(\mathbf p,\mathbf q).
\end{equation}
Thus $x_B=x_C$ is equivalent to
$H_{\mathcal R}(\mathbf p,\mathbf q)=0$.

Assume that there are positive labels
$i^+\in\Omega_{\mathcal R}$ and
$(i')^+\in\Omega_a\setminus\Omega_{\mathcal R}$. For $z\in\mathbb R$,
replace $p_i$ by $p_i+z$ and $p_{i'}$ by $p_{i'}-z$. The cycle total
$t_a$ is unchanged, while $H_{\mathcal R}$ changes by $z$. The same
argument applies if there are negative labels
$j^-\in\Omega_{\mathcal R}$ and
$(j')^-\in\Omega_a\setminus\Omega_{\mathcal R}$. Replacing $q_j$ by
$q_j+z$ and $q_{j'}$ by $q_{j'}-z$ preserves $t_a$ and changes
$H_{\mathcal R}$ by $-z$. These replacements are used only to test whether
the expression in \eqref{eq:coincident-block-directed-segment} is constant
on the affine space defined by the equality
constraints for the relaxed configurations. The replaced values
need not satisfy the positivity and position inequalities. In either case,
$H_{\mathcal R}=0$ is an affine equation with a non-zero coefficient in
the chosen variables.

It remains to consider the case in which neither sign occurs in both
$\Omega_{\mathcal R}$ and
$\Omega_a\setminus\Omega_{\mathcal R}$. Both sets contain a signed label
because every block is non-empty. Admissibility ensures that $\Omega_a$
contains labels of both signs. The two sets must therefore contain opposite
signs. If $\Omega_{\mathcal R}$ contains the positive labels, then
$H_{\mathcal R}=t_a$. If it contains the negative labels, then
$H_{\mathcal R}=-t_a$. Every composition entry is positive, so $t_a>0$ and
the equation $H_{\mathcal R}=0$ has no solutions in either case.

For each pair of distinct blocks, the coincidence condition is now either
impossible or one affine equation with a non-zero coefficient in the
$2k-1$ free integer variables. Choose a variable with a non-zero
coefficient in this equation. Each of the other $2k-2$ variables has
$O_{k,\mathcal K}(N)$ possible integer values. Once these variables are
fixed, the equation determines at most one value of the chosen variable.
There are at most $\binom{2k}{2}$ pairs of distinct blocks, which gives the
stated bound.
\end{proof}

To identify the leading volume associated with an admissible pair, we now
introduce a continuous integral. The parameter $t$ is the common total of
the two compositions, and $M$ is the length of the interval available to
the block positions. For a cycle
$\mathcal C=(B_1,\ldots,B_\ell)$, write
$|\mathcal C|:=\ell$ for its number of blocks. By
\eqref{eq:D-cycle-lattice},
$D_{\mathcal C,|\mathcal C|}(\mathbf u,\mathbf v)$ is the total displacement
around the cycle:
\begin{equation}
\label{eq:total-cycle-displacement}
 D_{\mathcal C,|\mathcal C|}(\mathbf u,\mathbf v)
 =
 \sum_{j=1}^{\ell}d_{B_j}(\mathbf u,\mathbf v).
\end{equation}
The cycle balance condition means that the quantity in
\eqref{eq:total-cycle-displacement} is zero.

For $(\pi,\sigma)\in\mathfrak B_k$ and $t,M>0$, define
\begin{equation}
\label{eq:J-pi-sigma-definition}
 \mathcal J_{\pi,\sigma}(t;M)
 :=
 \int_{\mathbb R_{>0}^{2k}}
 \delta\!\left(\sum_{i=1}^k u_i-t\right)
 \prod_{\mathcal C\in\operatorname{Cyc}(\sigma)}
 \delta\!\left(
 D_{\mathcal C,|\mathcal C|}(\mathbf u,\mathbf v)
 \right)
 \bigl(M-\omega_{\mathcal C}(\mathbf u,\mathbf v)\bigr)_+
 \,d\mathbf u\,d\mathbf v.
\end{equation}
The first Dirac delta restricts the composition total to $t$. For each
cycle, the second Dirac delta enforces the cycle balance condition. On this
balance set, the positive part is the length of the interval available to
the reference position. Equation~\eqref{eq:cycle-kernel-lattice} shows that
$\mathcal J_{\pi,\sigma}(c;1)$ is the integral associated with
$(\pi,\sigma)$ in Theorem~\ref{thm:delta-cycle-formula}.

After division by $N$, each composition entry is at least $1/N$, the common
composition total is $\lfloor cN\rfloor/N$, and the block positions lie in
$[0,(N-1)/N]$. The next lemma gives uniform estimates for replacing these
three quantities by $0$, $c$, and $1$.

\begin{lemma}
\label{lem:cycle-integral-stability}
Fix an integer $k\geq1$, an admissible pair
$(\pi,\sigma)\in\mathfrak B_k$, and a compact interval
$\mathcal K\subset(0,\infty)$. Let
$\mathcal J_{\pi,\sigma}^{(\varepsilon)}(t;M)$ denote the integral in
\eqref{eq:J-pi-sigma-definition} with the additional restrictions
$u_i,v_j\geq\varepsilon$. Uniformly for $t,t'\in\mathcal K$,
$M,M'\in[1/2,3/2]$, and all sufficiently small $\varepsilon\geq0$,
\begin{align*}
 \left|
 \mathcal J_{\pi,\sigma}(t;M)
 -\mathcal J_{\pi,\sigma}(t';M')
 \right|
 &\ll_{k,\mathcal K}
 |t-t'|+|M-M'|,\\
 \mathcal J_{\pi,\sigma}^{(\varepsilon)}(t;M)
 &=
 \mathcal J_{\pi,\sigma}(t;M)
 +O_{k,\mathcal K}(\varepsilon).
\end{align*}
\end{lemma}

\begin{proof}
Write
$\operatorname{Cyc}(\sigma)=\{\mathcal C_1,\ldots,\mathcal C_s\}$, and
retain the sets $P_a,Q_a$ and the indices $i_a,j_a$ chosen above. To
separate the scale $t$ from the proportions assigned to the composition
entries, define the \emph{normalised balance polytope}
\begin{equation}
\label{eq:normalised-balance-polytope}
 \mathcal S_{\pi,\sigma}
 :=
 \left\{
 (\boldsymbol\alpha,\boldsymbol\beta)\in\mathbb R_{\geq0}^{2k}:
 \begin{array}{l}
 \displaystyle\sum_{i=1}^k\alpha_i=1,\\[0.4em]
 D_{\mathcal C,|\mathcal C|}
 (\boldsymbol\alpha,\boldsymbol\beta)=0
 \quad\text{for every }\mathcal C\in\operatorname{Cyc}(\sigma)
 \end{array}
 \right\}.
\end{equation}
Thus $\mathcal S_{\pi,\sigma}$ consists of non-negative composition
proportions for which $\sum_i\alpha_i=1$ and every cycle is balanced. Adding
the cycle balance equations gives
\begin{equation*}
 \sum_{i=1}^k\alpha_i-
 \sum_{j=1}^k\beta_j=0.
\end{equation*}
The first equation in \eqref{eq:normalised-balance-polytope} then gives
$\sum_j\beta_j=1$. Both coordinate sums are fixed, so
$\mathcal S_{\pi,\sigma}$ is bounded. We shall also need a point of this polytope at which every coordinate is
positive. Choose $\rho_1,\ldots,\rho_s>0$ with
$\rho_1+\cdots+\rho_s=1$. For each $a$, distribute $\rho_a$ among the
coordinates $\alpha_i$ with $i\in P_a$, using positive values, and
distribute the same amount among the coordinates $\beta_j$ with
$j\in Q_a$. The sets $P_a$ and $Q_a$ are non-empty by admissibility, so
this defines a point of $\mathcal S_{\pi,\sigma}$ with every coordinate
positive.

To describe the measure on the polytope in
\eqref{eq:normalised-balance-polytope}, first write its balance equation for
$\mathcal C_a$ explicitly. By \eqref{eq:total-cycle-displacement}, this
equation is
\begin{equation*}
D_{\mathcal C_a,|\mathcal C_a|}
(\boldsymbol\alpha,\boldsymbol\beta)
=
\sum_{i\in P_a}\alpha_i
-
\sum_{j\in Q_a}\beta_j
=
0.
\end{equation*}
Use this equation to eliminate $\beta_{j_a}$ for each $1\leq a\leq s$,
and use $\sum_i\alpha_i=1$ to eliminate $\alpha_{i_s}$. Order the equations
by the $s$ cycle balance equations and then the total equation, and order
the eliminated variables as
\begin{equation*}
\beta_{j_1},\ldots,\beta_{j_s},\alpha_{i_s}.
\end{equation*}
Each $\beta_{j_a}$ occurs in its own cycle balance equation with coefficient
$-1$ and in no other equation. The variable $\alpha_{i_s}$ occurs with
coefficient $1$ in the balance equation for $\mathcal C_s$ and in the
total equation. The resulting coefficient matrix is therefore triangular,
with diagonal entries $-1,\ldots,-1,1$. Its determinant is $(-1)^s$, whose
absolute value is $1$. Thus the equations are independent, and eliminating
them introduces no additional factor into the measure. They leave
$2k-s-1$ free coordinates, and we denote ordinary Lebesgue measure in
these coordinates by $d\mu_{\pi,\sigma}$. When $2k-s-1=0$, this measure assigns mass $1$ to the single point.

The integral in \eqref{eq:J-pi-sigma-definition} uses strictly positive
composition variables, whereas
\eqref{eq:normalised-balance-polytope} includes points with zero
coordinates. The positive point constructed above shows that no coordinate
function $\alpha_i$ or $\beta_j$ vanishes identically on the polytope. If
one of these coordinate functions is not constant, its zero set is given by
one affine equation with a non-zero coefficient in the free coordinates.
Such a set has measure zero with respect to $d\mu_{\pi,\sigma}$. If the
coordinate function is constant, it is positive and its zero set is empty.
We may therefore integrate over the closed polytope without changing the
value. The same observation will estimate the parts near the boundary that
are removed by a positive lower bound.

Set $\mathbf u=t\boldsymbol\alpha$ and
$\mathbf v=t\boldsymbol\beta$. The $2k$ differentials contribute
$t^{2k}$. The constraint on the composition total contributes $t^{-1}$,
because
\begin{equation*}
 \delta\!\left(\sum_i u_i-t\right)
 =
 t^{-1}\delta\!\left(\sum_i\alpha_i-1\right).
\end{equation*}
Each of the $s$ cycle balance constraints contributes another factor
$t^{-1}$. The linearity of the block displacements in
\eqref{eq:block-displacement-lattice} and the definition of the cycle width
in \eqref{eq:cycle-width-lattice} give
\begin{equation*}
 \omega_{\mathcal C}
 (t\boldsymbol\alpha,t\boldsymbol\beta)
 =
 t\omega_{\mathcal C}
 (\boldsymbol\alpha,\boldsymbol\beta).
\end{equation*}
After the $s+1$ constraints are taken into account, the remaining power of
$t$ is $t^{2k-s-1}$. Therefore
\begin{equation}
\label{eq:J-fixed-polytope}
 \mathcal J_{\pi,\sigma}(t;M)
 =
 t^{2k-s-1}
 \int_{\mathcal S_{\pi,\sigma}}
 \prod_{\mathcal C\in\operatorname{Cyc}(\sigma)}
 \bigl(M-t\omega_{\mathcal C}
 (\boldsymbol\alpha,\boldsymbol\beta)\bigr)_+
 \,d\mu_{\pi,\sigma}.
\end{equation}
We first compare the parameters $(t,M)$ and $(t',M')$. Set
\begin{align*}
 A_a
 :={}
 \bigl(M-t\omega_{\mathcal C_a}
 (\boldsymbol\alpha,\boldsymbol\beta)\bigr)_+, \qquad
 A'_a
 :={}
 \bigl(M'-t'\omega_{\mathcal C_a}
 (\boldsymbol\alpha,\boldsymbol\beta)\bigr)_+.
\end{align*}
The cycle widths are bounded on the fixed polytope
$\mathcal S_{\pi,\sigma}$. The parameter ranges are compact, so the
factors $A_a$ and $A'_a$ are bounded uniformly. The inequality
$|x_+-y_+|\leq|x-y|$, where $x_+:=\max\{x,0\}$, gives
\begin{equation*}
 |A_a-A'_a|
 \ll_{k,\mathcal K}
 |t-t'|+|M-M'|.
\end{equation*}
The identity
\begin{equation*}
 \prod_{a=1}^s A_a-
 \prod_{a=1}^s A'_a
 =
 \sum_{a=1}^s
 (A_a-A'_a)
 \prod_{b<a}A'_b
 \prod_{b>a}A_b
\end{equation*}
then gives
\begin{equation*}
 \left|
 \prod_{a=1}^s A_a-
 \prod_{a=1}^s A'_a
 \right|
 \ll_{k,\mathcal K}
 |t-t'|+|M-M'|.
\end{equation*}
The mean value theorem also gives
\begin{equation*}
 \left|t^{2k-s-1}-(t')^{2k-s-1}\right|
 \ll_{k,\mathcal K}|t-t'|.
\end{equation*}
Applying these estimates to \eqref{eq:J-fixed-polytope} proves the first
claim of the lemma.

We now estimate the effect of the positive lower bound. The same change of
variables gives
\begin{align*}
 \mathcal J_{\pi,\sigma}^{(\varepsilon)}(t;M)
 ={}&
 t^{2k-s-1}
 \int_{\mathcal S_{\pi,\sigma}}
 \mathds{1}_{\{\alpha_i,\beta_j\geq\varepsilon/t
 \text{ for every }i,j\}}
 \prod_{\mathcal C\in\operatorname{Cyc}(\sigma)}
 \bigl(M-t\omega_{\mathcal C}
 (\boldsymbol\alpha,\boldsymbol\beta)\bigr)_+
 \,d\mu_{\pi,\sigma}.
\end{align*}
Put $m_{\mathcal K}:=\min\mathcal K>0$. The lower bound in
$\mathcal J_{\pi,\sigma}^{(\varepsilon)}(t;M)$ restricts the integral over
$\mathcal S_{\pi,\sigma}$ to points satisfying
\begin{equation*}
\alpha_i,\beta_j\geq\frac{\varepsilon}{t}
\qquad\text{for every }i,j.
\end{equation*}
A point omitted by this restriction therefore has at least one coordinate
$\alpha_i$ or $\beta_j$ in the interval
\begin{equation*}
0\leq\alpha_i\leq\frac{\varepsilon}{t}
\qquad\text{or}\qquad
0\leq\beta_j\leq\frac{\varepsilon}{t}.
\end{equation*}
As $t\geq m_{\mathcal K}$, the selected coordinate is at most
$\varepsilon/m_{\mathcal K}$.

By a coordinate function, we mean one of the maps
\begin{equation*}
(\boldsymbol\alpha,\boldsymbol\beta)\longmapsto\alpha_i,
\qquad
(\boldsymbol\alpha,\boldsymbol\beta)\longmapsto\beta_j.
\end{equation*}
Denote any one of these functions by $L$. Write
$d:=2k-s-1$ and let
$\boldsymbol\xi=(\xi_1,\ldots,\xi_d)$ be the free coordinates used to
define $d\mu_{\pi,\sigma}$. Eliminating the remaining coordinates expresses
$L$ as an affine function
\begin{equation*}
L(\boldsymbol\xi)
=
b+\sum_{r=1}^{d}c_r\xi_r.
\end{equation*}
The positive point constructed above shows that $L$ is not identically
zero. If $c_1=\cdots=c_d=0$, then $L=b>0$, so
$0\leq L\leq\varepsilon/m_{\mathcal K}$ has no solutions for sufficiently
small $\varepsilon$.

Now suppose that $L$ is not constant. Choose $r$ with $c_r\neq0$, and set
\begin{equation*}
z:=\xi_r,
\qquad
\mathbf w:=
(\xi_1,\ldots,\xi_{r-1},\xi_{r+1},\ldots,\xi_d),
\qquad
a:=c_r.
\end{equation*}
Then
\begin{equation*}
L(z,\mathbf w)
=
az+L_0(\mathbf w),
\qquad
L_0(\mathbf w)
:=
b+\sum_{h\neq r}c_h\xi_h.
\end{equation*}
For each fixed $\mathbf w$, the inequalities
\begin{equation*}
0\leq az+L_0(\mathbf w)
\leq\frac{\varepsilon}{m_{\mathcal K}}
\end{equation*}
restrict $z$ to an interval of length at most
\begin{equation*}
\frac{\varepsilon}{m_{\mathcal K}|a|}.
\end{equation*}
The set of possible values of $\mathbf w$ has bounded volume. Fubini's
theorem therefore gives
\begin{equation*}
 \int_{\mathcal S_{\pi,\sigma}}
 \mathds{1}_{\{0\leq L\leq\varepsilon/m_{\mathcal K}\}}
 \,d\mu_{\pi,\sigma}
 =O_{k,\mathcal K}(\varepsilon).
\end{equation*}
There are $2k$ coordinate functions, so the removed part has volume
$O_{k,\mathcal K}(\varepsilon)$. The factor outside the integral and the
integrand in \eqref{eq:J-fixed-polytope} are bounded uniformly for the
stated parameter ranges. The removed part therefore contributes
$O_{k,\mathcal K}(\varepsilon)$, which proves the second claim.
\end{proof}

\begin{proof}[Proof of Theorem~\ref{thm:delta-cycle-formula}]
Fix $k\in\mathbb Z_{\geq1}$ and $\mathcal K=[c^-,c^+]\subset(0,\infty)$. It is enough to prove the asymptotic estimate uniformly on compact
intervals, because every compact subset of $(0,\infty)$ is contained in
such an interval. For
$c\in\mathcal K$, set
\begin{equation*}
n:=\lfloor cN\rfloor,
\qquad
c_N:=\frac{n}{N},
\qquad
M_N:=\frac{N-1}{N}.
\end{equation*}
The value $c_N$ is the scaled composition total, while $M_N$ is the upper
endpoint for the scaled block positions. To place $c$ and $c_N$ in one
fixed interval, define
\begin{equation*}
\mathcal K_*:=[c^-/2,c^++1].
\end{equation*}
The lower bound $c\geq c^->0$ ensures that $n\geq1$ uniformly for
$c\in\mathcal K$ once $N$ is sufficiently large. For such $N$, we also
have $c,c_N\in\mathcal K_*$ and $M_N,1\in[1/2,3/2]$. We work with such
$N$ throughout. Constants depending on $\mathcal K_*$ may therefore be
included in the notation $O_{k,\mathcal K}$.

Fix an admissible pair $(\pi,\sigma)\in\mathfrak B_k$. Let
$\tilde{\mathcal C}_N(\pi,\sigma;n)$ denote the relaxed configurations
obtained from $\mathcal C_N(\pi,\sigma;n)$ by omitting condition
\textup{(ii)} in
Lemma~\ref{lem:exact-admissible-block-decomposition}. The $2k-1$ integer
coordinates chosen above identify this set with the points of
$\mathbb Z^{2k-1}$ in a convex polytope $\mathcal P_N$. Each cycle total
and each composition coordinate is bounded by $n$, while each reference
position lies in $[0,N-1]$. Thus $\mathcal P_N$ is contained in a closed
ball of radius
\begin{equation*}
 \rho=O_{k,\mathcal K}(N).
\end{equation*}
Applying Theorem~\ref{thm:davenport} with $D=2k-1$ gives
\begin{equation}
\label{eq:relaxed-count-volume}
 \#\tilde{\mathcal C}_N(\pi,\sigma;n)
 =
 \operatorname{vol}_{2k-1}(\mathcal P_N)
 +O_{k,\mathcal K}(N^{2k-2}).
\end{equation}
This remains valid when $\mathcal P_N$ is empty or lower-dimensional,
in which case $\operatorname{vol}_{2k-1}(\mathcal P_N)=0$.
The term $1$ in \eqref{eq:davenport} is absorbed into the error term,
including when $k=1$. The relaxed count in \eqref{eq:relaxed-count-volume} includes
configurations in which two distinct blocks occupy the same position.
Lemma~\ref{lem:coincident-blocks-lower-order} bounds their number by
$O_{k,\mathcal K}(N^{2k-2})$. Removing these configurations gives
\begin{equation}
\label{eq:exact-count-volume}
 \#\mathcal C_N(\pi,\sigma;n)
 =
 \operatorname{vol}_{2k-1}(\mathcal P_N)
 +O_{k,\mathcal K}(N^{2k-2}).
\end{equation}
We next scale the polytope so that its variables remain bounded as
$N\to\infty$. Define
\begin{equation*}
 u_i:=\frac{p_i}{N},
 \qquad
 v_j:=\frac{q_j}{N},
 \qquad
 y_B:=\frac{x_B}{N},
 \qquad
 \tau_a:=\frac{t_a}{N}.
\end{equation*}
The variable $\tau_a$ is the scaled cycle total. Dividing each of the
$2k-1$ free coordinates by $N$ defines a scaled polytope
$\hat{\mathcal P}_N$. Each coordinate contributes one factor of $N$ when
the original volume is recovered, so
\begin{equation}
\label{eq:scaled-polytope-volume}
 \operatorname{vol}_{2k-1}(\mathcal P_N)
 =
 N^{2k-1}
 \operatorname{vol}_{2k-1}(\hat{\mathcal P}_N).
\end{equation}
The scaled composition entries and block positions satisfy
\begin{equation}
\label{eq:scaled-block-conditions}
 u_i,v_j\geq\frac1N,
 \qquad
 \sum_{i=1}^k u_i=c_N,
 \qquad
 0\leq y_B\leq M_N,
 \qquad
 d_B(\mathbf u,\mathbf v)=y_B-y_{\sigma(B)}.
\end{equation}
We use Dirac delta factors to impose the composition total and the cycle
balance conditions. For each cycle $\mathcal C_a$, eliminate $v_{j_a}$
using its balance equation, and eliminate $u_{i_s}$ using
$\sum_i u_i=c_N$. With the cycle equations listed first, the coefficient
matrix for these variables is triangular with diagonal entries
$-1,\ldots,-1,1$. Its determinant has absolute value $1$, so integration
against the Dirac delta factors leaves ordinary Lebesgue measure in the
remaining composition coordinates.

The free coordinates used to define $\hat{\mathcal P}_N$ contain
$\tau_1,\ldots,\tau_{s-1}$ rather than
$u_{i_1},\ldots,u_{i_{s-1}}$. For $a<s$, the definition of the cycle total
gives
\begin{equation*}
 u_{i_a}
 =
 \tau_a-
 \sum_{i\in P_a\setminus\{i_a\}}u_i.
\end{equation*}
Replacing $u_{i_a}$ by $\tau_a$ for every $a<s$ is a triangular change of
coordinates with determinant $1$. The measure obtained after applying the
Dirac delta factors therefore agrees with the measure in the scaled free
coordinates used in \eqref{eq:scaled-polytope-volume}.

Next, we integrate over the block positions. For a cycle
$\mathcal C_a=(B_1,\ldots,B_\ell)$, choose the ordering with
$B_1=B_a^0$, the reference block selected above. Iterating the scaled block
equations in \eqref{eq:scaled-block-conditions} and using the partial
displacements from \eqref{eq:D-cycle-lattice} gives
\begin{equation*}
 y_{B_{j+1}}
 =
 y_{B_1}-D_{\mathcal C_a,j}(\mathbf u,\mathbf v),
 \qquad 0\leq j<\ell.
\end{equation*}
Returning to the reference block imposes the balance condition
\begin{equation*}
 D_{\mathcal C_a,\ell}(\mathbf u,\mathbf v)=0.
\end{equation*}
Under this condition, all positions in the cycle lie in $[0,M_N]$ exactly
when
\begin{equation}
\label{eq:scaled-reference-position-interval}
 y_{B_1}
 \in
 \bigcap_{j=0}^{\ell-1}
 \left[
 D_{\mathcal C_a,j}(\mathbf u,\mathbf v),
 M_N+D_{\mathcal C_a,j}(\mathbf u,\mathbf v)
 \right].
\end{equation}
By the definition of the cycle width in
\eqref{eq:cycle-width-lattice}, the interval in
\eqref{eq:scaled-reference-position-interval} has length
\begin{equation*}
 \bigl(M_N-
 \omega_{\mathcal C_a}(\mathbf u,\mathbf v)\bigr)_+.
\end{equation*}
Using the first $\ell-1$ block equations in cycle order to eliminate
$y_{B_2},\ldots,y_{B_\ell}$ gives a triangular coefficient matrix with
diagonal entries $-1$. Its determinant has absolute value $1$, so no
further factor appears. Integrating the reference position in each cycle
and multiplying the interval lengths gives
\begin{equation}
\label{eq:scaled-volume-cycle-integral}
 \operatorname{vol}_{2k-1}(\hat{\mathcal P}_N)
 =
 \mathcal J_{\pi,\sigma}^{(1/N)}(c_N;M_N).
\end{equation}

Equation \eqref{eq:scaled-volume-cycle-integral} contains three quantities
that depend on $N$: the lower bound $1/N$, the composition total $c_N$,
and the position interval length $M_N$.
Lemma~\ref{lem:cycle-integral-stability} compares them with $0$, $c$, and $1$.
Apply the lemma on $\mathcal K_*$ with
\begin{equation*}
 t=c_N,
 \qquad
 t'=c,
 \qquad
 M=M_N,
 \qquad
 M'=1,
 \qquad
 \varepsilon=\frac1N.
\end{equation*}
The estimate for the lower bound gives
\begin{equation*}
 \mathcal J_{\pi,\sigma}^{(1/N)}(c_N;M_N)
 =
 \mathcal J_{\pi,\sigma}(c_N;M_N)
 +O_{k,\mathcal K}(N^{-1}),
\end{equation*}
while the parameter estimate gives
\begin{align*}
 \left|
 \mathcal J_{\pi,\sigma}(c_N;M_N)
 -\mathcal J_{\pi,\sigma}(c;1)
 \right|
 &\ll_{k,\mathcal K}
 |c_N-c|+|M_N-1|\\
 &\ll_{k,\mathcal K}N^{-1}.
\end{align*}
Thus
\begin{equation}
\label{eq:finite-cycle-integral-limit}
 \mathcal J_{\pi,\sigma}^{(1/N)}(c_N;M_N)
 =
 \mathcal J_{\pi,\sigma}(c;1)
 +O_{k,\mathcal K}(N^{-1}).
\end{equation}
Substituting \eqref{eq:scaled-polytope-volume} into
\eqref{eq:exact-count-volume} gives
\begin{equation}
\label{eq:scaled-count-volume}
 \#\mathcal C_N(\pi,\sigma;n)
 =
 N^{2k-1}
 \operatorname{vol}_{2k-1}(\hat{\mathcal P}_N)
 +O_{k,\mathcal K}(N^{2k-2}).
\end{equation}
Equation \eqref{eq:scaled-volume-cycle-integral} identifies the volume in
\eqref{eq:scaled-count-volume}. The approximation in
\eqref{eq:finite-cycle-integral-limit} changes the resulting leading term
by $O_{k,\mathcal K}(N^{2k-2})$. We obtain
\begin{equation}
\label{eq:uniform-block-cycle-asymptotic}
 \#\mathcal C_N(\pi,\sigma;\lfloor cN\rfloor)
 =
 N^{2k-1}\mathcal J_{\pi,\sigma}(c;1)
 +O_{k,\mathcal K}(N^{2k-2}),
\end{equation}
uniformly for $c\in\mathcal K$. Insert \eqref{eq:uniform-block-cycle-asymptotic} into
\eqref{eq:exact-admissible-block-decomposition} and divide by
$N^{2k-1}$. The set $\mathfrak B_k$ is finite, with size depending only
on $k$, so summing the error terms gives, uniformly for
$c\in\mathcal K$,
\begin{equation*}
\frac{I_k(\lfloor cN\rfloor;N)}{N^{2k-1}}
=
\sum_{(\pi,\sigma)\in\mathfrak B_k}
\operatorname{sgn}(\sigma)\,
\mathcal J_{\pi,\sigma}(c;1)
+
O_{k,\mathcal K}(N^{-1}).
\end{equation*}
Therefore, by \eqref{eq:rmt-delta-final}, we have
\begin{equation}
\label{eq:delta-admissible-sum}
\delta_k(c)
=
\sum_{(\pi,\sigma)\in\mathfrak B_k}
\operatorname{sgn}(\sigma)\,
\mathcal J_{\pi,\sigma}(c;1).
\end{equation}

It remains to express \eqref{eq:delta-admissible-sum} in the form stated in
\eqref{eq:delta-cycle-final}. By \eqref{eq:J-pi-sigma-definition}, each
summand is the integral associated with $(\pi,\sigma)$. We may add the pairs
in $\mathfrak J_k\setminus\mathfrak B_k$ without changing the value. For
any such pair, the union of the blocks in some cycle $\mathcal C$ contains
labels of only one sign. Its total displacement
$D_{\mathcal C,|\mathcal C|}(\mathbf u,\mathbf v)$ is either a non-empty sum
of positive $u_i$ or the negative of a non-empty sum of positive $v_j$.
The balance condition has no solution in $\mathbb R_{>0}^{2k}$, so the
integral associated with the pair is zero. Adding these pairs extends the
sum to $\mathfrak J_k$ and gives \eqref{eq:delta-cycle-final}.
As the interval $\mathcal K$ was arbitrary, the identity holds for every
$c>0$ and all error estimates are uniform on compact subsets of
$(0,\infty)$.
\end{proof}

\section{Finite-spline and Lauricella structure}
\label{section:finite-spline-structure}

The formulas in Theorem~\ref{thm:lauricella} are obtained by resolving the
volume integral of Theorem~\ref{thm:leading-volume} into its monomial
contributions. We first determine the coefficients
$p_{\alpha,\beta}^{(k,r)}$ in the expansion \eqref{eq:p-definition}. For $\alpha\in\mathbb Z_{\geq0}^r$, write
$|\alpha|=\sum_{i=1}^r \alpha_i$ and
$\mathbf x^\alpha= \prod_{i=1}^{r}x_i^{\alpha_i}$. Recall
\begin{equation}\label{eq:p-definition2}
 \Phi_{k,r}(\mathbf x,\mathbf y)^2
 =
 \sum_{\alpha,\beta\in\mathbb Z_{\geq0}^r}
 p_{\alpha,\beta}^{(k,r)}
 \mathbf x^\alpha\mathbf y^\beta.
\end{equation}
For integers $k\geq1$ and $1\leq r\leq k$, set
\begin{equation}
\label{eq:p-support-set}
\mathcal A_{k,r}
:=
\left\{
(\alpha,\beta)\in\mathbb Z_{\geq0}^r\times\mathbb Z_{\geq0}^r:
|\alpha|+|\beta|=2(k-r)
\right\}.
\end{equation}
The proof below shows that $\Phi_{k,r}^2$ is homogeneous of degree
$2(k-r)$, and hence that $p_{\alpha,\beta}^{(k,r)}=0$ whenever
$(\alpha,\beta)\notin\mathcal A_{k,r}$. For $u,v\in\mathbb Z_{\geq0}^r$, let
\begin{equation}\label{eq:Dr-definition}
 D_r(u,v)
 :=
 \det_{1\leq i,j\leq r}
 \left[\frac{1}{u_i+v_j+1}\right].
\end{equation}
With the standard abbreviations
$\alpha!=\prod_i\alpha_i!$ and
$\binom{\alpha}{u}=\prod_i\binom{\alpha_i}{u_i}$, where
$u\leq\alpha$ is understood componentwise, the coefficients are as follows.

\begin{proposition}\label{prop:p-explicit}
Let $k\geq 1$ and $1 \leq r \leq k$ be integers. If $(\alpha,\beta)\in\mathcal A_{k,r}$, then
\begin{equation}\label{eq:p-explicit}
 p_{\alpha,\beta}^{(k,r)}
 =
 \frac{(-1)^{|\beta|}(k!)^2}{(\alpha!\beta!)^2}
 \sum_{\substack{0\leq u\leq\alpha,\;0\leq v\leq\beta\\
                 |u|+|v|=k-r}}
 \binom{\alpha}{u}^{\!2}
 \binom{\beta}{v}^{\!2}
 D_r(u,v)D_r(\alpha-u,\beta-v),
\end{equation}
with $\mathcal A_{k,r}$ as defined in \eqref{eq:p-support-set} and $D_r(u,v)$ in \eqref{eq:Dr-definition}.
\end{proposition}

First, we will prove the above proposition and then proceed to prove the finite-spline form for $\delta_k(c)$ in Theorem~\ref{thm:lauricella}.

\begin{proof}[Proof of Proposition~\ref{prop:p-explicit}]
We begin with the binomial form of the Legendre polynomial,
\begin{equation}\label{eq:Legendre-binomial-form}
 P_n(z)
 =
 2^{-n}\sum_{a=0}^n
 \binom{n}{a}^{\!2}
 (z-1)^{n-a}(z+1)^a.
\end{equation}
Substituting $z=(x-y)/(x+y)$ into the above and multiplying by $(x+y)^n$ gives the
polynomial identity
\begin{equation}\label{eq:Legendre-homogeneous-expansion}
 (x+y)^nP_n\!\left(\frac{x-y}{x+y}\right)
 =
 \sum_{\substack{a,b\geq0\\a+b=n}}
 (-1)^b\binom{n}{a}^{\!2}x^ay^b.
\end{equation}
Here the latter is not a double sum over independent indices: the constraint $a+b=n$ uniquely determines $b$, so this remains a single sum with $n+1$ terms.
Using \eqref{eq:Theta-legendre-def}, we obtain
\begin{equation}\label{eq:Psi-monomial-expansion}
 \Theta_m(x,y)
 =
 \sum_{\substack{a,b\geq0\\a+b=m-1}}
 \frac{(-1)^b}{(a!b!)^2(a+b+1)}x^ay^b.
\end{equation}
In particular, $\Theta_m$ is a homogeneous polynomial of degree $m-1$, i.e. every monomial $x^ay^b$ appearing in it has total degree $m-1$. For every choice of positive integers $m_1,\ldots,m_r$ with $m_1+\cdots+m_r=k$, each term in
$\det[\Theta_{m_i}(x_i,y_j)]$ selects exactly one entry from each row.
The entry selected from row $i$ has total degree $m_i-1$, so every term in the determinant has total degree
\begin{equation*}
\sum_{i=1}^r(m_i-1)=k-r.
\end{equation*}
Therefore $\Phi_{k,r}$ is homogeneous of degree $k-r$, and $\Phi_{k,r}^2$ is homogeneous of degree $2(k-r)$.

We now expand the determinant in \eqref{eq:Phi-def} using Leibniz's formula and insert \eqref{eq:Psi-monomial-expansion}. For a fixed permutation $\sigma\in S_r$, the associated product in the expansion is
\begin{equation*}
 \prod_{i=1}^r \Theta_{m_i}(x_i,y_{\sigma(i)}).
\end{equation*}
We reindex the exponents by writing $u_i=a_i$ and $b_i=v_{\sigma(i)}$. Thus, for fixed $u,v\in\mathbb Z_{\geq0}^r$, the $i$th factor $\Theta_{m_i}(x_i,y_{\sigma(i)})$ must contribute the monomial $x_i^{u_i}y_{\sigma(i)}^{v_{\sigma(i)}}$. By \eqref{eq:Psi-monomial-expansion}, its coefficient is
\begin{equation*}
 \frac{(-1)^{v_{\sigma(i)}}}
 {(u_i!v_{\sigma(i)}!)^2
 \bigl(u_i+v_{\sigma(i)}+1\bigr)},
\end{equation*}
and the same expansion requires the exponents in this factor to satisfy $u_i+v_{\sigma(i)}=m_i-1$, and hence $m_i=u_i+v_{\sigma(i)}+1$. Summing over $i$ gives
\begin{equation*}
 \sum_{i=1}^r m_i
 =
 \sum_{i=1}^r u_i
 +
 \sum_{i=1}^r v_{\sigma(i)}
 +
 r
 =
 |u|+|v|+r,
\end{equation*}
and multiplying the row coefficients gives
\begin{align*}
 \prod_{i=1}^r
 \frac{(-1)^{v_{\sigma(i)}}}
 {(u_i!v_{\sigma(i)}!)^2
 \bigl(u_i+v_{\sigma(i)}+1\bigr)}
 &=
 \frac{(-1)^{|v|}}{(u!v!)^2}
 \prod_{i=1}^r
 \frac{1}{u_i+v_{\sigma(i)}+1},
\end{align*}
because $\sum_i v_{\sigma(i)}=|v|$ and
$\prod_i v_{\sigma(i)}!=v!$. Summing these contributions over
$\sigma\in S_r$, with the $\operatorname{sgn}(\sigma)$ from
Leibniz's formula, yields
\begin{align}
 \Phi_{k,r}(\mathbf x,\mathbf y)
 &={}
 k!\!\sum_{\substack{u,v\in\mathbb Z_{\geq0}^r\\
                      |u|+|v|=k-r}}
 \frac{(-1)^{|v|}}{(u!v!)^2}
 \left(
  \sum_{\sigma\in S_r}
  \operatorname{sgn}(\sigma)
  \prod_{i=1}^r\frac{1}{u_i+v_{\sigma(i)}+1}
 \right)
 \mathbf x^u\mathbf y^v
 \notag\\
 &={}
 k!\!\sum_{\substack{u,v\in\mathbb Z_{\geq0}^r\\
                      |u|+|v|=k-r}}
 \frac{(-1)^{|v|}D_r(u,v)}{(u!v!)^2}
 \mathbf x^u\mathbf y^v.
 \label{eq:Phi-explicit-monomial}
\end{align}
The inner sum in the first line is exactly the determinant
in \eqref{eq:Dr-definition}.

It remains to square \eqref{eq:Phi-explicit-monomial}.  Suppose that
$(\alpha,\beta)\in\mathcal A_{k,r}$.  If one factor contributes the
exponents $(u,v)$, the other must contribute
$(\alpha-u,\beta-v)$.  The first factor has total degree $k-r$ exactly
when $|u|+|v|=k-r$; because
$|\alpha|+|\beta|=2(k-r)$, the complementary pair then has the same total
degree.  Hence
\begin{align*}
 p_{\alpha,\beta}^{(k,r)}
 ={}&(k!)^2
 \sum_{\substack{0\leq u\leq\alpha,\;0\leq v\leq\beta\\
                  |u|+|v|=k-r}}
 \frac{(-1)^{|v|+|\beta-v|}
       D_r(u,v)D_r(\alpha-u,\beta-v)}
 {(u!v!)^2((\alpha-u)!(\beta-v)!)^2}.
\end{align*}
The sign becomes $(-1)^{|\beta|}$, and the componentwise factorial identities
give
\begin{equation*}
 \frac{1}{(u!(\alpha-u)!)^2}
 =
 \frac{1}{(\alpha!)^2}\binom{\alpha}{u}^{\!2},
 \qquad
 \frac{1}{(v!(\beta-v)!)^2}
 =
 \frac{1}{(\beta!)^2}\binom{\beta}{v}^{\!2}.
\end{equation*}
Substitution gives \eqref{eq:p-explicit}.  If
$(\alpha,\beta)\notin\mathcal A_{k,r}$, the coefficient is zero by the
homogeneity proved above.
\end{proof}

\begin{proof}[Proof of Theorem~\ref{thm:lauricella}]
We begin by substituting the expansion \eqref{eq:p-definition2} into \eqref{eq:delta-def}. As this expansion is finite, we may integrate term by term, so the evaluation of $\delta_k(c)$ reduces to computing the contribution from each term $\mathbf x^\alpha\mathbf y^\beta$. For $(\alpha,\beta)\in\mathcal A_{k,r}$, set
\begin{equation}\label{eq:monomial-delta-integral}
 \mathcal I_{\alpha,\beta}^{(r)}(c)
 :=
 \int_{\mathbb R_{\geq0}^r\times[0,1]^r}
 \mathbf x^\alpha\mathbf y^\beta
 \delta \hspace{0.02in} \biggl(\sum_{i=1}^r(x_i+y_i)-c\biggr)
 \,d\mathbf x\,d\mathbf y.
\end{equation}
Using Theorem~\ref{thm:leading-volume}, we obtain
\begin{equation}\label{eq:delta-monomial-integral}
 \delta_k(c)
 =
 \sum_{r=1}^{\lfloor\sqrt{k}\rfloor}\frac{1}{(r!)^2}
 \sum_{(\alpha,\beta)\in\mathcal A_{k,r}}
 p_{\alpha,\beta}^{(k,r)}
 \mathcal I_{\alpha,\beta}^{(r)}(c).
\end{equation}
We use the convention $t_+^0:=\mathds{1}_{\{t>0\}}$.
For $t\neq0$, the Dirichlet simplex identity gives
\begin{equation}\label{eq:Dirichlet-simplex-integral}
 \int_{\mathbb R_{\geq0}^r}
 \mathbf x^\alpha \,
 \delta \hspace{0.01in} \biggl(\sum_{i=1}^r x_i-t\biggr)d\mathbf x
 =
 \frac{\alpha!}{(|\alpha|+r-1)!}
 t_+^{|\alpha|+r-1}.
\end{equation}
The value at $t=0$ does not affect the subsequent integrals. With $ d=d_{\alpha,r}=|\alpha|+r-1$ and $t=c-\sum_i y_i$, the above gives
\begin{equation}\label{eq:monomial-cube-reduction}
 \mathcal I_{\alpha,\beta}^{(r)}(c)
 =
 \frac{\alpha!}{d!}\,I_{\beta,d}(c),
 \qquad
 I_{\beta,d}(c)
 :=
 \int_{[0,1]^r}
 \mathbf y^\beta
 \biggl(c-\sum_{i=1}^r y_i\biggr)_+^d d\mathbf y.
\end{equation}
We first evaluate $I_{\beta,d}(c)$ for arbitrary $c>0$. The integral is
over the cube $[0,1]^r$, so in addition to $y_i\geq0$ we must impose the
upper bounds $y_i\leq1$. We do this by an inclusion--exclusion mechanism. 
Let $[r] := \{1,2,\ldots, r\}$ and
\begin{equation*}
 f(\mathbf y)
 :=
 \mathbf y^\beta
 \biggl(c-\sum_{i=1}^r y_i\biggr)_+^d.
\end{equation*}
On the region $\mathbf y\in\mathbb R_{\geq0}^r$, the condition
$0\leq y_i<1$ has indicator $1-\mathds{1}_{\{y_i\geq1\}}$. Hence the indicator of the cube $[0,1)^r$ is
\begin{equation*}
 \prod_{i=1}^r\left(1-\mathds{1}_{\{y_i\geq1\}}\right).
\end{equation*}
Expanding this product amounts to choosing, for each subset
$S\subseteq[r]$, the factor $-\mathds{1}_{\{y_i\geq1\}}$ for every
$i\in S$ and the factor $1$ for every $i\notin S$. Therefore
\begin{equation*}
 \prod_{i=1}^r\left(1-\mathds{1}_{\{y_i\geq1\}}\right)
 =
 \sum_{S\subseteq[r]}(-1)^{|S|}
 \prod_{i\in S}\mathds{1}_{\{y_i\geq1\}}.
\end{equation*}
Multiplying by $f(\mathbf y)$ and integrating over
$\mathbb R_{\geq0}^r$ gives
\begin{align}
 I_{\beta,d}(c)
 ={}&
 \sum_{S\subseteq[r]}(-1)^{|S|}
 \int_{\substack{\mathbf y\in\mathbb R_{\geq0}^r\\
                  y_i\geq1\ (i\in S)}}
 \mathbf y^\beta
 \biggl(c-\sum_{i=1}^r y_i\biggr)_+^d
 \,d\mathbf y.
 \label{eq:cube-inclusion-exclusion}
\end{align}
Replacing
$[0,1)^r$ by $[0,1]^r$ does not change the integral, because their
difference lies on the boundary $y_i=1$, which has measure zero.

We now evaluate the integral corresponding to each subset $S\subseteq[r]$. For $i\in S$, the condition $y_i\geq1$ gives a lower bound of $1$, while the simplex formula is most naturally applied when all variables start at $0$. We therefore set $y_i=1+z_i$ for $i\in S$ and $y_i=z_i$ for $i\notin S$. This change of variables maps the region of integration onto
$\mathbb R_{\geq0}^r$ and has Jacobian $1$. It also gives
\begin{equation*}
 c-\sum_{i=1}^r y_i
 =
 c-|S|-\sum_{i=1}^r z_i, \qquad \text{and} \qquad  \mathbf y^\beta
 =
 \prod_{i\in S}(1+z_i)^{\beta_i}
 \prod_{i\notin S}z_i^{\beta_i}.
\end{equation*}
To reduce the shifted factors to ordinary powers of the $z_i$, we expand
\begin{equation*}
 (1+z_i)^{\beta_i}
 =
 \sum_{q_i=0}^{\beta_i}
 \binom{\beta_i}{q_i}z_i^{\beta_i-q_i},
 \qquad i\in S.
\end{equation*}
For each fixed choice of $q_i$, define
\begin{equation*}
 \gamma_i
 =
 \begin{cases}
  \beta_i-q_i, & i\in S,\\
  \beta_i, & i\notin S.
 \end{cases}
\end{equation*}
Then $|\gamma|=|\beta|-\sum_{i\in S}q_i$ and the corresponding integral has the form
\begin{equation*}
 \int_{\mathbb R_{\geq0}^r}
 \mathbf z^\gamma
 \Bigl(c-|S|-\sum_{i=1}^r z_i\Bigr)_+^d
 \,d\mathbf z.
\end{equation*}
We evaluate these integrals using the following consequence of the simplex
integral formula: for $\gamma\in\mathbb Z_{\geq0}^r$,
$d\in\mathbb Z_{\geq0}$, and $\rho\in\mathbb R$,
\begin{equation}
\label{eq:augmented-Dirichlet-integral}
 \int_{\mathbb R_{\geq0}^r}
 \mathbf z^\gamma
 \left(\rho-\sum_{i=1}^r z_i\right)_+^d
 \,d\mathbf z
 =
 \frac{\gamma!\,d!}{(d+|\gamma|+r)!}
 \rho_+^{d+|\gamma|+r}.
\end{equation}
To verify this identity, introduce one additional variable
$w=\rho-\sum_{i=1}^r z_i$. The integral may then be written as
\begin{equation*}
 \int_{\mathbb R_{\geq0}^{r+1}}
 \mathbf z^\gamma w^d
 \delta\!\left(\sum_{i=1}^r z_i+w-\rho\right)
 \,d\mathbf z\,dw,
\end{equation*}
and \eqref{eq:augmented-Dirichlet-integral} follows from
\eqref{eq:Dirichlet-simplex-integral} in dimension $r+1$.

Applying \eqref{eq:augmented-Dirichlet-integral} with $\rho=c-|S|$ to each
term in the binomial expansion gives
\begin{align*}
 I_{\beta,d}(c)
 ={}&
 \sum_{S\subseteq[r]}(-1)^{|S|}
 \sum_{\substack{0\leq q_i\leq\beta_i\\ i\in S}}
 \left(\prod_{i\in S}\binom{\beta_i}{q_i}\right)
 \frac{\gamma!\,d!}
 {\left(d+|\gamma|+r\right)!}
 (c-|S|)_+^{d+|\gamma|+r}.
\end{align*}
The factorials simplify using
\begin{equation*}
 \left(\prod_{i\in S}\binom{\beta_i}{q_i}\right)\gamma!
 =
 \frac{\beta!}{\displaystyle\prod_{i\in S}q_i!}.
\end{equation*}
Setting $M:=d+|\beta|+r$ and using $d+|\gamma|+r
 = M-\sum_{i\in S}q_i$, we obtain
\begin{equation}
\label{eq:cube-moment-spline}
 I_{\beta,d}(c)
 =
 \beta!d!
 \sum_{S\subseteq[r]}(-1)^{|S|}
 \sum_{\substack{0\leq q_i\leq\beta_i\\ i\in S}}
 \frac{(c-|S|)_+^{M-\sum_{i\in S}q_i}}
 {\left(M-\sum_{i\in S}q_i\right)!
 \displaystyle\prod_{i\in S}q_i!}.
\end{equation}
For $(\alpha,\beta)\in\mathcal A_{k,r}$, the support relation
$|\alpha|+|\beta|=2(k-r)$ gives
\begin{equation}
\label{eq:M-equals-2k-minus-1}
 M
 =
 d_{\alpha,r}+|\beta|+r
 =
 |\alpha|+|\beta|+2r-1
 =
 2k-1.
\end{equation}
Substituting \eqref{eq:cube-moment-spline} into
\eqref{eq:monomial-cube-reduction}, and then using
\eqref{eq:delta-monomial-integral}, gives the finite-spline expansion stated
in the theorem.

We now suppose that $c\geq\lfloor\sqrt{k}\rfloor$. Every rank in
\eqref{eq:delta-monomial-integral} satisfies $r\leq\lfloor\sqrt{k}\rfloor$, and therefore $c\geq r$. Moreover, for $\mathbf y\in[0,1]^r$, we have $\sum_{i=1}^r y_i\leq r\leq c$. Hence $c-\sum_i y_i\geq0$ throughout the cube, so the positive-part truncation in
\eqref{eq:monomial-cube-reduction} may be omitted (the possible equality at
the boundary does not affect the integral). As $d$ is a non-negative
integer, the multinomial theorem gives the finite expansion
\begin{equation}\label{eq:multinomial-FA-expansion}
 \biggl(c-\sum_{i=1}^r y_i\biggr)^d
 =
 c^d
 \sum_{m\in\mathbb Z_{\geq0}^r, \, |m|\leq d}
 \frac{(-d)_{|m|}}{m!}
 \prod_{i=1}^r\left(\frac{y_i}{c}\right)^{m_i}.
\end{equation}
The sum is finite, so it may be integrated term by term.  For every $i$,
\begin{equation}\label{eq:Pochhammer-ratio}
 \int_0^1y_i^{\beta_i+m_i}\,dy_i
 =
 \frac{1}{\beta_i+1}
 \frac{(\beta_i+1)_{m_i}}{(\beta_i+2)_{m_i}}.
\end{equation}
Combining these last two results, and comparing the resulting finite
series with the definition in \eqref{eq:FA-definition}, we obtain
\begin{equation}\label{eq:cube-moment-FA}
 I_{\beta,d}(c)
 =
 \frac{c^d}{\displaystyle\prod_{i=1}^r(\beta_i+1)}
 F_A^{(r)}
 \left(
  -d;
  \beta+\mathbf1_r;
  \beta+2\mathbf1_r;
  c^{-1}\mathbf1_r
 \right).
\end{equation}
Substituting \eqref{eq:cube-moment-FA} into
\eqref{eq:monomial-cube-reduction}, and then into
\eqref{eq:delta-monomial-integral}, proves
\eqref{eq:delta-Lauricella}. Since $(-d)_{|m|}=0$ for $|m|>d$, every
Lauricella series here terminates. Thus the argument requires neither a
convergence condition nor analytic continuation, and it remains valid at
$c=\lfloor\sqrt{k}\rfloor$.

For completeness, the same formulas justify the two structural observations
following the theorem. Each term in \eqref{eq:delta-Lauricella} is a
polynomial in $c$ of degree at most $d_{\alpha,r}$, and
\begin{equation*}
 d_{\alpha,r}
 =|\alpha|+r-1
 \leq2(k-r)+r-1
 \leq2k-2.
\end{equation*}
At a candidate breakpoint $j$, only the terms with $|S|=j$ in
\eqref{eq:cube-moment-spline} change from zero to non-zero. Their truncated
powers have exponent
\begin{equation*}
 2k-1-\sum_{i\in S}q_i
 \geq2k-1-|\beta|
 \geq2r-1
 \geq2j-1.
\end{equation*}
Consequently each such term is $C^{2j-2}$ at $c=j$, and so is
$\delta_k(c)$. 
\end{proof}

\begin{remark}
\label{rem:FB-spline}
The Lauricella family $F_B$ is distinct from the family $F_A$ used in \eqref{eq:delta-Lauricella}. For $s\geq1$,
$\mathbf a,\mathbf b,\mathbf z\in\mathbb C^s$, 
$\gamma\notin\{0,-1,-2,\ldots\}$, and $|z_i|<1$ for every $i$, define
\begin{equation*}
 F_B^{(s)}(\mathbf a;\mathbf b;\gamma;\mathbf z)
 =
 \sum_{m_1,\ldots,m_s\geq0}
 \frac{1}{(\gamma)_{|m|}}
 \prod_{i=1}^s
 \frac{(a_i)_{m_i}(b_i)_{m_i}}{m_i!}z_i^{m_i},
 \qquad |m|=m_1+\cdots+m_s.
\end{equation*}
For $S\subseteq[r]$, put $B_S=\sum_{i\in S}\beta_i,$
and $ t_S=(c-|S|)_+$. Write $\beta_S=(\beta_i)_{i\in S}$ and $\mathbf1_S=(1)_{i\in S}$.
Then the innermost sum in the finite-spline formula of Theorem~\ref{thm:lauricella} may be replaced exactly by
\begin{align*}
 &\sum_{\substack{0\leq q_i\leq\beta_i\\ i\in S}}
 \hspace{-0.02in} \frac{t_S^{\,2k-1-\sum_{i\in S}q_i}}
 {\left(2k-1-\sum_{i\in S}q_i\right)!
  \displaystyle\prod_{i\in S}q_i!}
 \hspace{-0.02in} = \hspace{-0.02in}
 \frac{t_S^{\,2k-1-B_S}}
 {(2k-1-B_S)!\displaystyle\prod_{i\in S}\beta_i!}
 F_B^{(|S|)}
 \left(
 -\beta_S;\mathbf1_S;
 2k-B_S;
 -t_S\mathbf1_S
 \right).
\end{align*}
Since every component of $-\beta_S$ is a non-positive integer, this $F_B^{(|S|)}$ series terminates and hence introduces no convergence restriction.
For $S=\varnothing$, both sides are interpreted as $c^{2k-1}/(2k-1)!$, while on the final polynomial range the resulting alternating sum of terminating $F_B$ polynomials is equal to the $F_A$ expression in \eqref{eq:delta-Lauricella}.
\end{remark}

\subsubsection{Related literature} Within random matrix theory, Lauricella functions may be viewed as scalar
specialisations of certain matrix-variate hypergeometric functions \cite{kn:mathai93, kn:st18}; for special parameter choices, they can be related to broader hypergeometric theory associated with $\beta$-ensembles \cite{kn:st18}. These matrix-argument functions, developed by Herz~\cite{kn:h55} amongst others, generalise ${}_pF_q$ by replacing powers of a scalar variable with Jack polynomials of the eigenvalues of a matrix argument. Related hypergeometric functions of
matrix argument have appeared in the study of characteristic polynomials \cite{kn:deh10, kn:ssd23}.

\section{Cumulative Trace Averages and block Hankel Determinants}
\label{section:block-Hankel-determinant}

In this section, we relate $\delta_k(c)$ to a two-weight multiple Hankel determinant by proving Theorem~\ref{thm:two-centre}. The central external input is a contour integral
formula for the cumulative leading coefficient $M_k(\alpha)$ \cite{kn:bcis26}, stated below in Theorem~\ref{thm:two-centre-P-input}. This formula ensures that the limit defining $M_k(\alpha)$ exists, and Proposition~\ref{prop:derivative-relationship} then identifies $M_k'(\alpha)$ with $\delta_k(\alpha)$.

Fix $\tau\in(0,\tfrac12)$, and let $C_\zeta$ denote a positively oriented
circle of radius $\tau$ centred at $\zeta$.  For $0\leq r\leq k$, define
\begin{equation}\label{eq:two-centre-contours}
 c_j(r)=
 \begin{cases}
  C_1,&1\leq j\leq r,\\
  C_0,&r<j\leq k,
 \end{cases}
 \qquad
 \hat c_\ell(r)=
 \begin{cases}
  C_{-1},&1\leq \ell\leq r,\\
  C_0,&r<\ell\leq k.
 \end{cases}
\end{equation}
Thus the index $r$ specifies how many of the $x$-contours are centred at
$1$, and simultaneously how many of the $y$-contours are centred at $-1$.
We write $T(\mathbf x)=\sum_{j=1}^k x_j$ and use the abbreviations
$\int_{\mathbf c(r)}:=\int_{c_1(r)}\!\cdots\!\int_{c_k(r)}$ and
$\int_{\hat{\mathbf c}(r)}:=\int_{\hat c_1(r)}\!\cdots\!\int_{\hat c_k(r)}$ to denote the two sets of $k$-fold contour integrals.

\begin{theorem}
\label{thm:two-centre-P-input}
Let $k\in\mathbb N$. For $\alpha>0$,
\begin{equation}\label{eq:two-centre-M-P}
 M_k(\alpha)=\frac1{k!^2}\sum_{0\leq r<\alpha}\binom{k}{r}^{\!2}P_{r,k}(\alpha),
\end{equation}
where, for $0\leq r\leq k$,
\begin{align}
 P_{r,k}(\alpha)
 ={}&\frac1{(2\pi i)^{2k}}
 \sum_{a=0}^k\frac{(-1)^{k-a}}{(k+a)!}\binom{k}{a}
 \int_{\mathbf c(r)}\int_{\hat{\mathbf c}(r)}
 \bigl(T(\mathbf x)-\alpha\bigr)^{k+a} \, \Delta(\mathbf x)^2\Delta(\mathbf y)^2
 \prod_{j,\ell=1}^k(x_j-y_\ell-1)
 \notag\\[-1mm]
 &\times
 \left(
   \sum_{j=1}^k\frac1{x_j}
   +\sum_{\ell=1}^k\frac1{y_\ell+1}
 \right)^{\!a}
 \left(
   \sum_{j=1}^k\frac1{x_j-1}
   +\sum_{\ell=1}^k\frac1{y_\ell}
 \right)^{\!k}
 \prod_{j=1}^k x_j^{-k}
 \prod_{\ell=1}^k y_\ell^{-k}
 \,d\mathbf x\,d\mathbf y,
 \label{eq:two-centre-P-original}
\end{align}
where $\Delta(\mathbf{x})=\prod_{i < j}(x_j - x_i)$ and we use
$d\mathbf x=\prod_{j=1}^k dx_j$, $d\mathbf y=\prod_{\ell=1}^k dy_\ell$.
For $r>k$, set $P_{r,k}(\alpha)=0$.  Moreover, $P_{r,k}(\alpha)$ has a zero of
order at least $2r^2$ at $\alpha=r$.
\end{theorem}
\begin{proposition}
\label{prop:derivative-relationship}
Let $k\geq1$ be a fixed integer. For $\alpha>0$,
\begin{equation}
\label{eq:Mk-derivative}
 M_k'(\alpha)=\delta_k(\alpha).
\end{equation}
\end{proposition}

\begin{proof}
Equations~\eqref{eq:Jkr-lipschitz} and
\eqref{eq:Jr-unordered-delta} show that $\delta_k$ is locally Lipschitz,
and hence continuous, on $(0,\infty)$. Haar invariance under
$U\mapsto e^{i\varphi}U$ gives
\begin{equation}
\label{eq:Itilde-riemann-sum}
 \frac{\tilde I_k(\lfloor \alpha N\rfloor;N)}{N^{2k}}
 =
 \frac1N\sum_{m=k}^{\lfloor \alpha N\rfloor}
 \frac{I_k(m;N)}{N^{2k-1}}.
\end{equation}
Fix $0<\eta<\min(1,\alpha)$ and split the sum at
$m=\lfloor\eta N\rfloor$. On the range $ \lfloor\eta N\rfloor<m\leq\lfloor \alpha N\rfloor$, the ratio $m/N$ lies in the compact interval $[\eta,\alpha]$.
Theorem~\ref{thm:leading-volume} therefore applies uniformly to these
summands. After the normalisation in \eqref{eq:Itilde-riemann-sum}, the
accumulated error is $O(N^{-1})$, while the main terms form a Riemann sum
converging to $\int_\eta^\alpha \delta_k(c) \, dc $.

The remaining range
\begin{equation*}
 k\leq m\leq\lfloor\eta N\rfloor
\end{equation*}
is not uniformly bounded away from $m/N=0$ and hence is not covered by
the compact-uniform estimate in Theorem~\ref{thm:leading-volume}.
Since $\eta<1$, however, this range satisfies $m\leq N$ for all
sufficiently large $N$ and therefore lies in the stable range. By
\eqref{eq:Ik-stable-range} and \eqref{eq:delta-stable-range},
\begin{equation*}
 \frac1N
 \sum_{m=k}^{\lfloor\eta N\rfloor}
 \frac{I_k(m;N)}{N^{2k-1}}
 \longrightarrow
 \int_0^\eta\delta_k(c)\,dc.
\end{equation*}
Combining the two ranges and using the definition of $M_k(\alpha)$ gives
\begin{equation}
\label{eq:Mk-integral}
 M_k(\alpha)
 =
 \int_0^\eta\delta_k(c)\,dc
 +
 \int_\eta^\alpha \delta_k(c)\,dc
 =
 \int_0^\alpha \delta_k(c)\,dc.
\end{equation}
Equation~\eqref{eq:Mk-derivative} now follows from the continuity of
$\delta_k$ and the fundamental theorem of calculus.
\end{proof}

\begin{proof}[Proof of Theorem~\ref{thm:two-centre}]
Fix an integer $k\geq1$, and let $\alpha>0$ be real. We differentiate the 
$2k$-contour formula in Theorem~\ref{thm:two-centre-P-input}, and then organise the result in terms of a two-weight multiple Hankel
determinant. The key observation is that the $r$-th contour sector, i.e. the $r$-th term in \eqref{eq:two-centre-M-P},
contributes only for $\alpha>r$. The only point requiring care in differentiating
\eqref{eq:two-centre-M-P} is that the summation range changes when
$\alpha$ crosses an integer. Since $P_{r,k}=0$ for $r>k$, we may replace
this varying range by the fixed sum
\begin{equation}\label{eq:two-centre-M-fixed-sum}
 M_k(\alpha)
 =
 \frac1{k!^2}
 \sum_{r=0}^k
 \binom{k}{r}^{\!2}
 \mathds{1}_{\{\alpha>r\}}P_{r,k}(\alpha),
 \qquad \alpha>0,
\end{equation}
where the indicator retains the original condition $r<\alpha$. For
$r\geq1$, Theorem~\ref{thm:two-centre-P-input} gives
$P_{r,k}(r)=P'_{r,k}(r)=0$, so the function
$\alpha\mapsto\mathds{1}_{\{\alpha>r\}}P_{r,k}(\alpha)$ is continuously
differentiable at $\alpha=r$. For $r=0$, the indicator is identically one
on $\alpha>0$. Evaluating at $\alpha=c$ and using $M_k'(c)=\delta_k(c)$, we obtain
\begin{equation}\label{eq:two-centre-delta-fixed-sum}
 \delta_k(c)
 =
 \frac1{k!^2}
 \sum_{r=0}^k
 \binom{k}{r}^{\!2}
 \mathds{1}_{\{c>r\}}P'_{r,k}(c),
 \qquad c>0.
\end{equation}
We may simplify $P_{r,k}(c)$ by using $(-1)^{k-a}\bigl(T(\mathbf x)-c\bigr)^{k+a}
 =
 \bigl(c -T(\mathbf x)\bigr)^{k+a}$. Thus, on the range $c>r$, differentiating the ordinary polynomial $P_{r,k}(c)$ gives the same contour integral as
\eqref{eq:two-centre-P-original}, with
\begin{equation*}
 \frac{\bigl(c -T(\mathbf x)\bigr)^{k+a-1}}
      {(k+a-1)!}
\end{equation*}
in place of the original power.

It remains to encode the factor $\mathds{1}_{\{c>r\}}$.  Because the variables $x_j$ lie on complex contours, the notation
$(c-T(\mathbf x))_+^m$ will not mean the pointwise positive part of a
complex number. Instead, for the rational contour factors occurring here, we use it as shorthand for the residue-first
inverse Laplace expression
\begin{align}
\int_{\mathbf c(r)}\int_{\hat{\mathbf c}(r)}
 \hspace{-0.05in} \frac{(c-T(\mathbf x))_+^m}{m!}
 F(\mathbf x,\mathbf y)\,d\mathbf x\,d\mathbf y
:=
 \mathcal L^{-1}_{s\to c}
 \left[
  \frac1{s^{m+1}}
  \int_{\mathbf c(r)}\int_{\hat{\mathbf c}(r)}
  e^{-sT(\mathbf x)}
  F(\mathbf x,\mathbf y)\,d\mathbf x\,d\mathbf y
 \right] \hspace{-0.05in}.
\label{eq:two-centre-residue-first}
\end{align}
The contour residues inside the square brackets are evaluated before the
inverse Laplace transform is taken. To see how this expression encodes the threshold $c=r$, write in the
$r$-th contour sector
\begin{equation*}
 x_j=\varepsilon_j+\xi_j,
 \qquad
 \varepsilon_j=
 \begin{cases}
  1,&j\leq r,\\
  0,&j>r.
 \end{cases}
\end{equation*}
Then
\begin{equation*}
 T(\mathbf x)=r+\sum_{j=1}^k\xi_j,
 \qquad
 e^{-sT(\mathbf x)}
 =
 e^{-rs}
 \exp\!\biggl(-s\sum_{j=1}^k\xi_j\biggr).
\end{equation*}
The contour poles have finite order, so evaluating the residues uses only
finitely many terms in the Taylor series of the final exponential.  The
result is therefore \(e^{-rs}\) multiplied by a finite Laurent polynomial
in \(s\).  Its negative powers are inverted using
\begin{equation}\label{eq:two-centre-shifted-Laplace}
 \mathcal L^{-1}_{s\to c}
 \left[e^{-rs}s^{-q}\right]
 =
 \frac{(c-r)_+^{q-1}}{(q-1)!},
 \qquad q\geq1.
\end{equation}
Every such term is zero for $c<r$ and becomes an ordinary polynomial in
$c-r$ for $c>r$.  Consequently, away from $c=r$, the residue-first
expression in \eqref{eq:two-centre-residue-first} is precisely the ordinary
contour polynomial multiplied by $\mathds{1}_{\{c>r\}}$. At $c=r$, the
complete $r$-th sector is assigned its continuous value; this is
well-defined because $P_{r,k}(r)=P'_{r,k}(r)=0$.

We now shift the $y$-variables by setting $ w_\ell=y_\ell+1$ for $1\leq\ell\leq k$. In the sector indexed by $r$, the contours become
\begin{equation}\label{eq:two-centre-contour-system}
 x_j\in
 \begin{cases}
  C_1,&j\leq r,\\
  C_0,&j>r,
 \end{cases}
 \qquad
 w_\ell\in
 \begin{cases}
  C_0,&\ell\leq r,\\
  C_1,&\ell>r.
 \end{cases}
\end{equation}
Denote these iterated contour systems by $\mathcal C_x(r)$ and $\mathcal C_w(r)$, and write $d\mathbf w=\prod_{\ell=1}^k dw_\ell$. Combining \eqref{eq:two-centre-delta-fixed-sum} with the residue-first
convention \eqref{eq:two-centre-residue-first}, after this change of
variables, gives
\begin{align}
 \delta_k(c)
 ={}&\frac1{(2\pi i)^{2k}k!^2}
 \mathcal L^{-1}_{s\to c}
 \Bigg[
 \sum_{a=0}^k\binom{k}{a}s^{-(k+a)}
 \sum_{r=0}^k\binom{k}{r}^{\!2}
 \notag\\[-1mm]
 &\hspace{22mm}\times
 \int_{\mathcal C_x(r)}\int_{\mathcal C_w(r)}
 e^{-sT(\mathbf x)}
 A(\mathbf x,\mathbf w)^a B(\mathbf x,\mathbf w)^k
 \mathcal V(\mathbf x,\mathbf w)
 \,d\mathbf x\,d\mathbf w
 \Bigg],
 \label{eq:two-centre-delta-Laplace-shifted}
\end{align}
where
\begin{equation*}
 \mathcal V(\mathbf x,\mathbf w)
 :=\Delta(\mathbf x)^2\Delta(\mathbf w)^2
   \prod_{j,\ell=1}^k(x_j-w_\ell)
   \prod_{j=1}^k x_j^{-k}
   \prod_{\ell=1}^k(w_\ell-1)^{-k},
\end{equation*}
and
\begin{equation*}
 A(\mathbf x,\mathbf w)
 :=\sum_{j=1}^k\frac1{x_j}+\sum_{\ell=1}^k\frac1{w_\ell},
 \qquad
 B(\mathbf x,\mathbf w)
 :=\sum_{j=1}^k\frac1{x_j-1}+\sum_{\ell=1}^k\frac1{w_\ell-1}.
\end{equation*}
The inverse Laplace transform is understood in the residue-first sense of
\eqref{eq:two-centre-residue-first}; in particular, it already encodes the restriction $c>r$. The powers of $A(\mathbf x,\mathbf w)$ and $B(\mathbf x,\mathbf w)$ above are generated by
\begin{align}
A(\mathbf x,\mathbf w)^a
&=
\left.
\partial_u^a
\exp\!\bigl(uA(\mathbf x,\mathbf w)\bigr)
\right|_{u=0}, \qquad
B(\mathbf x,\mathbf w)^k
&=
\left.
\partial_v^k
\exp\!\bigl(vB(\mathbf x,\mathbf w)\bigr)
\right|_{v=0}.
\label{eq:two-centre-source-generation}
\end{align}
Consider the exponential generating integral
\begin{equation}
\label{eq:two-centre-generating-target}
\frac1{(2\pi i)^{2k}k!^2}
\sum_{r=0}^k\binom{k}{r}^{\!2}
\int_{\mathcal C_x(r)}\int_{\mathcal C_w(r)}
e^{-sT(\mathbf x)+uA(\mathbf x,\mathbf w)+vB(\mathbf x,\mathbf w)}
\mathcal V(\mathbf x,\mathbf w)
\,d\mathbf x\,d\mathbf w.
\end{equation}
We will relate this expression to $D_k(0;s,u,v)$ after applying
$\partial_u^a\partial_v^k$ and evaluating at $u=v=0$, for $0\leq a\leq k$. The reason the moments $X_n$ and $W_n$ arise is visible directly from
the integrand in \eqref{eq:two-centre-generating-target}. Its one-variable
weight factors are
\begin{equation*}
x^{-k}\exp\!\left(-sx+\frac ux+\frac v{x-1}\right)
\qquad\text{and}\qquad
(w-1)^{-k}\exp\!\left(\frac uw+\frac v{w-1}\right).
\end{equation*}
The Laplace variable $s$ appears only in the $x$-weight because the
cutoff depends only on $T(\mathbf x)=\sum_jx_j$. To connect these weights with the moments defined in
\eqref{eq:two-centre-X} and \eqref{eq:two-centre-W}, introduce 
\begin{align}
 d\mu_X^{(\theta)}(z)
 &:=\frac1{2\pi i}
 \left(dz\big|_{C_0}+e^{i\theta}dz\big|_{C_1}\right)
 z^{-k}\exp\!\left(-sz+\frac{u}{z}+\frac{v}{z-1}\right),
 \label{eq:two-centre-muX}
 \\
 d\mu_W^{(\theta)}(z)
 &:=\frac1{2\pi i}
 \left(dz\big|_{C_1}+e^{-i\theta}dz\big|_{C_0}\right)
 (z-1)^{-k}\exp\!\left(\frac{u}{z}+\frac{v}{z-1}\right).
 \label{eq:two-centre-muW}
\end{align}
Then
\begin{equation*}
X_n(\theta;s,u,v)=\int z^n\,d\mu_X^{(\theta)}(z),
\qquad
W_n(\theta;u,v)=\int z^n\,d\mu_W^{(\theta)}(z).
\end{equation*}
Suppressing the
common arguments, so that
\begin{equation*}
 X_n=X_n(\theta;s,u,v),
 \qquad
 W_n=W_n(\theta;u,v),
\end{equation*}
we have
\begingroup
\setlength{\arraycolsep}{3.5pt}
\begin{equation}\label{eq:two-centre-D-matrix}
 D_k(\theta;s,u,v)
 =
 \det\!\begin{pmatrix}
 X_0&X_1&\cdots&X_{k-1}&W_0&W_1&\cdots&W_{k-1}\\
 X_1&X_2&\cdots&X_k&W_1&W_2&\cdots&W_k\\
 \vdots&\vdots&&\vdots&\vdots&\vdots&&\vdots\\
 X_{2k-1}&X_{2k}&\cdots&X_{3k-2}
 &W_{2k-1}&W_{2k}&\cdots&W_{3k-2}
 \end{pmatrix}.
\end{equation}
\endgroup
Therefore the first $k$ columns are generated by the $X$-moments and the
last $k$ columns by the $W$-moments. The above also makes the block Hankel description transparent; reorder the scalar rows in pairs as $(0,k),(1,k+1),\ldots,(k-1,2k-1)$
and reorder the columns so that each $X$-column is placed next to the corresponding $W$-column: 
\begin{equation*}
(X_0,W_0),(X_1,W_1),\ldots,(X_{k-1},W_{k-1}).
\end{equation*}
The row and column reorderings leave the determinant unchanged. Hence
\begin{equation}\label{eq:two-centre-block-Hankel}
 D_k(\theta;s,u,v)
 =\det\left[\mathbf M_{p+q}(\theta;s,u,v)\right]_{p,q=0}^{k-1},
 \qquad
 \mathbf M_n=
 \begin{pmatrix}
 X_n&W_n\\
 X_{n+k}&W_{n+k}
 \end{pmatrix}.
\end{equation}
In particular, the block in position $(p,q)$ depends only on $p+q$, i.e. the same $2 \times 2$ block appears along each block anti-diagonal. We now use a grouped form of Andr\'eief's identity; see
Kuijlaars~\cite[Lemma~2.1]{kuijlaarsMOP}. Let $f_0,\ldots,f_{2k-1}$, $g_0,\ldots,g_{k-1}$, and
$h_0,\ldots,h_{k-1}$ be functions whose products $f_mg_q$ and
$f_mh_q$ are absolutely integrable with respect to $\mu_X$
and $\mu_W$, respectively. Writing
$\mathbf z=(x_1,\ldots,x_k,w_1,\ldots,w_k)$, we have
\begin{align}
&\det\left[
 \int f_m(x)g_q(x)\,d\mu_X(x)
 \ \middle|\
 \int f_m(w)h_q(w)\,d\mu_W(w)
 \right]_{\substack{0\le m\le2k-1\\0\le q\le k-1}}
 \notag\\[-1mm]
&\quad\quad \quad\quad\quad\quad=\frac1{k!^2}
 \int
 \det[f_m(z_j)]_{\substack{0\le m\le2k-1\\1\le j\le2k}}
 \det[g_p(x_q)]_{\substack{0\le p\le k-1\\1\le q\le k}}
 \det[h_p(w_q)]_{\substack{0\le p\le k-1\\1\le q\le k}}
 \notag\\[-1mm]
&\qquad\qquad\qquad\qquad\qquad\times
 \prod_{j=1}^k d\mu_X(x_j)
 \prod_{\ell=1}^k d\mu_W(w_\ell).
 \label{eq:two-centre-Andreief}
\end{align}
Splitting the combined
Vandermonde gives
\begin{align}
\Delta(x_1,\ldots,x_k,w_1,\ldots,w_k)
&=
\Delta(\mathbf x)\Delta(\mathbf w)
\prod_{j,\ell=1}^k(w_\ell-x_j),
\end{align}
and therefore
\begin{align}
\Delta(x_1,\ldots,x_k,w_1,\ldots,w_k)
\Delta(\mathbf x)\Delta(\mathbf w)
&=
(-1)^k
\Delta(\mathbf x)^2\Delta(\mathbf w)^2
\prod_{j,\ell=1}^k(x_j-w_\ell).
\label{eq:two-centre-cross-sign}
\end{align}
In the above, replacing each of the $k^2$ factors $w_\ell-x_j$ by
$-(x_j-w_\ell)$ contributes
$(-1)^{k^2}=(-1)^k$. Next, we aim to recover the integrals in
\eqref{eq:two-centre-delta-Laplace-shifted} from $D_k(0;s,u,v)$.
Apply \eqref{eq:two-centre-Andreief} at $\theta=0$, then
differentiate by $\partial_u^a\partial_v^k$ and evaluate at
$u=v=0$, where $0\leq a\leq k$. The contours are compact and
the integrands are entire in $u,v$, so differentiation may be
taken under the contour integrals. By
\eqref{eq:two-centre-source-generation}, this produces the factor
$A(\mathbf x,\mathbf w)^aB(\mathbf x,\mathbf w)^k$. In the expansion of the product measures, suppose that $p$ of
the $x$-variables are integrated over $C_1$ and $q$ of the
$w$-variables over $C_0$. We show that only the balanced
choices $p=q$ contribute.

After setting $y_\ell=w_\ell-1$, the integrand and contours
are those defining $J_{-s,p,q}^{(a)}(1)$ in
\cite[proof of Lemma~2.6]{kn:bcis26}, with
$\Phi_L(\mathbf x)=e^{LT(\mathbf x)}$.
This function is entire and symmetric in the $x$-variables,
so the shifted pole-count argument there applies and gives
$J_{-s,p,q}^{(a)}(1)=0$ whenever $p\ne q$, for every
$0\leq a\leq k$. Hence only the balanced choices $p=q=r$
survive.

There are $\binom{k}{r}^2$ balanced contour choices. Since the
integrand is symmetric separately in the $x$-variables and
in the $w$-variables, these choices may be relabelled into
$\mathcal C_x(r)$ and $\mathcal C_w(r)$ from
\eqref{eq:two-centre-contour-system}. Using
\eqref{eq:two-centre-cross-sign}, we obtain
\begin{align}
 \left.
 \partial_u^a\partial_v^kD_k(0;s,u,v)
 \right|_{u=v=0}
 ={}&
 \frac{(-1)^k}{(2\pi i)^{2k}k!^2}
 \sum_{r=0}^k\binom{k}{r}^{\!2}
 \int_{\mathcal C_x(r)}\int_{\mathcal C_w(r)}
 e^{-sT(\mathbf x)}
 \notag\\
 &\qquad\qquad\qquad\times
 A(\mathbf x,\mathbf w)^a
 B(\mathbf x,\mathbf w)^k
 \mathcal V(\mathbf x,\mathbf w)
 \,d\mathbf x\,d\mathbf w.
 \label{eq:two-centre-D-derivative-integral}
\end{align}

Write the contour integral in
the $r$-th term as $e^{-rs}Q_{r,a}(s)$, where $Q_{r,a}$ is a
polynomial because the contour poles have finite order. If we put $h=k-r$, the total vanishing order of the Vandermonde
and cross factors at the contour centres is
\begin{equation*}
 Z=2h(h-1)+2r(r-1)+2hr,
\end{equation*}
and the total pole order is at most $2kh+a+k$. Taking a
residue in $2k$ variables therefore gives
\begin{equation*}
 \deg Q_{r,a}
 \leq 2kh+a+k-Z-2k
 =k+a-2r^2.
\end{equation*}
For $r=0$, counting only the $x$-variables gives the stronger
bound
\begin{equation*}
 \deg Q_{0,a}\leq k^2+a-k(k-1)-k=a.
\end{equation*}
A negative degree bound means that the residue vanishes. After multiplication by $s^{-(k+a)}$, only negative powers
of $s$ occur, with exponents at most $-k$ for $r=0$ and
$-2r^2$ for $r\geq1$. Thus
\eqref{eq:two-centre-shifted-Laplace} applies, and the
contributions with $r\geq1$ vanish continuously at $c=r$. Combining \eqref{eq:two-centre-D-derivative-integral} with
\eqref{eq:two-centre-delta-Laplace-shifted} gives
\begin{equation}
 \label{eq:two-centre-before-packing}
 \delta_k(c)
 =
 (-1)^k\mathcal L^{-1}_{s\to c}
 \left[
 \sum_{a=0}^k\binom{k}{a}s^{-(k+a)}
 \left.
 \partial_u^a\partial_v^kD_k(0;s,u,v)
 \right|_{u=v=0}
 \right].
\end{equation}
Here the derivatives in $u,v$ are evaluated at $u=v=0$
before Laplace inversion. Finally, the binomial identity gives
\begin{equation}
 \label{eq:two-centre-binomial-operator}
 \sum_{a=0}^k
 \binom{k}{a}s^{-(k+a)}\partial_u^a
 =
 s^{-k}\left(1+\frac{\partial_u}{s}\right)^k.
\end{equation}
Substituting the above into \eqref{eq:two-centre-before-packing} proves \eqref{eq:two-centre-final-formula}.
\end{proof}
\begin{remark}
The expression inside the inverse Laplace transform in
\eqref{eq:two-centre-final-formula} can be written using
a single Kummer function.
Define
\begin{equation}
 \varphi_k(s):=
 \frac{1}{2\pi i}\left(\int_{C_0}+\int_{C_1}\right)
 \frac{e^{-sz}}{z^{2k}(z-1)^k}\,dz
 =
 \frac{(-s)^{3k-1}}{(3k-1)!}\,
 {}_1F_1(k;3k;-s).
\end{equation}
Here ${}_1F_1$ is Kummer's confluent hypergeometric function;
the equality follows by combining the contours and extracting
the coefficient of $z^{-1}$ in the expansion for $|z|>1$.
For $n\geq0$ and $0\leq a,b\leq k$, \eqref{eq:two-centre-X}
shows that
$\left.\partial_u^a\partial_v^bX_n(0;s,u,v)\right|_{u=v=0}$
is obtained by applying
$(-\partial_s)^{n+k-a}(-\partial_s-1)^{k-b}$ to $\varphi_k$ since $-\partial_s$ and $-\partial_s-1$ insert factors
of $z$ and $z-1$, respectively, into the contour integral.
The corresponding derivatives of $W_n$ are constants,
by \eqref{eq:two-centre-W}.

Kummer's differential equation gives
$s\varphi_k''+(s+2-3k)\varphi_k'-(2k-1)\varphi_k=0$.
For $s\neq0$, this reduces every higher derivative to a
rational linear combination of $\varphi_k$ and $\varphi_k'$.
Since each determinant term in \eqref{eq:two-centre-D-matrix}
contains exactly $k$ $X$-moments, the expression inverted in
\eqref{eq:two-centre-before-packing} is therefore a homogeneous
polynomial of degree $k$ in $\varphi_k$ and $\varphi_k'$,
with rational coefficients in $s$.
The function $\varphi_k$ provides a Kummer seed for the
classical Painlev\'e V solution $w_k$, given by
\begin{equation}
 w_k(s):=1-\frac{s\varphi_k'(s)}
 {(2k-1)\varphi_k(s)},
 \qquad
 s w_k'=(2k-1)w_k^2+(1-k-s)w_k-k.
\end{equation}
This is a standard Riccati reduction of Painlev\'e V with parameters
$(\alpha,\beta,\gamma,\delta)
=((2k-1)^2/2,-k^2/2,3k-2,-1/2)$;
see \cite[\S32.10(v)]{NIST:DLMF}.
\end{remark}

\subsubsection{Related literature}
Related determinant and Painlev\'e representations occur
in the study of characteristic-polynomial derivative moments.
Conrey, Rubinstein and Snaith~\cite{kn:crs06} obtain scalar
determinant formulas, while Basor et al.~\cite{kn:basor_et_al18}
express joint moments of CUE characteristic polynomials and
their derivatives through Painlev\'e V functions.
More recently, Assiotis et al.~\cite{kn:agkw26} treat joint
moments involving higher derivatives using Hankel determinants
shifted by partitions and derivatives of Painlev\'e V functions. At the level of the two-weight moment structure in
\eqref{eq:two-centre-D-matrix}, a direct precedent is the work
of Bailey et al.~\cite[Section~7]{kn:bbbcpRS19} on
near-unit-circle moments of logarithmic derivatives of CUE
characteristic polynomials.
Their determinant has the same structure as $D_k$, with
different weights.
They relate it to multiple orthogonal polynomials,
a $3\times3$ Riemann--Hilbert problem and, up to an explicit
factor, an isomonodromic tau function.
\vspace{-2mm}
\section{Acknowledgments}
This research was supported by the Heilbronn Institute for Mathematical Research. I would like to thank Nina Snaith, Brian Conrey and Sieg Baluyot for introducing me to this problem, Vivian Kuperberg for answering my questions, and Benoît Collins for helpful discussions relating to this topic.
\vspace{-4mm}
\bibliographystyle{plain} 
\bibliography{ref} 
\end{document}